\documentclass[a4paper,reqno]{amsart}
\usepackage{indentfirst}
\usepackage{amsfonts}
\usepackage{amscd}
\usepackage{amsthm}
\usepackage{stmaryrd}
\usepackage{amssymb}
\usepackage{amsmath}
\usepackage{thmtools}
\usepackage{mathtools}
\usepackage{epsfig}
\usepackage{graphicx}
\usepackage{slashed}
\usepackage[all,cmtip]{xy}
\usepackage{tikz}
\usepackage[disable]{todonotes}
\usetikzlibrary{decorations.markings}
\usepackage{enumitem}
\usepackage{setspace}
\usepackage{bbm}
\usepackage{hyperref}

\makeatletter
\providecommand\@dotsep{5}
\renewcommand{\listoftodos}[1][\@todonotes@todolistname]{%
  \@starttoc{tdo}{#1}}
\makeatother

\makeatletter
\providecommand*{\twoheadrightarrowfill@}{%
  \arrowfill@\relbar\relbar\twoheadrightarrow
}

\providecommand*{\xtwoheadrightarrow}[2][]{%
  \ext@arrow 0579\twoheadrightarrowfill@{#1}{#2}%
}
\makeatother

\usetikzlibrary{matrix,arrows,positioning}

\DeclareMathAlphabet{\mathpzc}{OT1}{pzc}{m}{it}

\def\itemNum$#1${\item $\displaystyle#1$
   \hfill\refstepcounter{equation}(\theequation)}

\def\pser#1{[\![#1]\!]} 
\newcommand{\Dim}{\text{\textnormal{\textbf{dim}}\,}}
\newcommand{\Sym}{\text{\textnormal{Sym}}}

\mathchardef\mhyphen="2D
\newcommand{\tr}{\text{\textnormal{tr}}}
\newcommand{\Hom}{\text{\textnormal{Hom}}}
\newcommand{\Ext}{\text{\textnormal{Ext}}}
\newcommand{\Aut}{\text{\textnormal{Aut}}}
\newcommand{\git}{/\!\!/}

\newcommand{\Spec}{\text{\textnormal{Spec}}}
\newcommand{\image}{\text{\textnormal{image}}}

\newtheorem{Lem}{Lemma}[section]
\newtheorem{Prop}[Lem]{Proposition}

\newtheorem*{Def}{Definition}

\theoremstyle{plain}
\newtheorem{Thm}[Lem]{Theorem}
\newtheorem{Cor}[Lem]{Corollary}
\newtheorem{Conj}[Lem]{Conjecture}

\newtheorem{thm}{Theorem}

\newtheorem{conj}{Conjecture}

\theoremstyle{definition}
\declaretheorem[numbered=no,name=Example,qed={\lower-0.3ex\hbox{$\triangleleft$}}]{Ex}
\newtheorem*{Rem}{Remark}
\newtheorem*{Rems}{Remarks}

\newcommand\numberthis{\addtocounter{equation}{1}\tag{\theequation}}

\begin{document}

\title[DT theory with classical structure groups]{Representations of cohomological Hall algebras and Donaldson-Thomas theory with classical structure groups}

\author[Matthew B. Young]{Matthew B. Young}
\address{Department of Mathematics\\
The University of Hong Kong\\
Pokfulam, Hong Kong}
\email{myoung@maths.hku.hk}

\date{\today}


\begin{abstract}
We introduce a new class of representations of the cohomological Hall algebras of Kontsevich and Soibelman, which we call cohomological Hall modules, or CoHM for short. These representations are constructed from self-dual representations of a quiver with contravariant involution $\sigma$ and provide a mathematical model for the space of BPS states in orientifold string theory. We use the CoHM to define a generalization of the cohomological Donaldson-Thomas theory of quivers which allows the quiver representations to have orthogonal and symplectic structure groups. The associated invariants are called orientifold Donaldson-Thomas invariants. We prove the integrality conjecture for orientifold Donaldson-Thomas invariants of $\sigma$-symmetric quivers. We also formulate precise conjectures regarding the geometric meaning of these invariants and the freeness of the CoHM of a $\sigma$-symmetric quiver. We prove the freeness conjecture for disjoint union quivers, loop quivers and the affine Dynkin quiver of type $\widetilde{A}_1$. We also verify the geometric conjecture in a number of examples. Finally, we describe the CoHM of finite type quivers by constructing explicit Poincar\'{e}-Birkhoff-Witt type bases of these representations.
\end{abstract}

\maketitle

\setcounter{footnote}{0}

\vspace{-10pt}

\tableofcontents

\section*{Introduction}

\subsection*{Motivation}
\addtocontents{toc}{\protect\setcounter{tocdepth}{1}}

Motivated by the Donaldson-Thomas theory of three dimensional Calabi-Yau categories, Kontsevich and Soibelman introduced in \cite{kontsevich2011} the cohomological Hall algebra (CoHA) of a quiver with potential. We briefly recall the connection between Donaldson-Thomas theory and the CoHA, leaving details to Section \ref{sec:coha}. For simplicity we assume that the potential is zero and that the quiver $Q$ is symmetric. Let $\Lambda_Q^+$ be the monoid of dimension vectors of $Q$. Denote by $\mathsf{Vect}_{\mathbb{Z}}$ the category of $\mathbb{Z}$-graded rational vector spaces and by $D^{lb}(\mathsf{Vect}_{\mathbb{Z}})_{\Lambda_Q^+}$ the full subcategory of $\Lambda_Q^+$-graded objects of the unbounded derived category of $\mathsf{Vect}_{\mathbb{Z}}$ with finite dimensional $\Lambda_Q^+ \times \mathbb{Z}$-homogeneous summands. The CoHA of $Q$ is defined to be the shifted direct sum of cohomology groups of stacks of representations of $Q$,
\[
\mathcal{H}_Q = \bigoplus_{d \in \Lambda_Q^+} H^{\bullet}(\mathbf{M}_d)\{ \chi(d,d) \slash 2\} \in D^{lb}(\mathsf{Vect}_{\mathbb{Z}})_{\Lambda_Q^+}.
\]
Here $\chi$ is the Euler form of $Q$ and the $\mathbb{Z}$-grading is the Hodge theoretic weight grading. The stack of flags of representations defines correspondences between the stacks $\mathbf{M}_d$, $d \in \Lambda_Q^+$, which can be used to make $\mathcal{H}_Q$ into an associative algebra object in $D^{lb}(\mathsf{Vect}_{\mathbb{Z}})_{\Lambda_Q^+}$. There exists an object $V_Q^{\mathsf{prim}} \in D^{lb}(\mathsf{Vect}_{\mathbb{Z}})_{\Lambda_Q^+}$ such that
\begin{equation}
\label{eq:cohaDTFactorization}
[\Sym(V_Q^{\mathsf{prim}} \otimes \mathbb{Q}[u])] = [\mathcal{H}_Q] \in K_0(D^{lb}(\mathsf{Vect}_{\mathbb{Z}})_{\Lambda_Q^+})
\end{equation}
where $u$ is an indeterminant of degree $(0,2) \in \Lambda_Q^+ \times \mathbb{Z}$ and $\Sym(V)$ is the free supercommutative algebra on $V$, the $\mathbb{Z}_2$-grading induced by the $\mathbb{Z}$-grading. The degree $d \in \Lambda_Q^+$ motivic Donaldson-Thomas invariant of $Q$ is defined to be the class
\[
\Omega_{Q,d} = [V^{\mathsf{prim}}_{Q,d}] \in K_0(D^{lb}(\mathsf{Vect}_{\mathbb{Z}})).
\]
The integrality conjecture \cite{kontsevich2008} states that
\[
\Omega_{Q,d} \in \image \left( K_0(D^b(\mathsf{Vect}_{\mathbb{Z}})) \rightarrow K_0(D^{lb}(\mathsf{Vect}_{\mathbb{Z}})) \right)
\]
for each $d \in \Lambda_Q^+$. A proof of this conjecture for quivers with potential was given in \cite[Theorem 10]{kontsevich2011}. However, positivity of $\Omega_Q$ was not proved.

While the definition of $\Omega_Q$ involves only the Grothendieck class of $\mathcal{H}_Q$, it is natural to expect that understanding the algebra structure of $\mathcal{H}_Q$ will provide additional insights into Donaldson-Thomas theory. Not unrelated, the algebra $\mathcal{H}_Q$ is a model for the algebra of closed oriented BPS states of a quantum field theory or string theory with extended supersymmetry \cite{harvey1998}, \cite{kontsevich2011}. In this direction, Efimov \cite{efimov2012} constructed a subobject $V^{\mathsf{prim}}_Q \otimes \mathbb{Q}[u] \subset \mathcal{H}_Q$, with $V^{\mathsf{prim}}_Q$ having finite dimensional $\Lambda_Q^+$-homogeneous summands, such that the canonical map
\begin{equation}
\label{eq:pbwMult}
\Sym(V_Q^{\mathsf{prim}} \otimes \mathbb{Q}[u]) \rightarrow \mathcal{H}_Q
\end{equation}
is an algebra isomorphism. Passing to Grothendieck rings, this confirms the integrality and positivity conjectures. The subobject $V^{\mathsf{prim}}_Q$ is a cohomologically refined Donaldson-Thomas invariant \cite{szendroi2015}. For an arbitrary quiver with potential $W$ and generic stability $\theta$, it was proved in \cite{davison2016a} that the slope $\mu$ cohomological Donaldson-Thomas invariant $V^{\mathsf{prim}, \theta}_{Q,W, \mu}$ can also be constructed as a subobject of the semistable critical CoHA $\mathcal{H}^{\theta \mhyphen \mathsf{ss}}_{Q,W,\mu}$ and that the analogue of the map \eqref{eq:pbwMult} is an isomorphism in $D^{lb}(\mathsf{Vect}_{\mathbb{Z}})_{\Lambda_Q^+}$. Moreover, $V^{\mathsf{prim}, \theta}_{Q,W, \mu}$ satisfies the integrality conjecture. In this way $\mathcal{H}^{\theta \mhyphen \mathsf{ss}}_{Q,W,\mu}$ acquires a Poincar\'{e}-Birkhoff-Witt type basis. As an application, the structure of $\mathcal{H}^{\theta \mhyphen \mathsf{ss}}_{Q,W,\mu}$ was used in \cite{davison2013} to give a new proof of the Ka\u{c} conjecture.

While less studied, the representation theory of the CoHA is also relevant to Donaldson-Thomas theory. Physical arguments suggest that the space of open BPS states in a theory with defects is a representation of the BPS algebra \cite{gukov2011}. By the work of \cite{chuang2014} such representations are expected to be related to CoHA representations constructed from stable framed objects \cite{szendroi2012}, \cite{soibelman2014}. Framed CoHA representations of quiver categories have been studied in \cite{franzen2015}, \cite{xiao2014}, \cite{davison2016a}. A similar construction, with framed quiver moduli replaced by Nakajima quiver varieties, was given in \cite{yang2014}.

In this paper we introduce a new class of CoHA representations which are defined using orthogonal and symplectic analogues of quiver representations. While the framing construction models open BPS states, the constructions in this paper model (unoriented) BPS states in orientifold string theory. A less general approach to this problem, using finite field Hall algebras and their representations, can be found in \cite{mbyoung2015}. From a related point of view, our formalism provides an extension of Donaldson-Thomas theory from structure group $\mathsf{GL}_n(\mathbb{C})$ to the classical groups $\mathsf{O}_n(\mathbb{C})$ and $\mathsf{Sp}_{2n}(\mathbb{C})$ in the following sense. If $\mathsf{G}$ is a reductive group, then the derived moduli stack of $\mathsf{G}$-bundles on a Calabi-Yau threefold $X$ has a canonical $(-1)$-shifted symplectic structure \cite[Corollary 2.6]{pantev2013} and its truncation has a symmetric perfect obstruction theory \cite[\S 3.2]{pantev2013} which could be used to define the $\mathsf{G}$-Donaldson-Thomas invariants of $X$. The usual Donaldson-Thomas theory is related to the case $\mathsf{G} = \mathsf{GL}_n(\mathbb{C})$. For $\mathsf{G}$ an orthogonal or symplectic group, $\mathsf{G}$-bundles on $X$ are precisely the (frame bundles of) self-dual objects of the category of vector bundles on $X$. More generally, we expect that the correct setting for orientifold Donaldson-Thomas theory is a three dimensional triangulated Calabi-Yau category together with a contravariant duality functor which preserves the Calabi-Yau pairing. The CoHA representations introduced below, and the resulting orientifold Donaldson-Thomas invariants, provide a concrete realization of this theory in the quiver setting.

\subsection*{Main results}
Let $Q$ be a quiver with contravariant involution $\sigma$. Let $\Lambda_Q^{\sigma,+} \subset \Lambda_Q^+$ the submonoid of $\sigma$-invariant dimension vectors. Then $D^{lb}(\mathsf{Vect}_{\mathbb{Z}})_{\Lambda_Q^{\sigma,+}}$ is a left module category over $D^{lb}(\mathsf{Vect}_{\mathbb{Z}})_{\Lambda_Q^+}$. After fixing some combinatorial data, the involution induces a contravariant duality functor on the category $\mathsf{Rep}_{\mathbb{C}}(Q)$. Denote by $\mathbf{M}_e^{\sigma}$ the stack of representations of dimension vector $e \in \Lambda_Q^{\sigma,+}$ which are symmetrically isomorphic to their duals (henceforth, self-dual representations). Set
\[
\mathcal{M}_Q =\bigoplus_{e \in \Lambda_Q^{\sigma,+}} H^{\bullet}(\mathbf{M}^{\sigma}_e)  \{ \mathcal{E}(e) / 2 \} \in D^{lb}(\mathsf{Vect}_{\mathbb{Z}})_{\Lambda_Q^{\sigma,+}}.
\]
The function $\mathcal{E}: \Lambda_Q \rightarrow \mathbb{Z}$ is the analogue of the Euler form for self-dual representations. Write $\mathbf{M}_{d,e}^{\sigma}$ for the stack of flags of representations $U \subset M$ with $M$ self-dual, $U$ isotropic in $M$ and $\Dim U=d$, $\Dim M = d + \sigma(d) + e$. In Theorem \ref{thm:CoHM} we prove that the correspondences
\[
\begin{array}{ccccl}
\mathbf{M}_d \times \mathbf{M}^{\sigma}_{e} & \leftarrow & \mathbf{M}^{\sigma}_{d,e}  & \rightarrow & \mathbf{M}^{\sigma}_{d + \sigma(d) + e} \\ 
(U, M \git U) & \mapsfrom & U \subset M & \mapsto & M
\end{array}
\]
where $\git$ is a categorical version of symplectic reduction, give $\mathcal{M}_Q$ the structure of a left $\mathcal{H}_Q$-module object in $D^{lb}(\mathsf{Vect}_{\mathbb{Z}})_{\Lambda_Q^{\sigma,+}}$. We call $\mathcal{M}_Q$ the cohomological Hall module (CoHM). In Theorem \ref{thm:cohmLoc} we use localization in equivariant cohomology to prove that $\mathcal{M}_Q$ has a combinatorial description as a signed shuffle module, analogous to the Fe\u{\i}gin-Odesski\u{\i} shuffle algebra structure of $\mathcal{H}_Q$ \cite{kontsevich2011}. The passage from shuffle algebras to signed shuffle modules reflects the passage from Weyl groups of general linear groups to (disconnected) Weyl groups of classical groups.

Suppose that $Q$ is $\sigma$-symmetric, that is, symmetric and $\sigma^* \mathcal{E}= \mathcal{E}$. Let $W_Q^{\mathsf{prim}}$ be a minimal generating subobject of $\mathcal{M}_Q$. Define the dimension vector $e \in \Lambda_Q^{\sigma, +}$ motivic orientifold Donaldson-Thomas invariant of $Q$ by
\[
\Omega^{\sigma}_{Q,e} = [W^{\mathsf{prim}}_{Q,e}] \in K_0(D^{lb}(\mathsf{Vect}_{\mathbb{Z}})).
\]
Our first main result is the following.

\begin{thm}[Theorem \ref{thm:oriDTfinite}]
\label{thm:oriDTfiniteIntro}
The orientifold integrality conjecture holds for $\sigma$-symmetric quivers. More precisely, for each $e\in \Lambda_Q^{\sigma,+}$ we have
\[
\Omega_{Q,e}^{\sigma} \in \image \big( K_0(D^b(\mathsf{Vect}_{\mathbb{Z}})) \hookrightarrow 
K_0(D^{lb}(\mathsf{Vect}_{\mathbb{Z}})) \big).
\]
\end{thm}

The proof is a modification of Efimov's proof \cite{efimov2012} of the integrality conjecture for $\mathcal{H}_Q$ and relies on the explicit signed shuffle description of $\mathcal{M}_Q$.

To better understand $\Omega_Q^{\sigma}$ and its relationship with $\Omega_Q$ we study the analogue of the map \eqref{eq:pbwMult}. The situation is more complicated than that of the CoHA since $\mathcal{M}_Q$ is not a free $\mathcal{H}_Q$-module.

\begin{conj}[Conjecture   \ref{conj:freeCOHM}]
\label{conj:cohm}
Let $Q$ be $\sigma$-symmetric and assume that $\mathcal{H}_Q$ is supercommutative. For each $e \in \Lambda_Q^{\sigma,+}$ there is an explicitly defined $\Lambda_Q^{\sigma,+} \times \mathbb{Z}$-graded subalgebra $\mathcal{H}_Q(e) \subset \mathcal{H}_Q$ such that the CoHA action map
\[
\bigoplus_{e \in \Lambda_Q^{\sigma,+}} \mathcal{H}_Q(e) \boxtimes W^{\mathsf{prim}}_{Q,e} \rightarrow \mathcal{M}_Q
\]
is an isomorphism in $D^{lb}(\mathsf{Vect}_{\mathbb{Z}})_{\Lambda_Q^{\sigma,+}}$. Moreover, the restriction to the summand $\mathcal{H}_Q(e) \boxtimes W^{\mathsf{prim}}_{Q,e}$ is a $\mathcal{H}_Q(e)$-module isomorphism onto its image.
\end{conj}

Passing to Grothendieck groups, Conjecture \ref{conj:cohm} implies the following orientifold analogue of the factorization \eqref{eq:cohaDTFactorization}:
\[
\sum_{e \in \Lambda_Q^{\sigma,+}} [\mathcal{H}_Q(e)] \cdot \Omega^{\sigma}_{Q,e} = [\mathcal{M}_Q].
\]
In general, knowing $\Omega_Q$ is insufficient to determine $[\mathcal{H}_Q(e)]$ and $\Omega_Q^{\sigma}$ cannot be computed directly from $\Omega_Q$. Instead, a $\mathbb{Z}_2$-equivariant refinement of $\Omega_Q$ is needed.

Turning to the geometry of orientifold Donaldson-Thomas invariants, let $\mathfrak{M}_e^{\sigma, \mathsf{st}}$ be the moduli scheme of stable self-dual representations of dimension vector $e$ and let $PH^{\bullet}(\mathfrak{M}_e^{\sigma, \mathsf{st}})$ be the pure part of its cohomology. In Proposition \ref{prop:hodgeBound} we construct a canonical surjection of graded vector spaces $
W^{\mathsf{prim}}_{Q,e} \twoheadrightarrow PH^{\bullet}(\mathfrak{M}_e^{\sigma, \mathsf{st}}) \{ \mathcal{E}(e) \slash 2\}$.

\begin{conj}[Conjecture \ref{conj:hodgeEqual}]
\label{conj:hodgeEqualIntro}
Let $Q$ be $\sigma$-symmetric. Then the canonical surjection $
W^{\mathsf{prim}}_{Q,e} \twoheadrightarrow PH^{\bullet}(\mathfrak{M}_e^{\sigma, \mathsf{st}}) \{ \mathcal{E}(e) \slash 2\}$ is an isomorphism.
\end{conj}

In Sections \ref{sec:examples} and \ref{sec:finiteTypeQuivers} we confirm Conjectures \ref{conj:cohm} and \ref{conj:hodgeEqualIntro} in a number of examples. The main results of Section \ref{sec:examples}, which focuses on the CoHM of $\sigma$-symmetric quivers, are be summarized by the following two theorems.

\begin{thm}[Theorem \ref{thm:loopFreeCoHM}]
\label{thm:loopFreeCoHMIntro}
Conjecture \ref{conj:cohm} holds if $Q$ is the $m$-loop quiver.
\end{thm}

When $m=0,1$ we compute $\Omega^{\sigma}_Q$ in closed form and verify Conjecture \ref{conj:hodgeEqualIntro}. In contrast to the ordinary case, there are already infinitely many orientifold Donaldson-Thomas invariants in some of these examples. Loop quivers have the special property that $\Omega_Q$ determines the $\mathbb{Z}_2$-equivariant Donaldson-Thomas invariants. In particular, in this case Theorem \ref{thm:loopFreeCoHMIntro} can be used to compute $\Omega_Q^{\sigma}$ from $\Omega_Q$.

\begin{thm}[Theorems \ref{thm:disjointCoHM} and \ref{thm:symmA1CoHM}]
\label{thm:symmCoHMIntro}
Conjectures \ref{conj:cohm} and \ref{conj:hodgeEqualIntro} hold for disjoint union quivers and for the symmetric orientation of the affine Dynkin quiver of type $\widetilde{A}_1$.
\end{thm}

In Section \ref{sec:finiteTypeQuivers} we study the CoHM of finite type quivers with involution, which except for trivial cases are not $\sigma$-symmetric. The non-trivial task is to describe the CoHM of Dynkin type $A$ quivers.

\begin{thm}[Theorem \ref{thm:finiteTypeCoHM}]
\label{thm:finiteCoHMIntro}
The CoHM $\mathcal{M}_Q$ of a Dynkin quiver of type $A$ admits two Poincar\'{e}-Birkhoff-Witt type bases, each of which is determined by a simple/indecomposable Poincar\'{e}-Birkhoff-Witt type basis of $\mathcal{H}_Q$ and the set of simple/indecomposable self-dual representations of $Q$.
\end{thm}

Theorem \ref{thm:finiteCoHMIntro} generalizes and categorifies the orientifold quantum dilogarithm identities found in \cite{mbyoung2015} using finite field methods. To prove Theorem \ref{thm:finiteCoHMIntro} we modify Rim\'{a}nyi's approach to the study of the CoHA of a finite type quiver \cite{rimanyi2013}. Along the way we prove a number of results that are of independent interest. For example, Corollary \ref{cor:fundClassStructConst} states that Thom polynomials of orbit closures of self-dual quiver representations appear as structure constants of the CoHM.


\subsection*{Notation}
All cohomology groups have coefficient ring $\mathbb{Q}$ and all tensor products are over $\mathbb{Q}$.

\subsection*{Acknowledgements}
The author would like to thank Ben Davison and Sven Meinhardt for helpful discussions. Parts of this work were completed while the author was visiting the National Center for Theoretical Sciences at National Taiwan University and the Korea Institute for Advanced Study during the Winter School on Derived Categories and Wall-Crossing. The author would like to thank Wu-yen Chuang, Michel van Garrel and Bumsig Kim for the invitations. The author was partially supported by the Research Grants Council of the Hong Kong SAR, China (GRF HKU 703712).

\section{Background material}
\label{sec:background}

\subsection{Classical groups}
\label{sec:rootSys}
\addtocontents{toc}{\protect\setcounter{tocdepth}{2}}

We fix notation regarding the classical groups. Each such group $\mathsf{G}_n$ is the automorphism group of a pair $(V_n, \langle \cdot, \cdot \rangle)$ consisting of a finite dimensional complex vector space with a nondegenerate bilinear form.

\begin{enumerate}
\item Types $B_n$ and $D_n$. Let $V_n = \mathbb{C}^{2n+1}$ with basis $x_1, \dots, x_n, w, y_1, \dots, y_n$ in type $B_n$ and $V_n = \mathbb{C}^{2n}$ with basis $x_1, \dots, x_n, y_1, \dots, y_n$ in type $D_n$. Define a symmetric bilinear form on $V_n$ by $\langle x_i, y_j \rangle = \delta_{i,j}$ and, in type $B_n$, $\langle w, w \rangle =1$, all other pairings between basis vectors being zero. Then $\mathsf{G}_n$ is the orthogonal group $\mathsf{O}_{2n+1}(\mathbb{C})$ or $\mathsf{O}_{2n}(\mathbb{C})$. It is important in what follows that we use the full orthogonal group and not the special orthogonal group. 

\item Type $C_n$. Let $V_n = \mathbb{C}^{2n}$ with basis $x_1, \dots, x_n, y_1, \dots, y_n$. Define a skewsymmetric bilinear form on $V_n$ by $\langle x_i, y_j \rangle = \delta_{i,j}$, all other pairings between basis vectors being zero. Then $\mathsf{G}_n$ is the symplectic group $\mathsf{Sp}_{2n}(\mathbb{C})$.
\end{enumerate}

Define a (connected) maximal torus
\[
\mathsf{T}_n = \{ \mbox{diag}(t_1, \dots, t_n, (1), t_1^{-1}, \dots, t_n^{-1} ) \mid t_i \in \mathbb{C}^{\times} \} \subset \mathsf{G}_n,
\]
omitting the middle $1$ except in type $B_n$. For each $1 \leq i \leq n$ there is a character $e_i: \mathsf{T}_n \rightarrow \mathbb{C}^{\times}$, $t \mapsto t_i$. The positive roots in each type are
\[
\begin{array}{ll}
\mbox{Type } B_n: & \Delta = \{ e_i \pm e_j \mid 1 \leq i < j \leq n\} \sqcup \{ e_i \mid 1 \leq i \leq n\}  \\
\mbox{Type } C_n: & \Delta = \{ e_i \pm e_j \mid 1 \leq i < j \leq n\} \sqcup \{ 2e_i \mid 1 \leq i \leq n\}  \\
\mbox{Type } D_n: & \Delta = \{ e_i \pm e_j \mid 1 \leq i < j \leq n\}.
\end{array}
\]
The Weyl groups $\mathfrak{W}_{\mathsf{G}_n} = N_{\mathsf{G}_n}(\mathsf{T}_n) \slash \mathsf{T}_n$ are
\[
\mathfrak{W}_{\mathsf{O}_{2n+1}} \simeq (\mathbb{Z}_2^n \rtimes \mathfrak{S}_n) \times \mathbb{Z}_2 , \qquad \mathfrak{W}_{\mathsf{Sp}_{2n}} \simeq \mathbb{Z}_2^n \rtimes \mathfrak{S}_n, \qquad \mathfrak{W}_{\mathsf{O}_{2n}} \simeq \mathbb{Z}_2^n \rtimes \mathfrak{S}_n
\]
where $\mathfrak{S}_n$ is the symmetric group on $n$ letters. Note that $\mathfrak{W}_{\mathsf{O}_n}$ is an extension of $\mathfrak{W}_{\mathsf{SO}_n}$ by $\mathbb{Z}_2 \simeq \pi_0(\mathsf{O}_n(\mathbb{C}))$.

\subsection{Equivariant cohomology}
\label{sec:equivCohom}

Fix an integer $n > 0$. If $N> n$, then the variety $\mathsf{Mat}_{N \times n}^*$ of complex $N \times n$ matrices of rank $n$ is $2(N-n)$-connected and carries a free action of $\mathsf{GL}_n$. The quotients $\mathsf{Mat}_{N \times n}^* \rightarrow \mathsf{Mat}_{N \times n}^*\slash \mathsf{GL}_n$ form an injective system $\{E_N \rightarrow B_N\}_{N >n}$ of finite dimensional approximations by varieties to the universal bundle $E\mathsf{GL}_n \rightarrow B \mathsf{GL}_n$. More generally, if $\mathsf{G}$ is a linear algebraic group with a closed embedding $\mathsf{G} \hookrightarrow \mathsf{GL}_n$, then $\{E_N \rightarrow E_N \slash \mathsf{G} \}_{N >n}$ approximates $E\mathsf{G} \rightarrow B\mathsf{G}$. If $\mathsf{H} \subset \mathsf{G}$ is a closed subgroup, then the canonical morphism $B\mathsf{H}\rightarrow B\mathsf{G}$ is a fibration with fibre $\mathsf{G} \slash \mathsf{H}$.

Let $\mathsf{G}$ act on a variety $X$. The $\mathsf{G}$-equivariant cohomology of $X$ is
\[
H_{\mathsf{G}}^{\bullet}(X) = \lim_{\longleftarrow} H^{\bullet}(X \times_{\mathsf{G}} E_N; \mathbb{Q}).
\]
Here $H^{\bullet}(-; \mathbb{Q})$ denotes singular cohomology with rational coefficients. We write $H^{\bullet}_{\mathsf{G}}$ for $H^{\bullet}_{\mathsf{G}}(\Spec (\mathbb{C}))$. If $\mathsf{T}_{\mathsf{GL}_n} \subset \mathsf{GL}_n$ is a maximal torus, then there are ring isomorphisms
\[
H^{\bullet}_{\mathsf{GL}_n} \simeq H^{\bullet}(B\mathsf{T}_{\mathsf{GL}_n})^{\mathfrak{W}_{\mathsf{GL}_n}} \simeq \mathbb{Q}[x_1, \dots, x_n]^{\mathfrak{S}_n}.
\]
Similarly, if $\mathsf{G}_n$ is a classical group of type $B_n$, $C_n$ or $D_n$, then the inclusion $\mathsf{T}_n \hookrightarrow \mathsf{G}_n$ induces ring isomorphisms
\begin{equation}
\label{eq:classicalEquivCohom}
H_{\mathsf{G}_n}^{\bullet} \simeq H^{\bullet}(B \mathsf{T}_n)^{\mathfrak{W}_{\mathsf{G}_n}} \simeq \mathbb{Q}[z_1^2, \dots, z_n^2]^{\mathfrak{S}_n}.
\end{equation}
Here it is essential that $\mathsf{G}_n$ is the full orthogonal group in type $D_n$. The generators $x_i$, $z_i$ have cohomological degree two.

We record the following results for later use.

\begin{Lem}
\label{lem:equivCoRestr}
\leavevmode
\begin{enumerate}
\item Let $\phi: \mathsf{GL}_n \rightarrow \mathsf{GL}_n$ be the automorphism $\phi(g) = (g^{-1})^t$. The induced map $(B\phi)^*: H^{\bullet}_{\mathsf{GL}_n}  \rightarrow H^{\bullet}_{\mathsf{GL}_n}$ is given by $(B\phi)^* f(x_1, \dots, x_n) = f(-x_1, \dots, -x_n)$.

\item Let $h : \mathsf{GL}_n \hookrightarrow \mathsf{G}_n$ be the hyperbolic embedding. The induced map $(Bh)^*: H^{\bullet}_{\mathsf{G}_n} \rightarrow H^{\bullet}_{\mathsf{GL}_n}$ is given by $(Bh)^*z_i = x_i$.

\item Let $\iota :\mathsf{G}_n \hookrightarrow \mathsf{GL}_{2n + \epsilon}$ be the embedding arising from the description of $\mathsf{G}_n$ given in Section \ref{sec:rootSys}, where $\epsilon=1$ in type $B_n$ and $\epsilon=0$ otherwise. Under the identification $H^{\bullet}_{\mathsf{GL}_{2n + \epsilon}} \simeq \mathbb{Q}[x_1, \dots, x_n, (w), y_1, \dots, y_n]^{\mathfrak{S}_{2n+\epsilon}}$ the induced map $(B\iota)^*: H^{\bullet}_{\mathsf{GL}_{2n+\epsilon}}  \rightarrow H^{\bullet}_{\mathsf{G}_n}$ is given by
\[
(B\iota)^* x_i = z_i, \qquad (B\iota)^* w=0, \qquad (B\iota)^* y_i = -z_i.
\]
\end{enumerate}
\end{Lem}

\begin{proof}
The statements can be proved by picking compatible maximal tori for the domain and codomain of each group homomorphism.
\end{proof}

Finally, recall that $H^{\bullet}_{\mathsf{G}}(X)$ and its compactly supported variant $H_{c,\mathsf{G}}^{\bullet}(X)$ each have a canonical mixed Hodge structure \cite{deligne1974}. The pure part of, say, $H^{\bullet}_{\mathsf{G}}(X)$ is
\[
PH^{\bullet}_{\mathsf{G}}(X) = \bigoplus_{k \geq 0} W_k H_{\mathsf{G}}^k(X)
\]
where $0=W_{-1} \subset W_0 \subset \cdots \subset W_{2k} = H_{\mathsf{G}}^k(X)$ is the weight filtration.

\subsection{Quiver representations}
Let $Q$ be a quiver with finite sets of nodes $Q_0$ and arrows $Q_1$. Write $\alpha: i \rightarrow j$ for an arrow $\alpha$ with tail $i$ and head $j$. Let $\mathsf{Rep}_{\mathbb{C}}(Q)$ be the hereditary abelian category of finite dimensional complex representations of $Q$. Objects of $\mathsf{Rep}_{\mathbb{C}}(Q)$ are pairs $(U,u)$, often abbreviated to $U$, where $U = \bigoplus_{i \in Q_0} U_i$ is a finite dimensional $Q_0$-graded complex vector space and $u = \{U_i \xrightarrow[]{u_{\alpha}} U_j \}_{i \xrightarrow[]{\alpha} j \in Q_1}$ is a collection of linear maps. Let $\Lambda_Q^+ = \mathbb{Z}_{\geq 0} Q_0$ be the monoid of dimension vectors. We write $\Dim U \in \Lambda_Q^+$ for the dimension vector of $U$. Set also $\Lambda_Q = \mathbb{Z} Q_0$.

The Euler form of $\mathsf{Rep}_{\mathbb{C}}(Q)$ is
\[
\chi(U,V)= \dim_{\mathbb{C}} \Hom(U,V) - \dim_{\mathbb{C}} \Ext^1(U,V).
\]
It descends to the following bilinear form on $\Lambda_Q$:
\[
\chi ( d, d^{\prime} ) = \sum_{i \in Q_0} d_i d_i^{\prime} - \sum_{i \xrightarrow[]{\alpha} j \in Q_1 } d_i d_j^{\prime}.
\]

For each $d \in \Lambda_Q^+$ let $R_d= \bigoplus_{i \xrightarrow[]{\alpha} j} \Hom_{\mathbb{C}} (\mathbb{C}^{d_i}, \mathbb{C}^{d_j} )$. The algebraic group $\mathsf{GL}_d= \prod_{i \in Q_0} \mathsf{GL}_{d_i}(\mathbb{C})$ acts on $R_d$ by change of basis. Its orbits are in bijection with the isomorphism classes of representations of dimension vector $d$.

\subsection{Self-dual quiver representations}
\label{sec:quiverReps}

For a detailed discussion of self-dual quiver representations the reader is referred to \cite{derksen2002}.

An involution of a quiver $Q$ is a pair of involutions $\sigma: Q_0 \rightarrow Q_0$ and $\sigma: Q_1 \rightarrow Q_1$ such that
\begin{enumerate}[label=(\roman*)]
\item if $i \xrightarrow[]{\alpha} j \in Q_1$, then $\sigma(j) \xrightarrow[]{\sigma(\alpha)} \sigma(i)  \in Q_1$, and
\item if $i \xrightarrow{\alpha} \sigma(i)  \in Q_1$, then $\alpha = \sigma(\alpha)$. 
\end{enumerate}
Given an involution, let $\Lambda_Q^{\sigma}$ be the subgroup of fixed points of the induced involution $\sigma: \Lambda_Q \rightarrow \Lambda_Q$. Set also $\Lambda_Q^{\sigma,+} = \Lambda_Q^+ \cap \Lambda_Q^{\sigma}$. The group homomorphism
\[
H: \Lambda_Q \rightarrow \Lambda_Q^{\sigma}, \qquad d \mapsto d + \sigma(d)
\]
makes $\Lambda_Q^{\sigma}$ into a $\Lambda_Q$-module.

A duality structure on $(Q,\sigma)$ is a pair of functions $s: Q_0 \rightarrow \{ \pm 1 \}$ and $\tau: Q_1 \rightarrow \{ \pm 1 \}$ such that $s$ is $\sigma$-invariant and $\tau_{\alpha} \tau_{\sigma(\alpha)} = s_i s_j$ for every arrow $i \xrightarrow[]{\alpha} j$. A duality structure defines an exact contravariant functor $S: \mathsf{Rep}_{\mathbb{C}} (Q) \rightarrow \mathsf{Rep}_{\mathbb{C}} (Q)$ as follows. At the level of objects
\[
S(U)_i =U_{\sigma(i)}^{\vee}, \qquad S(u)_{\alpha} = \tau_{\alpha} u_{\sigma(\alpha)}^{\vee}.
\]
Here $(-)^{\vee} = \Hom_{\mathbb{C}}(-, \mathbb{C})$ is the linear duality functor. If $\phi: U \rightarrow U^{\prime}$ is a morphism with components $\phi_i: U_i \rightarrow U_i^{\prime}$, then $S(\phi) : S(U^{\prime}) \rightarrow S(U)$ has components $S(\phi)_i = \phi^{\vee}_{\sigma(i)}$. Let $\mbox{ev}_{U_i} : U_i \rightarrow U_i^{\vee \vee}$ be the evaluation isomorphism. Then $\Theta_U= \oplus_{i \in Q_0} s_i \cdot \mbox{ev}_{U_i}$ defines an isomorphism of functors $\Theta: \mathbf{1}_{\mathsf{Rep}(Q)} \xrightarrow[]{\sim} S \circ S$ which satisfies $S(\Theta_U) \Theta_{S(U)} = \mathbf{1}_{S(U)}$. The triple $(\mathsf{Rep}_{\mathbb{C}} (Q), S, \Theta)$ is therefore an abelian category with duality in the sense of \cite{balmer2005}.

A self-dual representation is a pair $(M, \psi_M)$ consisting of a representation $M$ and an isomorphism $\psi_M: M \xrightarrow[]{\sim} S(M)$ which satisfies $S(\psi_M) \Theta_M = \psi_M$. More geometrically, a self-dual representation is a representation $M$ with a nondegenerate bilinear form $\langle \cdot, \cdot \rangle$ such that
\begin{enumerate}[label=(\roman*)]
\item $M_i$ and $M_j$ are orthogonal unless $i =\sigma(j)$,

\item the restriction of $\langle \cdot, \cdot \rangle$ to $M_i + M_{\sigma(i)}$ satisfies $\langle x,x^{\prime} \rangle = s_i \langle x^{\prime}, x \rangle$, and

\item for each arrow $i \xrightarrow[]{\alpha} j$ the structure maps of $M$ satisfy
\begin{equation}
\label{eq:strSymm}
\langle m_{\alpha} x,x^{\prime} \rangle - \tau_{\alpha}  \langle x, m_{\sigma(\alpha)} x^{\prime} \rangle =0, \qquad x \in M_i, \; x^{\prime} \in M_{\sigma(j)}.
\end{equation}
\end{enumerate}

As a basic example, let $U \in \mathsf{Rep}_{\mathbb{C}}(Q)$. Then the hyperbolic representation $H(U)$ is the self-dual representation $(U \oplus S(U), \psi_{H(U)} = \left( \begin{smallmatrix} 0 & \mathbf{1}_{S(U)} \\ \Theta_U & 0 \end{smallmatrix} \right))$.

Fix once and for all a partition $Q_0 = Q_0^- \sqcup Q_0^{\sigma} \sqcup Q_0^+$ such that $Q_0^{\sigma}$ consists of the nodes fixed by $\sigma$ and $\sigma(Q_0^-) = Q_0^+$. Similarly, fix a partition $Q_1 = Q_1^- \sqcup Q_1^{\sigma} \sqcup Q_1^+$.

Let $e \in \Lambda_Q^{\sigma, +}$. We will always assume that $e_i$ is even if $i \in Q_0^{\sigma}$ and $s_i =-1$. The trivial representation $\mathbb{C}^e$ then admits a self-dual structure $\langle \cdot, \cdot \rangle$ which is unique up to $Q_0$-graded isometry. Denote by $R_e^{\sigma} \subset R_e$ the linear subspace of representations whose structure maps satisfy equation \eqref{eq:strSymm} with respect to $\langle \cdot, \cdot \rangle$. There is an isomorphism
\[
R_e^{\sigma} \simeq \bigoplus_{ i \xrightarrow[]{\alpha} j  \in Q_1^+} \Hom_{\mathbb{C}} (\mathbb{C}^{e_i}, \mathbb{C}^{e_j} ) \oplus \bigoplus_{ i \xrightarrow[]{\alpha} \sigma(i)  \in Q_1^{\sigma}} \mbox{Bil}^{s_i \tau_{\alpha}}(\mathbb{C}^{e_i})
\]
where $\mbox{Bil}^{\epsilon}(\mathbb{C}^{e_i})$ is the vector space of symmetric ($\epsilon=1$) or skew-symmetric ($\epsilon=-1$) bilinear forms on $\mathbb{C}^{e_i}$. The subgroup $\mathsf{G}_e^{\sigma} \subset \mathsf{GL}_e$ which preserves $\langle \cdot, \cdot \rangle$ is 
\[
\mathsf{G}_e^{\sigma} \simeq  \prod_{i \in Q_0^+} \mathsf{GL}_{e_i} (\mathbb{C}) \times \prod_{i \in Q_0^{\sigma}} \mathsf{G}_{e_i}^{s_i}
\]
where
\[
\mathsf{G}_{e_i}^{s_i} = \begin{cases} \mathsf{Sp}_{e_i}(\mathbb{C}) & \mbox{ if } s_i=-1, \\ \mathsf{O}_{e_i}(\mathbb{C}) & \mbox{ if } s_i=1. \end{cases}
\]
The group $\mathsf{G}_e^{\sigma}$ acts linearly on $R_e^{\sigma}$. Its orbits are in bijection with the set of isometry classes of self-dual representations of dimension vector $e$.

Let $U \in \mathsf{Rep}_{\mathbb{C}}(Q)$. The pair $(S, \Theta_U)$ defines a linear $\mathbb{Z}_2$-action on $\Ext^i(S(U),U)$. Write $\Ext^i(S(U),U)^{\pm S}$ for the subspace of (anti-)invariants and define
\[
\mathcal{E}(U) =\dim_{\mathbb{C}} \Hom(S(U),U)^{-S} - \dim_{\mathbb{C}} \Ext^1(S(U),U)^S.
\]
It was proved in \cite[Proposition 3.3]{mbyoung2016} that $\mathcal{E}(U)$ 
depends only on the dimension vector of $U$ and that the resulting function $\mathcal{E}: \Lambda_Q \rightarrow \mathbb{Z}$ is given by
\begin{align*}
\mathcal{E}(d) =\sum_{i \in Q_0^{\sigma}} \frac{d_i(d_i -s_i)}{2}  + &   \sum_{i \in Q_0^+}  d_{\sigma(i)} d_i -  \\
  &  \sum_{\sigma(i) \xrightarrow[]{\alpha} i \in Q_1^{\sigma}} \frac{d_i(d_i + \tau_{\alpha} s_i)}{2} -\sum_{i \xrightarrow[]{\alpha} j  \in Q_1^+} d_{\sigma(i)} d_j.  \numberthis \label{eq:sdEulerForm}
\end{align*}
We will also use the identity
\begin{equation}
\label{eq:sdEulerIdentity}
\mathcal{E}(d + d^{\prime}) = \mathcal{E}(d) + \mathcal{E}(d^{\prime}) + \chi( \sigma(d) , d^{\prime} ).
\end{equation}

Self-dual representations admit reductions along isotropic subrepresentations. More precisely, if $M$ is a self-dual representation with isotropic subrepresentation $U$, then the orthogonal complement $U^{\perp} \subset M$ is a subrepresentation which contains $U$ and the quotient $M \git U =U^{\perp} \slash U$ inherits a canonical self-dual structure.

Following \cite{kontsevich2008}, to a quiver $Q$ we associate the quantum torus $\hat{\mathbb{T}}_Q =\mathbb{Q}(q^{\frac{1}{2}}) \pser{ \Lambda_Q^+}$. This is the $\mathbb{Q}(q^{\frac{1}{2}})$-vector space with topological basis $\{ t^d  \mid d \in \Lambda_Q^+\}$ and multiplication
\[
t^d \cdot t^{d^{\prime}} = q^{\frac{1}{2}( \chi(d,d^{\prime}) - \chi(d^{\prime}, d))} t^{d+d^{\prime}}.
\]
As in \cite{mbyoung2015}, given a duality structure, we also consider the vector space $\hat{\mathbb{S}}_Q =\mathbb{Q}(q^{\frac{1}{2}}) \pser{ \Lambda_Q^{\sigma,+}}$ with topological basis $\{ \xi^e \mid e \in \Lambda_Q^{\sigma,+} \}$. The formula
\[
t^d \star \xi^e = q^{\frac{1}{2} ( \chi(d,e) - \chi(e,d) + \mathcal{E}(\sigma(d)) - \mathcal{E}(d)) } \xi^{H(d) + e}
\]
gives $\hat{\mathbb{S}}_Q$ the structure of a left $\hat{\mathbb{T}}_Q$-module.

Finally, we recall how the standard theory of stability of quiver representations \cite{king1994} can be adapted to the self-dual setting. For details see \cite[\S 3] {mbyoung2015}. A stability $\theta \in \Hom_{\mathbb{Z}}(\Lambda_Q, \mathbb{Z})$ is called $\sigma$-compatible if it satisfies $\sigma^* \theta = - \theta$. Fix a $\sigma$-compatible stability $\theta$. A self-dual representation $M$ is called $\sigma$-semistable if $\mu(U) \leq \mu(M)$ for all non-zero isotropic subrepresentations $U \subset M$. If the previous inequality is strict, then $M$ is called $\sigma$-stable. Here $\mu(U) = \frac{\theta(\mathbf{dim}\, U)}{\dim U}$ is the slope of $U$. Note that the slope of a self-dual representation is necessarily zero. The moduli scheme of $\sigma$-semistable self-dual representations of dimension vector $e$ is the $\theta$-linearized geometric invariant theory quotient $\mathfrak{M}_e^{\sigma, \theta \mhyphen \mathsf{ss}} = R_e^{\sigma} \git_{\theta} \mathsf{G}_e^{\sigma}$. It parameterizes $S$-equivalence classes of $\sigma$-semistable representations. There is an open subscheme $\mathfrak{M}_e^{\sigma, \theta \mhyphen \mathsf{st}} \subset \mathfrak{M}_e^{\sigma, \theta \mhyphen \mathsf{ss}}$ which parameterizes isometry classes of $\sigma$-stable representations and has at worst orbifold singularities. A $\sigma$-stable representation can be written uniquely as an orthogonal direct sum $M=\bigoplus_{i=1}^k M_i$ with $M_i$ pairwise non-isometric self-dual representations which are stable as ordinary representations \cite[Proposition 3.5]{mbyoung2015}. In this case the isometry group of $M$ is $\Aut_S(M) \simeq \mathbb{Z}_2^k$. If $k=1$, then $M$ is called regularly $\sigma$-stable. By convention $\mathfrak{M}_0^{\sigma, \theta \mhyphen \mathsf{st}}=\Spec(\mathbb{C})$. In contrast, as usual for ordinary quiver moduli spaces, we set $\mathfrak{M}_0^{\theta \mhyphen \mathsf{st}}=\varnothing$.

\begin{Rem}
The bounded derived category $D^b(\Gamma_Q\mhyphen \mathsf{mod})$ of the Ginzburg differential graded algebra associated to $Q$ is a three dimensional Calabi-Yau category for which $\mathsf{Rep}_{\mathbb{C}}(Q)$ is the heart of a bounded $t$-structure \cite{ginzburg2006}. A duality structure on $Q$ induces a triangulated duality structure on $D^b(\Gamma_Q\mhyphen \mathsf{mod})$ which, up to a sign, preserves the Calabi-Yau pairing. This gives an abstract version of the three dimensional Calabi-Yau orientifolds considered in the string theory literature \cite{diaconescu2007}, \cite{hori2008}.
\end{Rem}

\section{Cohomological Hall algebras} 
\label{sec:coha}

\subsection{Definition of the CoHA}
\label{sec:cohaDef}
We recall some material from \cite[\S 2]{kontsevich2011}.

Fix a quiver $Q$. Let $\mathsf{Vect}_{\mathbb{Z}}$ be the abelian category of finite dimensional $\mathbb{Z}$-graded rational vector spaces and let $D^{lb}(\mathsf{Vect}_{\mathbb{Z}}) \subset D(\mathsf{Vect}_{\mathbb{Z}})$ be the full subcategory of objects whose cohomological and $\mathbb{Z}$ degrees are bounded from below. Let also $D^{lb}(\mathsf{Vect}_{\mathbb{Z}})_{\Lambda^+_Q}$ be the category whose objects are $\Lambda^+_Q$-graded objects of $D^{lb}(\mathsf{Vect}_{\mathbb{Z}})$ with finite dimensional $\Lambda_Q^+ \times \mathbb{Z}$-homogeneous summands and whose morphisms preserve the $\Lambda_Q^+ \times \mathbb{Z}$-grading. Define a monoidal product $\boxtimes^{\mathsf{tw}}$ on $D^{lb}(\mathsf{Vect}_{\mathbb{Z}})_{\Lambda^+_Q}$ by
\[
\bigoplus_{d \in \Lambda^+_Q} \mathcal{U}_d  \boxtimes^{\mathsf{tw}}  \bigoplus_{d \in \Lambda^+_Q} \mathcal{V}_{d}   = \bigoplus_{d \in \Lambda^+_Q}  \Big(\bigoplus_{\substack{ (d^{\prime},d^{\prime \prime}) \in \Lambda_Q^+ \times \Lambda_Q^+ \\ d= d^{\prime} + d^{\prime \prime}}} \mathcal{U}_{d^{\prime}} \otimes \mathcal{V}_{d^{\prime \prime}} \{(\chi(d^{\prime} , d^{\prime \prime}) - \chi(d^{\prime \prime} , d^{\prime}) )\slash 2 \} \Big).
\]
Here $\{ \frac{1}{2} \}$ denotes tensor product with the one dimensional vector space of cohomological and $\mathbb{Z}$ degree $-1$.

Fix $d^{\prime}, d^{\prime \prime} \in \Lambda_Q^+$ and put $d = d^{\prime} + d^{\prime \prime}$. Let $\mathbb{C}^{d^{\prime}} \subset \mathbb{C}^d$ be the $Q_0$-graded subspace spanned by the first $d^{\prime}$ coordinate directions. Let $R_{d^{\prime},d^{\prime \prime}} \subset R_d$ be the subspace of representations which preserve $\mathbb{C}^{d^{\prime}}$ and let $\mathsf{GL}_{d^{\prime},d^{\prime \prime}} \subset \mathsf{GL}_d$ be the subgroup which preserves $\mathbb{C}^{d^{\prime}}$. The cohomological Hall algebra (henceforth CoHA) of $Q$ is
\[
\mathcal{H}_Q = \bigoplus_{d \in \Lambda_Q^+} H^{\bullet}_{\mathsf{GL}_d}(R_d) \{ \chi(d,d) / 2 \} \in D^{lb}(\mathsf{Vect}_{\mathbb{Z}})_{\Lambda_Q^+}.
\]
The $\mathbb{Z}$-grading is the Hodge theoretic weight grading, which by purity coincides with the cohomological grading. The product $\mathcal{H}_Q \boxtimes^{\mathsf{tw}} \mathcal{H}_Q \rightarrow \mathcal{H}_Q$ is defined so that its restriction to $\mathcal{H}_{Q,d^{\prime}}\boxtimes^{\mathsf{tw}} \mathcal{H}_{Q,d^{\prime \prime}}$ is the composition
\begin{multline*}
H_{\mathsf{GL}_{d^{\prime}}}^{\bullet}(R_{d^{\prime}})  \otimes H_{\mathsf{GL}_{d^{\prime \prime}}}^{\bullet}(R_{d^{\prime \prime}}) \xrightarrow[]{\sim}  H_{\mathsf{GL}_{d^{\prime}} \times \mathsf{GL}_{d^{\prime \prime }}}^{\bullet} (R_{d^{\prime}}  \times R_{d^{\prime \prime}}) \xrightarrow[]{\sim} \\
H_{\mathsf{GL}_{d^{\prime}, d^{\prime \prime}}}^{\bullet} (R_{d^{\prime}, d^{\prime \prime}})  \rightarrow H_{\mathsf{GL}_{d^{\prime}, d^{\prime \prime}}}^{\bullet}(R_d) \{ (2 \Delta_1) \slash 2 \} \rightarrow H_{\mathsf{GL}_d}^{\bullet}(R_d) \{ (2 \Delta_1 + 2 \Delta_2) \slash 2 \}.
\end{multline*}
For ease of notation the degree shifts in $\mathcal{H}_Q$ and $\boxtimes^{\mathsf{tw}}$ have been omitted. The maps in the composition are constructed from the morphisms
\begin{equation}
\label{eq:diagMaps}
R_{d^{\prime}} \times R_{d^{\prime \prime}}
\overset{\pi}{\twoheadleftarrow} R_{d^{\prime},d^{\prime \prime}} \overset{i}{\hookrightarrow} R_d,  \qquad \mathsf{GL}_{d^{\prime}} \times \mathsf{GL}_{d^{\prime \prime}} \overset{p}{\twoheadleftarrow}\mathsf{GL}_{d^{\prime},d^{\prime \prime}} \overset{j}{\hookrightarrow} \mathsf{GL}_d.
\end{equation}
The first map in the CoHA multiplication is the K\"{u}nneth map, the second is induced by the homotopy equivalences $\pi$ and $p$, the third is pushforward along the $\mathsf{GL}_{d^{\prime}, d^{\prime \prime}}$-equivariant closed inclusion $i$ and the final is pushforward along $\mathsf{GL}_d \slash \mathsf{GL}_{d^{\prime},d^{\prime \prime}}$, the fibre of $B\mathsf{GL}_{d^{\prime},d^{\prime \prime}} \rightarrow \mathsf{B}\mathsf{GL}_d$. The degree shift is $\Delta_1 + \Delta_2 = - \chi(d^{\prime}, d^{\prime \prime})$. It is shown in \cite[Theorem 1]{kontsevich2011} that this multiplication gives $\mathcal{H}_Q$ the structure of an associative algebra object of $D^{lb}(\mathsf{Vect}_{\mathbb{Z}})_{\Lambda^+_Q}$.

The CoHA product can be written explicitly using localization in equivariant cohomology. To do so, identify $\mathcal{H}_{Q,d}$ with the vector space of polynomials in $\{x_{i,1}, \dots, x_{i, d_i} \}_{i \in Q_0}$ which are invariant under the Weyl group $\mathfrak{S}_d = \prod_{i\in Q_0} \mathfrak{S}_{d_i}$ of $\mathsf{GL}_d$. The product of $f_1 \in \mathcal{H}_{Q,d^{\prime}}$ and $f_2 \in \mathcal{H}_{Q,d^{\prime \prime}}$ will be viewed as a polynomial in $\{x_{i, 1}, \dots, x_{i, d_i} \}_{i \in Q_0}$ by identifying $x_{i,k}^{\prime}$ and $x_{i,k}^{\prime \prime}$ with $x_{i,k}$ and $x_{i,d_i^{\prime} + k}$, respectively. Let $\mathfrak{sh}_{d^{\prime}, d^{\prime \prime}} \subset \mathfrak{S}_d$ be the set of $2$-shuffles of type $(d^{\prime}, d^{\prime \prime})$, that is, the set of elements $\{ \pi_i \}_{i \in Q_0} \in \mathfrak{S}_d$ which satisfy
\[
\pi_i (1) < \cdots < \pi_i (d^{\prime}_i), \qquad  \pi_i (d^{\prime}_i + 1) < \cdots < \pi_i (d_i), \qquad \qquad i \in Q_0.
\]
Then $\mathfrak{sh}_{d^{\prime}, d^{\prime \prime}}$ acts on polynomials in $\{x_{i, 1}, \dots, x_{i, d_i} \}_{i \in Q_0}$ via the action of $\mathfrak{S}_d$.

\begin{Thm}[{\cite[Theorem 2]{kontsevich2011}}]
\label{thm:cohaLoc}
The product of $f_1 \in \mathcal{H}_{Q,d^{\prime}}$ and $f_2 \in \mathcal{H}_{Q,d^{\prime \prime}}$ is
\[
f_1 \cdot f_2= \sum_{\pi \in \mathfrak{sh}_{d^{\prime}, d^{\prime \prime}}}  \pi  \left( f_1(x^{\prime}) f_2(x^{\prime \prime})   \frac{  \prod_{i \xrightarrow[]{\alpha} j \in Q_1}  \prod_{b=1}^{d^{\prime \prime}_j} \prod_{a=1}^{d^{\prime}_i} \left(x_{j, b}^{\prime \prime} - x_{i, a}^{\prime} \right) }{\prod_{i \in Q_0}  \prod_{b=1}^{d^{\prime \prime}_i} \prod_{a=1}^{d^{\prime}_i} \left(x_{i, b}^{\prime \prime} - x_{i,  a}^{\prime} \right) } \right).
\]
\end{Thm}

The motivic Donaldson-Thomas (DT) series of $Q$ is the class of $\mathcal{H}_Q$ in the Grothendieck ring of $D^{lb}(\mathsf{Vect}_{\mathbb{Z}})_{\Lambda^+_Q}$,
\[
A_Q(q^{\frac{1}{2}}, t) = \sum_{(d,k) \in \Lambda_Q^+ \times \mathbb{Z}}\dim_{\mathbb{Q}} \mathcal{H}_{Q, (d,k)} (-q^{\frac{1}{2}})^k t^d \in \mathbb{Z}(q^{\frac{1}{2}}) \pser{ \Lambda_Q^+ }.
\]
It can be written explicitly as
\[
A_Q(q^{\frac{1}{2}}, t) = \sum_{d \in \Lambda_Q^+}   \frac{(-q^{\frac{1}{2}})^{\chi ( d, d )}}{ \prod_{i \in Q_0}  \prod_{j = 1}^{d_i} (1- q^j)} t^{d}.
\]
It is natural to view $A_Q$ as an element of the quantum torus $\hat{\mathbb{T}}_Q$ since the product in the latter agrees with the product induced by $\boxtimes^{\mathsf{tw}}$. Passing from motivic DT series to motivic DT invariants is most easily explained for symmetric quivers. We do this in the next section.

\subsection{The CoHA of a symmetric quiver}
\label{sec:symmQuiv}

In this section we assume that $Q$ is symmetric, that is, its Euler form is symmetric. In this case $\boxtimes^{\mathsf{tw}}$ reduces to a symmetric monoidal product $\boxtimes$ and $\mathcal{H}_Q$ is a $\Lambda_Q^+ \times \mathbb{Z}$-graded algebra.

Define a $\mathbb{Z}_2$-grading on $\mathcal{H}_Q$ as the reduction modulo two of its $\mathbb{Z}$-grading. If the Euler form satisfies
\begin{equation}
\label{eq:untwistSupercomm}
\chi(d, d^{\prime}) \equiv \chi(d,d) \chi(d^{\prime}, d^{\prime}) \mod 2
\end{equation}
for all $d,d^{\prime} \in \Lambda_Q^+$, then $\mathcal{H}_Q$ is a supercommutative algebra. Explicitly, writing $a_{ij}$ for the number of arrows from $i$ to $j$, equation \eqref{eq:untwistSupercomm} holds if and only if
\[
a_{ij} \equiv (1+ a_{ii})(1+a_{jj}) \mod 2
\]
for all distinct $i,j \in Q_0$. If the Euler form does not satisfy equation \eqref{eq:untwistSupercomm}, then the CoHA multiplication can be twisted by a sign so as to make $\mathcal{H}_Q$ supercommutative \cite[\S 2.6]{kontsevich2011}. Since all (connected) symmetric quivers studied in this paper satisfy equation \eqref{eq:untwistSupercomm}, we do not recall this twist here.

Write $\Sym(V)$ for the free supercommutative algebra generated by a $\Lambda_Q^+ \times \mathbb{Z}$-graded vector space $V$. The following result was conjectured by Kontsevich and Soibelman \cite[Conjecture 1]{kontsevich2011}. We consider $\mathcal{H}_Q$ as a supercommutative algebra.

\begin{Thm}[{\cite[Theorem 1.1]{efimov2012}}]
\label{thm:freeCoHA}
Let $Q$ be a symmetric quiver and let $u$ be a formal variable of degree $(0,2)$. There exists a $\Lambda_Q^+ \times \mathbb{Z}$-graded rational vector space of the form $
V_Q= V^{\mathsf{prim}}_Q \otimes \mathbb{Q}[u]$ such that $\Sym(V_Q) \simeq \mathcal{H}_Q$ as algebras. Moreover, each $\Lambda_Q^+$-homogeneous summand $V^{\mathsf{prim}}_{Q,d} \subset V^{\mathsf{prim}}_Q$ is finite dimensional.
\end{Thm}

Without the supercommutative twist, the isomorphism $\Sym(V_Q) \simeq \mathcal{H}_Q$ is as objects of $D^{lb}(\mathsf{Vect}_{\mathbb{Z}})_{\Lambda_Q^+}$. The second part of Theorem \ref{thm:freeCoHA}, known as the integrality conjecture \cite{kontsevich2008}, asserts that $V_Q^{\mathsf{prim}}$ is an element of $D^b(\mathsf{Vect}_{\mathbb{Z}})_{\Lambda_Q^+} \subset D^{lb}(\mathsf{Vect}_{\mathbb{Z}})_{\Lambda_Q^+}$, the full subcategory of objects whose $\Lambda_Q^+$-homogeneous components lie in $D^b(\mathsf{Vect}_{\mathbb{Z}})$.

\begin{Def}
The motivic Donaldson-Thomas invariant of a symmetric quiver $Q$ is the class of $V_Q^{\mathsf{prim}}$ in the Grothendieck ring of $D^b(\mathsf{Vect}_{\mathbb{Z}})_{\Lambda^+_Q}$,
\[
\Omega_Q (q^{\frac{1}{2}}, t) = \sum_{(d, k) \in \Lambda_Q^+ \times \mathbb{Z}} \dim_{\mathbb{Q}} V^{\mathsf{prim}}_{Q,(d, k)} (-q^{\frac{1}{2}})^k t^d \in \mathbb{Z}[q^{\frac{1}{2}}, q^{-\frac{1}{2}}]  \pser{ \Lambda_Q^+}.
\]
\end{Def}

Since $Q$ is symmetric, the parity-twisted Hilbert-Poincar\'{e} series of $\mathcal{H}_Q$ coincides with $A_Q$. Using this observation, Theorem \ref{thm:freeCoHA} implies that $A_Q$ can be written as a product of $q$-Pochhammer symbols $(t;q)_{\infty} = \prod_{i \geq 0}(1-q^i t)$.

\begin{Cor}[{\cite[Corollary 4.1]{efimov2012}}]
\label{cor:factorization}
Let $Q$ be a symmetric quiver. Then
\[
A_Q (q^{\frac{1}{2}}, t) = \prod_{(d, k) \in \Lambda_Q^+  \times \mathbb{Z} } (q^{\frac{k}{2}} t^d; q)_{\infty}^{- \Omega_{Q,(d,k)}}
\]
where $\Omega_{Q,(d,k)}$ is the coefficient of $q^{\frac{k}{2}}t^d$ in $\Omega_Q$.
\end{Cor}

To finish this section we recall a geometric interpretation of $\Omega_Q$. Let $\mathbf{M}_d^{\mathsf{st}}$ be the stack of stable representations of dimension vector $d$ with respect to the trivial stability, $\theta =0$. The map to the coarse moduli scheme $\mathbf{M}_d^{\mathsf{st}} \rightarrow \mathfrak{M}_d^{\mathsf{st}}$ is a $\mathbb{C}^{\times}$-gerbe and induces an isomorphism of mixed Hodge structures $H^{\bullet}(\mathbf{M}_d^{\mathsf{st}}) \simeq H^{\bullet}(\mathfrak{M}_d^{\mathsf{st}}) \otimes \mathbb{Q}[u]$.

\begin{Thm}[{\cite[Theorem 2.2]{chen2014}}]
\label{thm:geomDTInterp}
Let $Q$ be the double of a quiver. For each $d \in \Lambda_Q^+$ the restriction $H^{\bullet}_{\mathsf{GL}_d}(R_d) \rightarrow H^{\bullet}(\mathbf{M}_d^{\mathsf{st}})$ induces a $\mathbb{Z}$-graded vector space isomorphism $V^{\mathsf{prim}}_{Q,d} \xrightarrow[]{\sim} PH^{\bullet-\chi(d,d)}(\mathfrak{M}_d^{\mathsf{st}})$.
\end{Thm}

For other geometric interpretations of $\Omega_Q$ see \cite{hausel2013}, \cite{meinhardt2014}, \cite{davison2016a}.

\section{Cohomological Hall modules}
\label{sec:CoHM}

We introduce the cohomological Hall module of a quiver, establish its basic properties and formulate the main conjectures regarding its structure.

\subsection{Definition of the CoHM}
\label{sec:cohmDef}

Fix a quiver with involution $(Q, \sigma)$ and duality structure $(s, \tau)$. Using equation \eqref{eq:sdEulerIdentity} we verify that $D^{lb}(\mathsf{Vect}_{\mathbb{Z}})_{\Lambda_Q^{\sigma,+}}$ becomes a left module category over $(D^{lb}(\mathsf{Vect}_{\mathbb{Z}})_{\Lambda^+_Q}, \boxtimes^{\mathsf{tw}})$ via
\[
\bigoplus_{d \in \Lambda^+_Q} \mathcal{U}_d \boxtimes^{S \mhyphen \mathsf{tw}} \bigoplus_{e \in \Lambda_Q^{\sigma,+}} \mathcal{X}_e = \bigoplus_{e \in \Lambda_Q^{\sigma,+}} \Big( \bigoplus_{\substack{ (d^{\prime},e^{\prime \prime}) \in \Lambda_Q^+ \times \Lambda_Q^{\sigma,+} \\ e =H(d^{\prime}) + e^{\prime \prime}}} \mathcal{U}_{d^{\prime}} \otimes \mathcal{X}_{e^{\prime \prime}}  \{ \gamma(d^{\prime}, e^{\prime \prime}) \slash 2 \} \Big)
\]
where
\[
\gamma(d, e) = \chi( d, e ) - \chi(e , d) + \mathcal{E}(\sigma(d)) - \mathcal{E}(d).
\]

Fix $d \in \Lambda_Q^+$ and $e \in \Lambda_Q^{\sigma, +}$.  The subspace $R^{\sigma}_{d, e} \subset R^{\sigma}_{H(d) + e}$ of structure maps on the orthogonal direct sum $H(\mathbb{C}^d) \oplus \mathbb{C}^e$ which preserve the canonical $Q_0$-graded isotropic subspace $\mathbb{C}^d$ can be identified with the subspace of
\[
R_d \oplus R^{\sigma}_e \oplus \bigoplus_{i \xrightarrow[]{\alpha} j} \Hom_{\mathbb{C}}(\mathbb{C}^{e_i}, \mathbb{C}^{d_j}) \oplus \bigoplus_{i \xrightarrow[]{\alpha} j} \Hom_{\mathbb{C}}((\mathbb{C}^{d_{\sigma(i)}})^{\vee}, \mathbb{C}^{d_j})
\]
whose component $\{u_{\alpha}\} \in \bigoplus_{i \xrightarrow[]{\alpha} j} \Hom_{\mathbb{C}}((\mathbb{C}^{d_{\sigma(i)}})^{\vee}, \mathbb{C}^{d_j})$ satisfies $\Theta_{\mathbb{C}^{d_j}} u_{\alpha} = - \tau_{\alpha} u_{\sigma(\alpha)}^{\vee}$. Let also $\mathsf{G}_{d, e}^{\sigma} \subset \mathsf{G}_{H(d) + e}^{\sigma}$ be the subgroup of isometries which preserve $\mathbb{C}^d$.

The cohomological Hall module (henceforth CoHM) is
\[
\mathcal{M}_Q = \bigoplus_{e \in \Lambda_Q^{\sigma,+}}H^{\bullet}_{\mathsf{G}^{\sigma}_e}(R_{e}^{\sigma}) \{ \mathcal{E}(e) / 2 \} \in D^{lb}(\mathsf{Vect}_{\mathbb{Z}})_{\Lambda_Q^{\sigma,+}}.
\]
Define $\star: \mathcal{H}_Q \boxtimes^{S \mhyphen \mathsf{tw}} \mathcal{M}_Q \rightarrow \mathcal{M}_Q$ so that its restriction to $\mathcal{H}_{Q,d} \boxtimes^{S \mhyphen \mathsf{tw}} \mathcal{M}_{Q,e}$ is
\begin{multline*}
H^{\bullet}_{\mathsf{GL}_{d}}(R_{d}) \otimes H^{\bullet}_{\mathsf{G}_{e}^{\sigma}} (R_{e}^{\sigma}) \xrightarrow[]{\sim} H^{\bullet}_{\mathsf{GL}_{d} \times \mathsf{G}_{e}^{\sigma}}(R_{d} \times R_{e}^{\sigma}) \rightarrow H^{\bullet}_{\mathsf{G}_{d,e}^{\sigma}} (R_{d, e}^{\sigma})  \rightarrow \\
 H^{\bullet}_{\mathsf{G}_{d,e}^{\sigma}}(R_{H(d)+e}^{\sigma}) \{2 \delta_1 / 2\}  \rightarrow H^{\bullet}_{\mathsf{G}_{H(d)+e}^{\sigma}}(R_{H(d)+e}^{\sigma})\{ (2 \delta_1 + 2 \delta_2) / 2 \}.
\end{multline*}
Again, the degree shifts in $\mathcal{H}_{Q,d}$, $\mathcal{M}_{Q,e}$ and $\boxtimes^{S \mhyphen \mathsf{tw}}$ have been omitted. The maps in the composition are defined analogously to those appearing in the CoHA multiplication, where the maps \eqref{eq:diagMaps} have been replaced by
\[
R_d \times R^{\sigma}_e
\overset{\pi}{\twoheadleftarrow} R^{\sigma}_{d,e} \overset{i}{\hookrightarrow} R^{\sigma}_{H(d) + e},  \qquad \mathsf{GL}_d \times \mathsf{G}^{\sigma}_e \overset{p}{\twoheadleftarrow} \mathsf{G}^{\sigma}_{d,e} \overset{j}{\hookrightarrow} \mathsf{G}^{\sigma}_{H(d) + e}.
\]
The degree shifts are
\[
\delta_1 = \dim_{\mathbb{C}} R_{H(d)+e}^{\sigma} - \dim_{\mathbb{C}}  R_{d,e}^{\sigma}, \;\;\;\; \delta_2 = -\dim_{\mathbb{C}}   \mathsf{G}_{H(d)+e}^{\sigma} + \dim_{\mathbb{C}} \mathsf{G}_{d,e}^{\sigma}.
\]
A direct calculation shows that $\delta_1 + \delta_2 = -\chi( d, e ) - \mathcal{E}(\sigma(d))$.

\begin{Thm}
\label{thm:CoHM}
The map $\star$ gives $\mathcal{M}_Q$ the structure of a left $\mathcal{H}_Q$-module object of $D^{lb}(\mathsf{Vect}_{\mathbb{Z}})_{\Lambda_Q^{\sigma,+}}$.
\end{Thm}

\begin{proof}
The commutative diagram used to prove associativity of the CoHA in \cite[\S 2.3]{kontsevich2011} has a natural modification in the self-dual setting, obtained by requiring that the structure maps and isometry groups preserve multi-step isotropic flags. This modified commutative diagram establishes the $\mathcal{H}_Q$-module structure of $\mathcal{M}_Q$.
\end{proof}

Define an abelian group $\mathsf{W}(Q)$, the numerical Witt group, by the exact sequence
\[
\Lambda_Q \xrightarrow[]{H} \Lambda_Q^{\sigma} \xrightarrow[]{\nu} \mathsf{W}(Q) \rightarrow 0.
\]
Explicitly, $\mathsf{W}(Q) \simeq \prod_{i \in Q_0^{\sigma}} \mathbb{Z}_2$ with $\nu$ sending a dimension vector to its parities at $Q_0^{\sigma}$. The following result is immediate.

\begin{Prop}
\label{prop:cohmDecomp}
For each $w \in \mathsf{W}(Q)$ the subspace
\[
\mathcal{M}_Q^w = \bigoplus_{\{ e \in \Lambda_Q^{\sigma,+} \mid \nu(e) =w \} } \mathcal{M}_{Q,e} \subset \mathcal{M}_Q
\]
is a $\mathcal{H}_Q$-submodule which is trivial unless $s_i=1$ for all $i \in Q_0^{\sigma}$ with $w_i =1$. Moreover, $\mathcal{M}_Q = \bigoplus_{w \in \mathsf{W}(Q)} \mathcal{M}_Q^w$ as $\mathcal{H}_Q$-modules.
\end{Prop}

The motivic orientifold DT series of $Q$ is the class of $\mathcal{M}_Q$ in the Grothendieck group of $D^{lb}(\mathsf{Vect}_{\mathbb{Z}})_{\Lambda_Q^{\sigma,+}}$,
\[
A_Q^{\sigma}(q^{\frac{1}{2}}, \xi) = \sum_{(e,l) \in  \Lambda_Q^{\sigma, +} \times \mathbb{Z}}  \dim_{\mathbb{Q}} \mathcal{M}_{Q,(e,l)} (-q^{\frac{1}{2}})^l \xi^e \in \mathbb{Z}(q^{\frac{1}{2}}) \pser{ \Lambda_Q^{\sigma,+} }.
\]
Using the equivariant contractibility of $R_e^{\sigma}$ and the isomorphisms \eqref{eq:classicalEquivCohom} we find\begin{equation}
\label{eq:motivicOriDTSeries}
A_Q^{\sigma} =  \sum_{e \in \Lambda_Q^{\sigma, +}} \frac{(-q^{\frac{1}{2}})^{\mathcal{E}(e)}}{\prod_{i \in Q_0^+} \prod_{j=1}^{e_i} (1-q^j) \prod_{i \in Q_0^{\sigma}} \prod_{j=1}^{\lfloor \frac{e_i}{2} \rfloor} (1-q^{2j})} \xi^e.
\end{equation}
In what follows we will view $A_Q^{\sigma}$ as an element of the $\hat{\mathbb{T}}_Q$-module $\hat{\mathbb{S}}_Q$.

Also inspired by orientifold DT theory, in \cite{mbyoung2015} a different series was associated to a quiver with duality structure. Given a finite field $\mathbb{F}_q$ of odd characteristic, the $\mathcal{E}$-weighted number of $\mathbb{F}_q$-rational points of the stack of self-dual representations is
\[
\mathfrak{A}_{Q,\mathbb{F}_q}^{\sigma}(\xi)= \sum_M  \frac{(-q^{\frac{1}{2}})^{\mathcal{E}(\Dim M )}}{\# \Aut_S(M) } \xi^{\Dim M}.
\]
The sum is over isometry classes of self-dual representations. Comparing equation \eqref{eq:motivicOriDTSeries} and a renormalized version of \cite[Proposition 4.2]{mbyoung2015} shows $A_Q^{\sigma}(q^{-\frac{1}{2}},\xi) = \mathfrak{A}_{Q,\mathbb{F}_q}(\xi)$. It follows that the cohomological approach to orientifold DT theory developed in this paper is consistent with the finite field approach of \cite{mbyoung2015}.

\subsection{The CoHM as a signed shuffle module}
\label{sec:shuffMod}

We give a combinatorial description of $\mathcal{M}_Q$. The analogous result for $\mathcal{H}_Q$ can be found in \cite[\S 2.4]{kontsevich2011}.

Using the isomorphism \eqref{eq:classicalEquivCohom}, for each $e \in \Lambda_Q^{\sigma, +}$ identify $\mathcal{M}_{Q,e}$ with the vector space of $\prod_{i \in Q_0^+} \mathfrak{S}_{e_i} \times \prod_{i\in Q_0^{\sigma}} \mathfrak{S}_{\lfloor \frac{e_i}{2} \rfloor}$-invariant polynomials in the variables
\[
\{z_{i,1}, \dots, z_{i,e_i}\}_{i \in Q_0^+}, \qquad \{z^2_{i, 1}, \dots, z^2_{i, \lfloor \frac{e_i}{2} \rfloor}\}_{i \in Q_0^{\sigma} }.
\]
We also identify polynomials in
\[
\{ x^{\prime}_{i,1}, \dots, x^{\prime}_{i,d_i} \}_{i \in Q_0}, \;\; \mbox{ and } \;\; \{ z^{\prime \prime}_{i,1}, \dots,  z^{\prime \prime}_{i,e_i} \}_{i \in Q^+_0}, \;\; \{ z^{\prime \prime}_{i,1}, \dots,  z^{\prime \prime}_{i,\lfloor \frac{e_i}{2} \rfloor} \}_{i \in Q^{\sigma}_0}
\]
with polynomials in
\begin{equation}
\label{eq:polyVars}
\{ z_{i,1}, \dots,  z_{i,d_i + e_i + d_{\sigma(i)}} \}_{i \in Q^+_0}, \qquad  \{ z_{i,1}, \dots,  z_{i,d_i + \lfloor \frac{e_i}{2} \rfloor} \}_{i \in Q^{\sigma}_0}
\end{equation}
via
\[
x^{\prime}_{i,j} \mapsto z_{i,j}, \qquad z^{\prime \prime}_{i,j} \mapsto z_{i, d_i + j}, \qquad x^{\prime}_{\sigma(i),j} \mapsto - z_{i, d_i + e_i + j}, \qquad i \in Q_0^+
\]
and
\[
x^{\prime}_{i,j} \mapsto z_{i,j}, \qquad z^{\prime \prime}_{i,j} \mapsto z_{i,d_i+j}, \qquad i \in Q_0^{\sigma}.
\]
The signs arise from the first part of Lemma \ref{lem:equivCoRestr}.

Given $m,n, p \in \mathbb{Z}_{\geq 0}$ let $\mathfrak{sh}_{m,n,p} \subset \mathfrak{S}_{m+n+p}$ be the set of $3$-shuffles of type $(m,n,p)$. The set of $\sigma$-shuffles of type $(d,e) \in \Lambda_Q^+ \times \Lambda_Q^{\sigma, +}$ is then defined to be
\[
\mathfrak{sh}^{\sigma}_{d,e} = \prod_{i \in Q_0^+} \mathfrak{sh}_{d_i, e_i, d_{\sigma(i)}} \times \prod_{i \in Q_0^{\sigma}} \left( \mathbb{Z}_2^{d_i} \times \mathfrak{sh}_{d_i, \lfloor \frac{e_i}{2} \rfloor} \right).
\]
There is a natural action of $\mathfrak{sh}^{\sigma}_{d,e}$ on the space of polynomials in the variables \eqref{eq:polyVars}, the shuffle factors acting as usual and the $\mathbb{Z}_2$ factors acting by multiplication by $-1$ on the first $d_i$ elements of $\{ z_{i,1}, \dots,  z_{i,d_i + \lfloor \frac{e_i}{2} \rfloor} \}_{i \in Q^{\sigma}_0}$.

For each $i \in Q_0$ define $\varepsilon_i : \Lambda_Q \rightarrow \{ 0,1\}$ by $e \mapsto e_i \mod 2$. Write  $\leq_t$ for $<$ if $t=-1$ and $\leq$ if $t=+1$.

\begin{Thm}
\label{thm:cohmLoc}
Let $f \in \mathcal{H}_{Q,d}$ and $g \in \mathcal{M}_{Q,e}$. Then
\[
f \star g   = \sum_{\pi \in \mathfrak{sh}^{\sigma}_{d,e}}  \pi \left( f(x^{\prime}) g (z^{\prime \prime})  \frac{ \prod_{\alpha \in Q_1^+ \sqcup Q_1^{\sigma}} V_{\alpha}(x^{\prime}, z^{\prime \prime}) }{ \prod_{i \in Q_0^+ \sqcup Q_0^{\sigma}} D_i (x^{\prime}, z^{\prime \prime}) } \right)
\]
where the numerators and denominators are defined as follows. If $i \in Q_0^+$, then
\[
D_i= \prod_{k=1}^{e_i} \prod_{l=1}^{d_i}  (z^{\prime \prime}_{i,k} - x^{\prime}_{i, l}) \prod_{m=1}^{d_{\sigma(i)}} \prod_{l=1}^{d_i}  (- x^{\prime}_{\sigma(i), m} - x^{\prime}_{i, l}) \prod_{m=1}^{d_{\sigma(i)}} \prod_{k=1}^{e_i}  (- x^{\prime}_{\sigma(i), m} - z^{\prime \prime}_{i, k}).
\]
If $i \in Q_0^{\sigma}$, then
\[
D_i = g_i(x^{\prime}_{i,1}, \dots, x^{\prime}_{i,d_i}) \prod_{1 \leq k < l \leq d_i} (-x^{\prime}_{i, k} - x^{\prime}_{i,l}) \prod_{l=1}^{d_i} \prod_{k=1}^{\lfloor \frac{e_i}{2} \rfloor }  (x_{i,l}^{\prime 2} -z_{i,k}^{\prime \prime 2})
\]
with
\[
g_i(x_{i,1}, \dots, x_{i,d_i})
= \begin{cases}
(-1)^{d_i}\prod_{l=1}^{d_i} x_{i,l} & \hbox{ if } \mathsf{G}_{2d_i + e_i}^{s_i} \hbox{ is of type } B, \\ (-2)^{d_i} \prod_{l=1}^{d_i} x_{i,l} & \hbox{ if } \mathsf{G}_{2d_i + e_i}^{s_i} \hbox{ is of type } C, \\ 1 & \hbox{ if } \mathsf{G}_{2d_i + e_i}^{s_i} \hbox{ is of type } D. \end{cases}
\]
If $i \xrightarrow[]{\alpha} j \in Q_1^+$, then $\displaystyle V_{\alpha} = \widetilde{V}_{\alpha}^{(i)} \widetilde{V}_{\alpha}^{(j)} \prod_{m=1}^{d_{\sigma(j)}} \prod_{l=1}^{d_i} (-x^{\prime}_{\sigma(j), m} - x^{\prime}_{i,l})$ where
\[
\widetilde{V}_{\alpha}^{(i)}=
\begin{cases}
\displaystyle \prod_{m=1}^{d_{\sigma(j)}}  \prod_{k=1}^{e_i}  (-x_{\sigma(j), m}^{\prime} - z^{\prime \prime}_{i,k}) & \mbox{ if } i \not \in Q_0^{\sigma},\\

\displaystyle \prod_{m=1}^{d_{\sigma(j)}} \prod_{k=1}^{\lfloor \frac{e_i}{2} \rfloor}  (x_{\sigma(j), m}^{\prime 2} - z^{\prime \prime 2}_{i,k}) \prod_{m=1}^{d_{\sigma(j)}}(-x^{\prime}_{\sigma(j),m})^{\varepsilon_i(e)} & \mbox{ if } i \in Q_0^{\sigma}
\end{cases}
\]
and
\[
\widetilde{V}_{\alpha}^{(j)}=
\begin{cases}
\displaystyle  \prod_{k=1}^{e_j} \prod_{l=1}^{d_i} (z^{\prime \prime}_{j,k} - x^{\prime}_{i,l}) & \mbox{ if } j \not \in Q_0^{\sigma}, \\

\displaystyle \prod_{l=1}^{d_i} \prod_{k=1}^{\lfloor \frac{e_j}{2} \rfloor} (x^{\prime 2}_{i,l} -z^{\prime \prime 2}_{j,k}) \prod_{l=1}^{d_i}(-x^{\prime}_{i,l})^{\varepsilon_j(e)} & \mbox{ if } j \in Q_0^{\sigma}.
\end{cases}
\]
If $\sigma(i) \xrightarrow[]{\alpha} i \in Q_1^{\sigma}$, then $V_{\alpha} = \widetilde{V}_{\alpha} \prod_{1 \leq j \leq_{s_i \tau_{\alpha}} k \leq d_{\sigma(i)}} (- x^{\prime}_{\sigma(i), j} - x^{\prime}_{\sigma(i), k})$ where
\[
\widetilde{V}_{\alpha} = \begin{cases}

\displaystyle  \prod_{k=1}^{e_i} \prod_{l=1}^{d_{\sigma(i)}} (z^{\prime \prime}_{i,k} - x^{\prime}_{\sigma(i),l}) & \mbox{ if } i \not\in Q_0^{\sigma},   \\

\displaystyle \prod_{l=1}^{d_{\sigma(i)}} \prod_{k=1}^{\lfloor \frac{e_i}{2} \rfloor} (x^{\prime 2}_{\sigma(i),l} - z^{\prime \prime 2 }_{i,k} ) 
\prod_{l=1}^{d_{\sigma(i)}}(-x^{\prime}_{\sigma(i),l})^{\varepsilon_i(e)} & \mbox{ if } i \in Q_0^{\sigma}.
\end{cases}
\]
\end{Thm}

\begin{proof}
Regard $f$ and $g$ as classes in $H^{\bullet}(B\mathsf{GL}_d \times B\mathsf{G}_e^{\sigma})$. Let $\mathsf{Eu}_{\mathsf{G}^{\sigma}_{d,e}}(N_{R^{\sigma}_{H(d)+e} \slash R_{d,e}^{\sigma}})$ be the $\mathsf{G}^{\sigma}_{d,e}$-equivariant Euler class of the fibre of the normal bundle to $R_{d, e}^{\sigma} \subset R_{H(d) + e}^{\sigma}$ at the origin. Then
\[
f \star g = \int_{[\mathsf{G}^{\sigma}_{H(d) + e} \slash \mathsf{G}^{\sigma}_{d, e}]} f \cdot g \cdot \mathsf{Eu}_{\mathsf{G}^{\sigma}_{d,e}}(N_{R^{\sigma}_{H(d)+e} \slash R_{d,e}^{\sigma}})
\]
where $[\mathsf{G}^{\sigma}_{H(d) + e} \slash \mathsf{G}^{\sigma}_{d, e}]$ is the $\mathsf{G}_{H(d)+e}^{\sigma}$-equivariant fundamental class of $\mathsf{G}^{\sigma}_{H(d) + e} \slash \mathsf{G}^{\sigma}_{d, e}$. As in \cite[\S 2.4]{kontsevich2011}, this integral can be computed by equivariant localization with respect to the maximal torus $\mathsf{T} = \mathsf{T}_{H(d) +e} \subset \mathsf{G}^{\sigma}_{H(d) + e}$.

Let $U \in R_d$ and $N \in R_{H(d)+e}^{\sigma}$. An inclusion $U \hookrightarrow N$ is isotropic if and only if for each arrow $\alpha: i \rightarrow j$ we have a commutative diagram
\begin{equation}
\label{eq:isoSubDiag}
\begin{tikzpicture}[every node/.style={on grid},  baseline=(current bounding box.center),node distance=1.5]
  \node (A) {$U_i$};
  \node (B) [right=of A]{$(U^{\perp})_i$};
  \node (C) [right=of B]{$N_i$};
  \node (D) [below=of A] {$U_j$};
  \node (E) [right=of D]{$(U^{\perp})_j$};
  \node (F) [right=of E]{$N_j$};
  \draw[right hook->] (A)-- (B);
  \draw[->] (A)-- node[left] {$u_{\alpha}$}(D);
  \draw[->] (B)-- (E);
  \draw[->] (C)-- node[right] {$n_{\alpha}$} (F);
  \draw[right hook->] (B)--(C);
  \draw[right hook->] (D)-- (E);
  \draw[right hook->] (E)-- (F);
\end{tikzpicture}
\end{equation}

We start by computing the equivariant Euler class of the tangent space at a $\mathsf{T}$-fixed point of $\mathsf{G}^{\sigma}_{H(d)+e} \slash \mathsf{G}^{\sigma}_{d, e}$. The inclusions of diagram \eqref{eq:isoSubDiag} lead to an isomorphism
\[
\mathsf{G}^{\sigma}_{H(d)+e} \slash \mathsf{G}^{\sigma}_{d, e} \simeq \prod_{i \in Q_0^+} \mathsf{Fl}(d_i, e_i, d_{\sigma(i)}) \times \prod_{i \in Q_0^{\sigma}} \mathsf{IGr}^{s_i}(d_i, 2d_i + e_i)
\]
where $\mathsf{Fl}(a_1,a_2,a_3)$ is the variety of flags $0=V_0 \subset V_1 \subset V_2 \subset \mathbb{C}^{a_1+a_2+a_3}$ with $\dim V_i \slash V_{i-1} = a_i$ and $\mathsf{IGr}^s(a,b)$ is the variety of $a$-dimensional isotropic subspaces of a $b$-dimensional orthogonal ($s=1$) or symplectic ($s=-1$) vector space. The $\mathsf{T}$-fixed points of $\mathsf{Fl}(d_i, e_i, d_{\sigma(i)})$ are two-step coordinate flags and can be labelled by disjoint pairs of increasing sequences in $\{1, \dots, d_i+ e_i + d_{\sigma(i)} \}$ of the form $\pi = \{ a_1 < \cdots < a_{d_i} ; \; b_1 < \cdots < b_{e_i} \}$. Such pairs are in bijection with $\mathfrak{sh}_{d_i, e_i, d_{\sigma(i)}}$. The character of the tangent space to the flag $U_i \subset (U^{\perp})_i \subset N_i$ corresponding to the trivial shuffle is the product of the weights:
\begin{eqnarray*}
\Hom_{\mathbb{C}}(U_i , (N \git U)_i)  & \rightsquigarrow & \prod_{k=1}^{e_i} \prod_{l=1}^{d_i} (z^{\prime \prime}_{i,k} - x^{\prime}_{i,l}) \\
\Hom_{\mathbb{C}}(U_i , U_{\sigma(i)}^{\vee}) & \rightsquigarrow & \prod_{m=1}^{d_{\sigma(i)}} \prod_{l=1}^{d_i} (-x^{\prime}_{\sigma(i), m} - x^{\prime}_{i,l}) \\
\Hom_{\mathbb{C}}((N \git U)_i, U_{\sigma(i)}^{\vee}) & \rightsquigarrow &  \prod_{m=1}^{d_{\sigma(i)}} \prod_{k=1}^{e_i} (-x^{\prime}_{\sigma(i), m} - z^{\prime \prime}_{i,k}).
\end{eqnarray*}
Similarly, the $\mathsf{T}$-fixed points of $\mathsf{IGr}^{s_i}(d_i, 2d_i + e_i)$ are isotropic coordinate planes and are in bijection with $\mathbb{Z}_2^{d_i} \times \mathfrak{sh}_{d_i,  d_i + \lfloor \frac{e_i}{2} \rfloor}$ via
\[
\mathbb{Z}_2^{d_i} \times \mathfrak{sh}_{d_i,d_i + \lfloor \frac{e_i}{2} \rfloor} \owns (p, \pi)  \mapsto \mbox{span}_{\mathbb{C}} \{v_{\pi(1), p_1}, \dots, v_{\pi(d_i), p_{d_i}} \}
\]
where, in the notation of Section \ref{sec:rootSys}, $v_{i,p}$ is $x_i$ or $y_i$ if $p=1$ or $p=-1$, respectively. The character of the tangent space at such a fixed point is the product of the positive roots of $\mathsf{G}_{2d_i + e_i}^{s_i}$ which are not in the corresponding parabolic Lie subalgebra. Hence the denominators $D_i$ are as stated.

Next, we compute the restriction of $\mathsf{Eu}_{\mathsf{G}^{\sigma}_{d,e}}(N_{R^{\sigma}_{H(d)+e} \slash R_{d,e}^{\sigma}})$ to a $\mathsf{T}$-fixed point. From the vertical arrows of diagram \eqref{eq:isoSubDiag}, the contribution $V_{\alpha}$ of $\alpha \in Q_1^+$ to $\mathsf{Eu}_{\mathsf{G}^{\sigma}_{d,e}}(N_{R^{\sigma}_{H(d)+e} \slash R_{d,e}^{\sigma}})$ is the product of the weights
\[
\Hom_{\mathbb{C}}( U_i , (N \git U)_j )\rightsquigarrow
\begin{cases}
\displaystyle \prod_{k=1}^{e_j} \prod_{l=1}^{d_i}  (z^{\prime \prime}_{j,k} - x^{\prime}_{i,l}) & \mbox{ if } j \not\in Q_0^{\sigma}, \\

\displaystyle \prod_{k=1}^{\lfloor \frac{e_j}{2} \rfloor} \prod_{l=1}^{d_i}  (x^{\prime 2}_{i,l} - z^{\prime \prime 2}_{j,k}) \prod_{l=1}^{d_i}(-x^{\prime}_{i,l})^{\varepsilon_j(e)} & \mbox{ if } j \in Q_0^{\sigma}
\end{cases}
\]
and
\[
\Hom_{\mathbb{C}}(U_i, U_{\sigma(j)}^{\vee}) \rightsquigarrow \prod_{m=1}^{d_{\sigma(j)}} \prod_{l=1}^{d_i} (-x^{\prime}_{\sigma(j), m} - x^{\prime}_{i,l})
\]
and
\[
\Hom_{\mathbb{C}}( (N \git U)_i ,  U_{\sigma(j)}^{\vee}) \rightsquigarrow
\begin{cases}
\displaystyle \prod_{m=1}^{d_{\sigma(j)}}  \prod_{k=1}^{e_i}  (-x_{\sigma(j), m}^{\prime} - z^{\prime \prime}_{i,k}) & \mbox{ if } i \not\in Q_0^{\sigma}, \\

\displaystyle \prod_{m=1}^{d_{\sigma(j)}} \prod_{k=1}^{\lfloor \frac{e_i}{2} \rfloor}  (x_{\sigma(j), m}^{\prime 2} - z^{\prime \prime 2}_{i,k}) \prod_{m=1}^{d_{\sigma(j)}}(-x^{\prime}_{\sigma(j),m})^{\varepsilon_i(e)} & \mbox{ if } i \in Q_0^{\sigma}.
\end{cases}
\]
The contribution of an arrow $\sigma(i) \xrightarrow[]{\alpha} i \in Q_1^{\sigma}$ is computed similarly. Putting together the above calculations completes the proof.
\end{proof}

\subsection{The CoHM of a \texorpdfstring{$\sigma$}{}-symmetric quiver}
\label{sec:symCOHM}

In general, we do not know if the supercommutative twist of the CoHA of a symmetric quiver $Q$ can be lifted to $\mathcal{M}_Q$. In this section we therefore consider $\mathcal{H}_Q$ with its standard multiplication. In the self-dual setting it is natural to impose the following stronger notion of symmetry.

\begin{Def}
A quiver with involution and duality structure is called $\sigma$-symmetric if it is symmetric and the equality $\mathcal{E}(d) = \mathcal{E}(\sigma(d))$ holds for all $d \in \Lambda_Q$.
\end{Def}

Concretely, using equation \eqref{eq:sdEulerForm}, a symmetric quiver is $\sigma$-symmetric if and only if
\[
\sum_{\sigma(i) \xrightarrow[]{\alpha} i \in Q_1^{\sigma}} \tau_{\alpha} = \sum_{i \xrightarrow[]{\alpha} \sigma(i) \in Q_1^{\sigma}} \tau_{\alpha}
\]
for all $i \in Q_0$. Unlike all other places in the paper, the sums run over arrows with fixed initial and final nodes. If $Q$ is $\sigma$-symmetric, then $\boxtimes^{S \mhyphen \mathsf{tw}}$ reduces to the untwisted $D^{lb}(\mathsf{Vect}_{\mathbb{Z}})_{\Lambda_Q^+}$-module structure of $D^{lb}(\mathsf{Vect}_{\mathbb{Z}})_{\Lambda_Q^{\sigma,+}}$ defined using only the $\Lambda_Q$-module $\Lambda_Q^{\sigma}$, which we denote by $\boxtimes$. In particular, the CoHM of a $\sigma$-symmetric quiver is a $\Lambda_Q^{\sigma, +} \times \mathbb{Z}$-graded $\mathcal{H}_Q$-module.

Let $\mathcal{H}_{Q,+}$ be the augmentation ideal of $\mathcal{H}_Q$. 

\begin{Def}
The cohomological orientifold Donaldson-Thomas invariant of a $\sigma$-symmetric quiver $Q$ is
\[
W_Q^{\mathsf{prim}} = \mathcal{M}_Q \slash (\mathcal{H}_{Q,+} \star \mathcal{M}_Q) \in D^{lb}(\mathsf{Vect}_{\mathbb{Z}})_{\Lambda_Q^{\sigma,+}}.
\]
\end{Def}

By picking a vector space splitting we will often view $W_Q^{\mathsf{prim}}$ as a subobject of $\mathcal{M}_Q$. The next result asserts that the orientifold integrality conjecture holds.

\begin{Thm}
\label{thm:oriDTfinite}
Let $Q$ be a $\sigma$-symmetric quiver. Then each $\Lambda_Q^{\sigma,+}$-homogeneous summand $W^{\mathsf{prim}}_{Q,e} \subset W^{\mathsf{prim}}_Q$ is finite dimensional.
\end{Thm}

\begin{proof}
We adapt the method of proof of the integrality conjecture \cite[\S 3]{efimov2012}. Define
\[
X_{Q,d} = \mathbb{Q}[x_{i,j} \mid i \in Q_0, \; 1 \leq j \leq d_i], \qquad d \in \Lambda_Q^+
\]
and
\[
Z_{Q,e}= \mathbb{Q}[z_{i,j} \mid i \in Q_0^+, \;  1 \leq j \leq e_i] \otimes \mathbb{Q}[z_{i,j} \mid i \in Q_0^{\sigma}, \;  1 \leq j \leq \lfloor \frac{e_i}{2} \rfloor], \qquad e \in \Lambda_Q^{\sigma,+}
\]
considered as $\mathbb{Z}$-graded algebras with generators in degree two. The Weyl groups $\mathfrak{S}_d$ and $\mathfrak{W}_e$ of $\mathsf{GL}_d$ and $\mathsf{G}_e^{\sigma}$ act on $X_{Q,d}$ and $Z_{Q,e}$, respectively. Up to degree shifts we have $
\mathcal{H}_{Q,d} = X_{Q,d}^{\mathfrak{S}_d}$ and $\mathcal{M}_{Q,e} = Z_{Q,e}^{\mathfrak{W}_e}$ as $\mathbb{Z}$-graded vector spaces.

Keeping the notation of Theorem \ref{thm:cohmLoc}, define
\[
\mathcal{K}_{d^{\prime},e^{\prime \prime}}^{\sigma}(x^{\prime}, z^{\prime \prime}) =\frac{ \prod_{\alpha \in Q_1^+ \sqcup Q_1^{\sigma}} V_{\alpha}(x^{\prime}, z^{\prime \prime}) }{ \prod_{i \in Q_0^+ \sqcup Q_0^{\sigma}} D_i (x^{\prime}, z^{\prime \prime}) }.
\]
Let $Z_{Q,e}^{\mathsf{loc}}$ be the localization of $Z_{Q,e}$ at all factors of the denominators $D_i$ of $\mathcal{K}_{d^{\prime},e^{\prime \prime}}^{\sigma}$, for all $(d^{\prime}, e^{\prime \prime}) \in \Lambda_Q^+ \times \Lambda_Q^{\sigma,+}$ with $H(d^{\prime}) + e^{\prime \prime}=e$ and $d^{\prime} \neq 0$. Let $L_{Q,e} \subset Z^{\mathsf{loc}}_{Q,e}$ be the smallest $\mathfrak{W}_e$-stable $Z_{Q,e}$-submodule which contains $\mathcal{K}^{\sigma}_{d^{\prime}, e^{\prime \prime}}$ for all $(d^{\prime}, e^{\prime \prime})$ as above. We claim that $L_{Q,e}^{\mathfrak{W}_e}$ is the image of the CoHA action map
\[
\bigoplus_{\substack{(d^{\prime}, e^{\prime \prime}) \in \Lambda_Q^+ \times \Lambda_Q^{\sigma,+} \\ H(d^{\prime}) + e^{\prime \prime} = e, \; d^{\prime} \neq 0}} \mathcal{H}_{Q,d^{\prime}} \boxtimes \mathcal{M}_{Q,e^{\prime \prime}} \xrightarrow[]{\star} \mathcal{M}_{Q,e}.
\]
That $L_{Q,e}^{\mathfrak{W}_e}$ contains the image of the action map follows from the observation that $L_{Q,e}^{\mathfrak{W}_e}$ is $\mathbb{Q}$-linearly spanned by $\mathfrak{W}_e$-symmetrizations of elements of the form $
f g \mathcal{K}^{\sigma}_{d^{\prime}, e^{\prime \prime}}$ with $f \in X_{Q,d^{\prime}}$ and $g \in Z_{Q,e^{\prime \prime}}$. For the reverse inclusion, suppose that we are given an element of the form $f g \mathcal{K}^{\sigma}_{d^{\prime}, e^{\prime \prime}}$. By symmetrizing with respect to $\mathfrak{S}_{d^{\prime}}$ and $\mathfrak{W}_{e^{\prime \prime}}$, both of which are subgroups of $\mathfrak{W}_e$, we may assume that $f \in \mathcal{H}_{Q, d^{\prime}}$ and $g \in \mathcal{M}_{Q,e^{\prime \prime}}$. Then, up to a non-zero constant, the $\mathfrak{W}_e$-symmetrization of $f g \mathcal{K}^{\sigma}_{d^{\prime}, e^{\prime \prime}}$ is $f \star g$.

The theorem is thus equivalent to finite codimensionality of $L_{Q,e}^{\mathfrak{W}_e} \subset \mathcal{M}_{Q,e}$. Adding a loop at each node, with duality structure $\tau=-1$ for nodes in $Q_0^{\sigma}$, does not increase $L_{Q,e}$. We therefore assume that each node has at least one such loop. In this case $L_{Q,e} \subset Z_{Q,e}$ and we need not localize. Then $\mathcal{M}_{Q,e} \slash L_{Q,e}^{\mathfrak{W}_e} \hookrightarrow Z_{Q,e} \slash L_{Q,e}$ and it suffices to show that $L_{Q,e} \subset Z_{Q,e}$ has finite codimension. Interpret $Z_{Q,e}$ as the algebra of functions on the affine space $\mathbb{Q}^D$ of dimension
\[
D =\sum_{i \in Q_0^+} e_i + \sum_{i \in Q_0^{\sigma}} \lfloor \frac{e_i}{2} \rfloor
\]
and suppose that all elements of $L_{Q,e}$ vanish at $z \in \overline{\mathbb{Q}}^D$. We claim that $z=0$. Indeed, if $z \neq 0$, then by using the $\mathfrak{W}_e$-action we will write $z=\{z_i\}_{i \in Q_0^+ \sqcup Q_0^{\sigma}}$ as
\[
z_i = (x^{\prime}_{i,1}, \dots, x^{\prime}_{i, d_i^{\prime}}, z^{\prime \prime}_{i,1}, \dots, z^{\prime \prime}_{i, e_i^{\prime \prime}}, -x^{\prime}_{\sigma(i),1}, \dots, -x^{\prime}_{\sigma(i), d^{\prime}_{\sigma(i)}}), \qquad i \in Q_0^+
\]
and
\[
z_i = (x^{\prime}_{i,1}, \dots, x^{\prime}_{i, d_i^{\prime}}, z^{\prime \prime}_{i,1}, \dots, z^{\prime \prime}_{i, \lfloor \frac{e^{\prime \prime}_i}{2} \rfloor}), \qquad i \in Q_0^{\sigma}
\]
for some $d^{\prime} \neq 0$ so that $\mathcal{K}_{d^{\prime}, e^{\prime \prime}}^{\sigma}(x^{\prime}, z^{\prime \prime}) \neq 0$, yielding a contradiction. Let $z^{\prime \prime}$ be the set of vanishing coordinates of $z$ and let $x$ be what remains. By assumption $x \neq 0$. Up to the $\mathfrak{W}_e$-action, we need to write $x = \{(x_i^{\prime},-x_{\sigma(i)}^{\prime})\}_{i\in Q_0^+} \sqcup \{x^{\prime}_i\}_{i \in Q_0^{\sigma}}$ so that $\mathcal{K}_{d^{\prime}, e^{\prime \prime}}^{\sigma}(x^{\prime}, 0) \neq 0$. By Theorem \ref{thm:cohmLoc} this is equivalent to the following conditions:\footnote{Because signs are included in $x^{\prime}$, additional sign substitutions are not needed in these equations.}
\begin{enumerate}
\item $\prod_{m=1}^{d_{\sigma(i)}} \prod_{l=1}^{d_i}  (- x^{\prime}_{\sigma(i), m} - x^{\prime}_{i, l}) \neq 0$ if $i \in Q_0^+$.

\item $\prod_{l \leq k < l \leq d_i} (x^{\prime}_{i,k} + x^{\prime}_{i,l}) \neq 0$ if $i \in Q_0^{\sigma}$.

\item $\prod_{1 \leq j \leq k \leq d_{\sigma(i)}} (-x^{\prime}_{\sigma(i),j} - x^{\prime}_{\sigma(i),k}) \neq 0$ if $\sigma(i) \xrightarrow[]{\sigma} i \in Q_1^{\sigma}$.

\item $\prod_{m=1}^{d_{\sigma(j)}} \prod_{l=1}^{d_i} (-x^{\prime}_{\sigma(j),m} - x^{\prime}_{i,l}) \neq 0$ if $i \xrightarrow[]{\alpha} j \in Q_1^+$.
\end{enumerate}
These conditions can be satisfied as follows. Use the symmetric group at each $i \in Q_0^+$ to split any $\pm$ tuples, that is tuples which up to a permutation are of the form $(a, \dots, a,-a, \dots, -a)$ for some $a \in \overline{\mathbb{Q}}$, so that $(a, \dots,a)$ lies in the $x_i^{\prime}$ variable and $(-a, \dots, -a)$ lies in $-x_{\sigma(i)}^{\prime}$ variable. Then (1) holds. Similarly, act by the sign change subgroup at each $i \in Q_0^{\sigma}$ to ensure that $x_i^{\prime}$ contains no $\pm$ tuples. Then (2) holds. Condition (3) now also holds. Condition (4) breaks into three cases:
\begin{enumerate}[label=(\roman*)]
\item Both $i$ and $j$ are in $Q_0^{\sigma}$. Use the sign change subgroups to ensure that there are no $\pm$ tuples among all $Q_0^{\sigma}$ variables.

\item Neither $i$ nor $j$ is in $Q_0^{\sigma}$. Use the symmetric groups to ensure that there are no $\pm$ tuples among all $Q_0^+$ variables, and similarly for $Q_0^-$ variables.

\item Exactly one of $i,j$ is in $Q_0^{\sigma}$. Use the sign change subgroups to ensure that there are no $\pm$ tuples among the $Q_0^{\sigma}$ or $Q_0^+$ variables.
\end{enumerate}
This completes the proof.
\end{proof}

\begin{Def}
The motivic orientifold Donaldson-Thomas invariant of a $\sigma$-symmetric quiver $Q$ is the class of $W_Q^{\mathsf{prim}}$ in the Grothendieck group of $D^b(\mathsf{Vect}_{\mathbb{Z}})_{\Lambda^{\sigma,+}_Q}$,
\[
\Omega_Q^{\sigma} (q^{\frac{1}{2}}, \xi) = \sum_{(e,l) \in \Lambda_Q^{\sigma,+} \times \mathbb{Z}}  \dim_{\mathbb{Q}} W_{Q,(e,l)}^{\mathsf{prim}} (-q^{\frac{1}{2}})^l \xi^e \in \mathbb{Z}[q^{\frac{1}{2}}, q^{-\frac{1}{2}}] \pser{ \Lambda_Q^{\sigma,+}}.
\]
\end{Def}

The invariant $\Omega_Q^{\sigma}$, like $\Omega_Q$ of Section \ref{sec:symmQuiv}, is defined for the trivial stability. By Theorem \ref{thm:oriDTfinite} the numerical orientifold DT invariants can be defined as the $q^{\frac{1}{2}} \mapsto 1$ specialization of $\Omega_Q^{\sigma} (q^{\frac{1}{2}}, \xi)$. Note that we do not remove from $W_Q^{\mathsf{prim}}$ an infinite factor of the form $\mathbb{Q}[u]$. In the ordinary setting this factor compensates for the difference between the cohomologies of the moduli stack and moduli scheme of stable representations. In the orientifold setting the analogous cohomology groups are isomorphic; see Lemma \ref{lem:orbiHodge} below.

Our next goal is to formulate for $\mathcal{M}_Q$ an analogue of the freeness of the CoHA of a symmetric quiver. To begin, note that a duality structure on an arbitrary quiver induces linear isomorphisms $R_d \rightarrow R_{\sigma(d)}$ which are equivariant with respect to the group isomorphisms $\mathsf{GL}_d \rightarrow \mathsf{GL}_{\sigma(d)}$, $\{g_i\}_{i \in Q_0} \mapsto \{ (g_{\sigma(i)}^{-1})^t \}_{i \in Q_0}$. Since the functor $S: \mathsf{Rep}_{\mathbb{C}}(Q) \rightarrow \mathsf{Rep}_{\mathbb{C}}(Q)$ is contravariant, this defines an anti-involution $S_{\mathcal{H}}: \mathcal{H}_Q \rightarrow \mathcal{H}_Q$. Explicitly, using the first part of Lemma \ref{lem:equivCoRestr} we find
\begin{equation}
\label{eq:explicInv}
S_{\mathcal{H}}(f)(\{x_{i,j}\}_{i \in Q_0, \; 1 \leq j \leq d_{\sigma(i)}}) =  f(\{\tilde{x}_{i,j}\}_{i \in Q_0, \; 1 \leq j \leq d_i} )_{\vert \tilde{x}_{i,j} = - x_{\sigma(i),j}}
\end{equation}
for $f \in \mathcal{H}_{Q,d}$.

\begin{Prop}
\label{prop:moduleRelations}
Let $Q$ be a $\sigma$-symmetric quiver. The equality
\[
S_{\mathcal{H}}(f) \star g = (-1)^{\chi(e,d) + \mathcal{E}(d)} f \star g
\]
holds for all $f \in \mathcal{H}_{Q,d}$ and $g \in \mathcal{M}_{Q,e}$.
\end{Prop}

\begin{proof}
Let $\varpi \in \mathfrak{sh}_{d,e}^{\sigma}$ be the signed shuffle defined by the maps of ordered sets
\[
[d_i] \sqcup [e_i ] \sqcup [d_{\sigma(i)} ] \mapsto [d_{\sigma(i)}] \sqcup [e_i ] \sqcup [d_i ], \qquad i \in Q_0^+
\]
and
\[
[d_i ] \sqcup \big[ \lfloor \frac{e_i}{2} \rfloor \big] \mapsto [- d_i ] \sqcup \big[ \lfloor \frac{e_i}{2} \rfloor \big], \qquad i \in Q_0^{\sigma}.
\]
Here $[n] = \{z_1, \dots, z_n \}$. Precomposition with $\varpi$ gives a bijection $\mathfrak{sh}_{\sigma(d),e}^{\sigma} \rightarrow \mathfrak{sh}_{d,e}^{\sigma}$. Equation \eqref{eq:explicInv} shows that, after identifying variables as in Section \ref{sec:shuffMod}, the polynomials $f$ and $S_{\mathcal{H}}(f)$ differ exactly by $\varpi$. Note that $\varpi$ fixes $g$.

Using the explicit form of $\mathcal{K}_{d,e}^{\sigma}$ from Theorem \ref{thm:cohmLoc}, we will show that
\begin{equation}
\label{eq:kernelSymm}
\varpi (\mathcal{K}_{\sigma(d),e}^{\sigma}) = (-1)^{\chi(e,d)+ \mathcal{E}(d)} \mathcal{K}_{d,e}^{\sigma}.
\end{equation}
Applying $\varpi$ to a factor $D_i$ of $\mathcal{K}_{\sigma(d),e}^{\sigma}$, $i \in Q_0^{\sigma}$, results in multiplication by $(-1)^{\frac{d_i(d_i+1)}{2}}$ in types $B$ and $C$ and $(-1)^{\frac{d_i(d_i-1)}{2}}$ in type $D$. If $i \in Q_0^+$, then the result is multiplication by $(-1)^{e_i d_i + d_i d_{\sigma(i)} + e_i d_{\sigma(i)}}$. The sign change of the denominator of $\mathcal{K}_{\sigma(d),e}^{\sigma}$ is thus $(-1)^{\chi_{Q_0}(\sigma(d),e) + \mathcal{E}_{Q_0}(\sigma(d))}$, the subscripts indicating that only summands of $\chi$ and $\mathcal{E}$ associated to $Q_0$ are included. Similarly, $\varpi$ acts on $V_{\alpha}$ by multiplication by $(-1)^{d_i d_{\sigma(j)} + e_i d_{\sigma(j)} + d_i e_j}$ for $i \xrightarrow[]{\alpha} j \in Q_1^+$ and by $(-1)^{e_i d_{\sigma(i)} + \frac{d_{\sigma(i)}(d_{\sigma(i)} + \tau_{\alpha} s_i)}{2}}$ for $\sigma(i) \xrightarrow[]{\alpha} i \in Q_1^{\sigma}$. 
The sign change of the numerator is thus $(-1)^{\chi_{Q_1}(\sigma(d),e) + \mathcal{E}_{Q_1}(\sigma(d))}$. Equation \eqref{eq:kernelSymm} now follows from $\sigma$-symmetry.

We now compute
\begin{eqnarray*}
S_{\mathcal{H}}(f) \star g &=& \sum_{\pi \in \mathfrak{sh}_{\sigma(d),e}^{\sigma}} \pi (S(f) g \mathcal{K}_{\sigma(d),e}^{\sigma}) \\
&=& \sum_{\pi \in \mathfrak{sh}_{\sigma(d),e}^{\sigma}} \pi (\varpi(f) g \mathcal{K}_{\sigma(d),e}^{\sigma}) \\
&=& (-1)^{\chi(e,d) + \mathcal{E}(d)} \sum_{\pi \in \mathfrak{sh}_{\sigma(d),e}^{\sigma}} \pi \circ \varpi (f g \mathcal{K}_{d,e}^{\sigma}) \\
&=& (-1)^{\chi(e,d) + \mathcal{E}(d)} \sum_{\pi^{\prime} \in \mathfrak{sh}_{d,e}^{\sigma}} \pi^{\prime} (f g \mathcal{K}_{d,e}^{\sigma}) \\
&=& (-1)^{\chi(e,d) + \mathcal{E}(d)} f\star g
\end{eqnarray*}
which is the desired result.
\end{proof}

Since $S_{\mathcal{H}}$ is an algebra anti-involution, the image of the multiplication map $\mathcal{H}_{Q,+} \boxtimes \mathcal{H}_{Q,+} \rightarrow \mathcal{H}_Q$, and hence $V_Q$, inherits the structure of a $\mathbb{Z}_2$-representation. Moreover $V_Q = V^{\mathsf{prim}}_Q \otimes \mathbb{Q}[u]$ as $\mathbb{Z}_2$-representations, with $S_{\mathcal{H}}$ sending $u^n$ to $(-u)^n$. Indeed, setting $\sigma_d = \sum_{i\in Q_0} \sum_{j=1}^{d_i} x_{i,j}$ in degree $(0,2)$, we have (see \cite[\S 3]{efimov2012})
\[
V_Q  \simeq \bigoplus_{d \in \Lambda_Q^+} \left( V^{\mathsf{prim}}_{Q,d} \otimes \mathbb{Q}[\sigma_d] \right).
\]
If $V_Q^{\mathsf{prim}}$ is interpreted geometrically as in Theorem \ref{thm:geomDTInterp} or \cite{meinhardt2014}, then its representation structure coincides with that induced by the $\mathbb{Z}_2$-action on $\bigsqcup_{d \in \Lambda_Q^+} \mathfrak{M}_d^{\mathsf{st}}$.

Motivated by Proposition \ref{prop:moduleRelations}, for each $e \in \Lambda_Q^{\sigma, +}$ define a twisted $\mathbb{Z}_2$-representation on $\mathcal{H}_Q$ by the formula
\[
f \mapsto (-1)^{\chi(e,d) + \mathcal{E}(d)} S_{\mathcal{H}} (f), \qquad f \in \mathcal{H}_{Q,d}.
\]
Consider $V_Q$ as a (twisted) $\Lambda_Q^{\sigma,+} \times \mathbb{Z}$-graded $\mathbb{Z}_2$-representation by redefining the grading of $V_Q^{\mathsf{prim}}$ by
\[
\widetilde{V}^{\mathsf{prim}}_{Q,e} = \bigoplus_{\substack{ d \in \Lambda_Q^+ \\ H(d) =e}} V^{\mathsf{prim}}_{Q,d}, \qquad  e \in \Lambda_Q^{\sigma,+}.
\]
Let $(\widetilde{V}_Q)_{(\mathbb{Z}_2, e)}$ be the space of coinvariants. Identifying invariants and coinvariants, we obtain a $\Lambda_Q^{\sigma,+} \times \mathbb{Z}$-graded subalgebra
\[
\Sym((\widetilde{V}_Q)_{(\mathbb{Z}_2, e)}) \subset \Sym(V_Q).
\]
When $\mathcal{H}_Q$ is supercommutative we denote by $\mathcal{H}_Q(e)$ the corresponding subalgebra of $\mathcal{H}_Q$. Proposition \ref{prop:moduleRelations} implies that the cyclic $\mathcal{H}_Q$-module $\mathcal{H}_Q \star g \subset \mathcal{M}_Q$ generated by $g \in \mathcal{M}_{Q,e}$ is naturally a $\mathcal{H}_Q(e)$-module.

\begin{Lem}
\label{lem:superModule}
Let $Q$ be $\sigma$-symmetric. Define a $\mathbb{Z}_2$-grading on $\mathcal{M}_Q$ by the reduction modulo two of its $\mathbb{Z}$-grading. Then $\mathcal{M}_Q$ is a super $\mathcal{H}_Q$-module.
\end{Lem}

\begin{proof}
Observe that for an arbitrary quiver with involution the equality
\begin{equation}
\label{eq:eulerSymmetry}
\chi(d,d^{\prime} ) = \chi(\sigma(d^{\prime}), \sigma(d)), \qquad d, d^{\prime} \in \Lambda_Q
\end{equation}
holds. In the $\sigma$-symmetric case, the parity of elements of $\mathcal{H}_{Q,(d,k)} \star \mathcal{M}_{Q,(e,l)}$ is $\mathcal{E}(H(d) + e)$. Working modulo two, we compute
\begin{eqnarray*}
\mathcal{E}(H(d) + e) &\equiv& \mathcal{E}(d) + \mathcal{E}(\sigma(d)) + \chi(d,d) + \mathcal{E}(e) + \chi(d ,e ) + \chi(\sigma(d) ,e ) \\
&\equiv& \mathcal{E}(d) + \mathcal{E}(\sigma(d)) + \chi(d,d) + \mathcal{E}(e) + \chi(d ,e ) + \chi(d,e )\\
&\equiv& \mathcal{E}(d) + \mathcal{E}(\sigma(d)) + \chi(d,d) + \mathcal{E}(e)\\
&\equiv& \chi(d,d) + \mathcal{E}(e).
\end{eqnarray*}
The first equality follows from equation \eqref{eq:sdEulerIdentity}, the second from equation \eqref{eq:eulerSymmetry}, the third from symmetry of $Q$ and the last from $\sigma$-symmetry.
Since $\chi(d,d) + \mathcal{E}(e)$ is the sum of the parities of $\mathcal{H}_{Q,(d,k)}$ and $\mathcal{M}_{Q,(e,l)}$, the lemma follows.
\end{proof}

We can now state the main conjecture regarding the structure of $\mathcal{M}_Q$.

\begin{Conj}
\label{conj:freeCOHM}
Let $Q$ be a $\sigma$-symmetric quiver. Then the CoHA action map
\[
\bigoplus_{e \in \Lambda_Q^{\sigma,+}} \Sym((\widetilde{V}_Q)_{(\mathbb{Z}_2, e)}) \boxtimes W^{\mathsf{prim}}_{Q,e} \xrightarrow[]{\star} \mathcal{M}_Q
\]
is an isomorphism in $D^{lb}(\mathsf{Vect}_{\mathbb{Z}})_{\Lambda_Q^{\sigma,+}}$. Moreover, if $\mathcal{H}_Q$ is supercommutative, then for each $e \in \Lambda_Q^{\sigma,+}$ the restriction to the summand $\mathcal{H}_Q(e) \boxtimes W^{\mathsf{prim}}_{Q,e}$ is a $\mathcal{H}_Q(e)$-module isomorphism onto its image.
\end{Conj}

When $\mathcal{H}_Q$ is not supercommutative the above action map is defined via the $D^{lb}(\mathsf{Vect}_{\mathbb{Z}})_{\Lambda_Q^+}$-isomorphism $\Sym(V_Q) \simeq\mathcal{H}_Q$; see the comments after Theorem \ref{thm:freeCoHA}. Some instances of Conjecture \ref{conj:freeCOHM} will be proved in Section \ref{sec:examples}.

\begin{Rem}
A duality structure induces an involution of the stack $\mathbf{M}^{\mathsf{st}}$ of stable representations and $H^{\bullet}(\mathbf{M}^{\mathsf{st}} \slash \mathbb{Z}_2) \simeq H^{\bullet}(\mathbf{M}^{\mathsf{st}})^{\mathbb{Z}_2}$ as mixed Hodge structures. The algebra $\Sym((\widetilde{V}_Q)_{(\mathbb{Z}_2, e)})$ is not $\Sym(PH^{\bullet}(\mathbf{M}^{\mathsf{st}} \slash \mathbb{Z}_2))$ but is instead $\Sym(PH^{\bullet}(\mathbf{M}^{\mathsf{st}})^{(\mathbb{Z}_2,e)})$ where the non-geometric $e$-twisted $\mathbb{Z}_2$-action is used.
\end{Rem}

Conjecture \ref{conj:freeCOHM} implies a factorization of the orientifold DT series in terms of orientifold DT invariants and equivariantly refined DT invariants, analogous to Corollary \ref{cor:factorization}. To explain this we make the following definition.

\begin{Def}
Let $e^{\prime} \in \Lambda_Q^{\sigma,+}$. The $\mathbb{Z}_2$-equivariant motivic Donaldson-Thomas invariant is the class of $\widetilde{V}^{\mathsf{prim}}_Q$ in the Grothendieck ring of $D^b(\mathsf{Rep}_{\mathbb{Z}}(\mathbb{Z}_2))_{\Lambda^{\sigma,+}_Q}$,
\begin{align*}
\widetilde{\Omega}_Q = \sum_{(e,k) \in \Lambda_Q^{\sigma,+} \times \mathbb{Z}} \Big( \dim_{\mathbb{Q}} \, (\widetilde{V}^{\mathsf{prim}}_{Q,(e,k)})^+ +\dim_{\mathbb{Q}} \, (& \widetilde{V}^{\mathsf{prim}}_{Q,(e,k)}  )^-   \eta \Big)  (-q^{\frac{1}{2}})^k \xi^e \\ &  \in \mathbb{Z}(q^{\frac{1}{2}}) \pser{ \Lambda_Q^{\sigma,+}} [\eta] \slash (\eta^2 -1).
\end{align*}
Here $(-)^{\pm}$ denotes the subspace of (anti-)invariants.
\end{Def}

Note that, contrary to the notation, $\widetilde{\Omega}_Q$ depends on a fixed dimension vector $e^{\prime} \in \Lambda_Q^{\sigma,+}$. The graded character of $\mathbb{Q}[u]$ is $\frac{1+q \eta}{1-q^2}$. Using this we compute
\[
[(\widetilde{V}_Q)_{(\mathbb{Z}_2,e^{\prime})}] = \frac{1}{1-q^2} \sum_{(e,k) \in \Lambda_Q^{\sigma,+} \times \mathbb{Z}} (\widetilde{\Omega}_{Q,(e,k)}^+ + \widetilde{\Omega}_{Q,(e,k)}^- q) q^{\frac{k}{2}} \xi^e. 
\]
It follows that the parity-twisted Hilbert-Poincar\'{e} series of $\Sym((\widetilde{V}_Q)_{(\mathbb{Z}_2, e)})$ is
\[
A_Q(e^{\prime}) = \prod_{\substack{(e,k) \in \Lambda_Q^{\sigma,+} \times \mathbb{Z} \\ \lambda \in \{ \pm \}}} (q^{\frac{k}{2} + \delta_{-1,\lambda}} \xi^e ; q^2)_{\infty}^{- \widetilde{\Omega}^{\lambda}_{Q,(e,k)}} \in \mathbb{Z}(q^{\frac{1}{2}}) \pser{ \Lambda_Q^{\sigma,+} }.
\]
Passing to Grothendieck groups, Conjecture \ref{conj:freeCOHM} implies the factorization
\begin{equation}
\label{eq:cohmNumericalFactorization}
A_Q^{\sigma} \overset{\mbox{\tiny (Conj. \ref{conj:freeCOHM})}}{=} \sum_{e \in \Lambda_Q^{\sigma,+}} A_Q(e) \cdot \Omega_{Q,e}^{\sigma} \xi^e,
\end{equation}
interpreted as an equality in $\hat{\mathbb{S}}_Q$ with its commutative multiplication. Equation \eqref{eq:cohmNumericalFactorization} uniquely determines $\Omega^{\sigma}_Q$ from $A^{\sigma}_Q$ and the $\mathbb{Z}_2$-equivariant motivic DT invariants. In general, knowing only $\Omega_Q$ is insufficient to compute $\Omega^{\sigma}_Q$.

\subsection{Orientifold DT invariants and Hodge theory}
\label{sec:hodgeThy}

We continue to assume that $Q$ is $\sigma$-symmetric. In this section we describe a connection between $W_Q^{\mathsf{prim}}$ and the Hodge theory of $\bigsqcup_{e \in \Lambda_Q^{\sigma,+}} \mathfrak{M}_e^{\sigma, \mathsf{st}}$. We use the trivial stability, $\theta=0$.

We begin with the following basic lemma.

\begin{Lem}
\label{lem:orbiHodge}
Let $e \in \Lambda_Q^{\sigma, +}$.
\begin{enumerate}
\item The canonical map
\begin{equation}
\label{eq:dmIsom}
H^{\bullet}(\mathfrak{M}_e^{\sigma, \mathsf{st}}) \rightarrow H^{\bullet}_{\mathsf{G}_e^{\sigma}}(R_e^{\sigma, \mathsf{st}})
\end{equation}
is an isomorphism of mixed Hodge structures.

\item For each $k \geq 0$ the subspace $W_{k-1} H^k(\mathfrak{M}_e^{\sigma, \mathsf{st}}) \subset H^k(\mathfrak{M}_e^{\sigma, \mathsf{st}})$ is trivial and, dually, $W^{k+1} H^k_c(\mathfrak{M}_e^{\sigma, \mathsf{st}}) = H^k_c(\mathfrak{M}_e^{\sigma, \mathsf{st}})$.
\end{enumerate}
\end{Lem}

\begin{proof}
Since $H^{\bullet}_{\mathsf{G}_e^{\sigma}}(R_e^{\sigma, \mathsf{st}}) \simeq H^{\bullet}(\mathbf{M}_e^{\sigma, \mathsf{st}})$ and $\mathbf{M}_e^{\sigma, \mathsf{st}} \rightarrow \mathfrak{M}_e^{\sigma, \mathsf{st}}$ is a coarse moduli scheme, \cite[Theorem 4.40]{edidin2013} implies that \eqref{eq:dmIsom} is a graded vector space isomorphism. In the notation of Section \ref{sec:equivCohom}, the morphisms $R_e^{\sigma, \mathsf{st}} \times_{\mathsf{G}_e^{\sigma}} E_N \rightarrow  \mathfrak{M}_e^{\sigma, \mathsf{st}}$ approximate $R_e^{\sigma, \mathsf{st}} \times_{\mathsf{G}_e^{\sigma}} E \mathsf{G}_e^{\sigma} \rightarrow \mathfrak{M}_e^{\sigma, \mathsf{st}}$ and respect mixed Hodge structures. Passing to the limit shows that \eqref{eq:dmIsom} is also an isomorphism of mixed Hodge structures.

The second statement follows from \cite[Th\'{e}or\`{e}m 8.2.4]{deligne1974} and Poincar\'{e} duality.
\end{proof}

The next result gives a partial analogue of Theorem \ref{thm:geomDTInterp}.

\begin{Prop}
\label{prop:hodgeBound}
Let $Q$ be a $\sigma$-symmetric quiver. For each $e \in \Lambda_Q^{\sigma, +}$ the restriction $H^{\bullet}_{\mathsf{G}_e^{\sigma}}(R_e^{\sigma}) \rightarrow H^{\bullet}_{\mathsf{G}_e^{\sigma}}(R_e^{\sigma, \mathsf{st}})$ factors through a surjective map $W^{\mathsf{prim}}_{Q,e} \rightarrow PH^{\bullet - \mathcal{E}(e)}(\mathfrak{M}_e^{\sigma, \mathsf{st}})$.
\end{Prop}

\begin{proof}
As the argument is similar to \cite{chen2014}, we will be brief. Poincar\'{e} duality for smooth Artin stacks gives a perfect pairing
\[
H^{\bullet}_{\mathsf{G}_e^{\sigma}}(R_e^{\sigma}) \otimes H_{c,\mathsf{G}_e^{\sigma}}^{-2\mathcal{E}(e) - \bullet}(R_e^{\sigma}) \rightarrow \mathbb{Q}(-\mathcal{E}(e)).
\]
Here we use that $\dim_{\mathbb{C}} \mathbf{M}_e^{\sigma} = - \mathcal{E}(e)$. By \cite[Th\'{e}or\`{e}me 9.1.1]{deligne1974} the mixed Hodge structure on $H^i_{\mathsf{G}_e^{\sigma}}(R_e^{\sigma})$ is pure of weight $i$. Hence $H_{c,\mathsf{G}_e^{\sigma}}^i(R_e^{\sigma})$ is also pure of weight $i$. Consider the long exact sequence associated to the pair $(R_e^{\sigma, \mathsf{st}}, R_e^{\sigma} \backslash R_e^{\sigma, \mathsf{st}})$:
\[
\cdots \rightarrow 
H^{i-1}_{c, \mathsf{G}_e^{\sigma}}(R_e^{\sigma} \backslash R_e^{\sigma, \mathsf{st}}) \rightarrow H^i_{c, \mathsf{G}_e^{\sigma}}(R_e^{\sigma, \mathsf{st}}) \rightarrow H^i_{c, \mathsf{G}_e^{\sigma}}(R_e^{\sigma}) \rightarrow H^i_{c, \mathsf{G}_e^{\sigma}}(R_e^{\sigma} \backslash R_e^{\sigma, \mathsf{st}}) \rightarrow \cdots.
\]
Since the weights of $H^{i-1}_{c,\mathsf{G}_e^{\sigma}}(R_e^{\sigma} \backslash R_e^{\sigma, \mathsf{st}})$ are bounded above by $i-1$, the restriction $PH^i_{c, \mathsf{G}_e^{\sigma}}(R_e^{\sigma, \mathsf{st}}) \rightarrow H^i_{c, \mathsf{G}_e^{\sigma}}(R_e^{\sigma})$ is an injection and, dually, $H^i_{\mathsf{G}_e^{\sigma}}(R_e^{\sigma}) \rightarrow  PH^i_{\mathsf{G}_e^{\sigma}}(R_e^{\sigma, \mathsf{st}})$ is a surjection. Here we have used the second part of Lemma \ref{lem:orbiHodge}.

A straightforward modification of \cite[Lemma 2.1]{chen2014} shows that the composition of
\[
\bigoplus_{\substack{(d^{\prime}, e^{\prime \prime}) \in \Lambda_Q^+ \times \Lambda_Q^{\sigma,+} \\ H(d^{\prime}) + e^{\prime \prime} = e, \; d^{\prime} \neq 0}} \mathcal{H}_{Q,d^{\prime}} \boxtimes \mathcal{M}_{Q,e^{\prime  \prime}} \xrightarrow[]{\star} \mathcal{M}_{Q,e} = H^{\bullet - \mathcal{E}(e)}_{\mathsf{G}_e^{\sigma}}(R_e^{\sigma})
\]
with the restriction $H^{\bullet}_{\mathsf{G}_e^{\sigma}}(R_e^{\sigma}) \rightarrow H^{\bullet}_{\mathsf{G}_e^{\sigma}}(R_e^{\sigma,\mathsf{st}}) \simeq H^{\bullet} (\mathfrak{M}_e^{\sigma, \mathsf{st}})$ is zero. Combined with the previous paragraph, this shows that $W^{\mathsf{prim}}_{Q,e} \rightarrow PH^{\bullet - \mathcal{E}(e)} (\mathfrak{M}_e^{\sigma, \mathsf{st}})$ is surjective.
\end{proof}

The proof of injectivity in Theorem \ref{thm:geomDTInterp} relies on an interpretation of $\Omega_Q$ in terms of the cohomology of (smooth) Nakajima quiver varieties \cite{hausel2013}. Since smooth analogues of Nakajima varieties do not exist in the self-dual setting, it is not clear if the proof from \cite{chen2014} can be adapted to the present setting. In any case, it is natural to make the following conjecture.

\begin{Conj}
\label{conj:hodgeEqual}
The surjection $W^{\mathsf{prim}}_{Q,e} \twoheadrightarrow PH^{\bullet - \mathcal{E}(e)}(\mathfrak{M}_e^{\sigma, \mathsf{st}})$ is an isomorphism.
\end{Conj}

We verify Conjecture \ref{conj:hodgeEqual} in some examples in Section \ref{sec:examples}. In view of \cite{meinhardt2014} it is also natural to conjecture that $W_{Q,e}^{\mathsf{prim}}$ is isomorphic to the intersection cohomology $IC^{\bullet - \mathcal{E}(e)}(\overline{\mathfrak{M}_e^{\sigma,\mathsf{st}}})$ of the closure of $\mathfrak{M}_e^{\sigma,\mathsf{st}} \subset \mathfrak{M}_e^{\sigma, \mathsf{ss}}$. This can be verified in all examples in which Conjecture \ref{conj:hodgeEqual} is verified below.

\subsection{The critical semistable CoHM}
\label{sec:criticalCoHM}

We define the CoHM of a quiver with potential and stability, generalizing Section \ref{sec:cohmDef}.

Fix a stability $\theta$ and potential $W \in \mathbb{C}Q \slash [\mathbb{C}Q, \mathbb{C}Q]$. Let $d^{\prime}, d^{\prime \prime} \in \Lambda_Q^+$ and set $d = d^{\prime} + d^{\prime \prime}$. Let $R_d^{\theta \mhyphen \mathsf{ss}} \subset R_d$ be the open subvariety of semistable representations and define $R^{\theta \mhyphen \mathsf{ss}}_{d^{\prime},d^{\prime \prime}} = R_{d^{\prime},d^{\prime \prime}} \cap R^{\theta \mhyphen \mathsf{ss}}_d$. The trace maps $\tr(W)_d : R^{\theta \mhyphen \mathsf{ss}}_d \rightarrow \mathbb{C}$ and $\tr(W)_{d^{\prime}, d^{\prime \prime}} : R^{\theta \mhyphen \mathsf{ss}}_{d^{\prime}, d^{\prime \prime}} \rightarrow \mathbb{C}$ are invariant under $\mathsf{GL}_d$ and $\mathsf{GL}_{d^{\prime}, d^{\prime \prime}}$, respectively. Recall that the full subcategory of $\mathsf{Rep}_{\mathbb{C}}(Q)$ consisting of the zero object and all semistable representations of a fixed slope is abelian. If $\mu(d^{\prime}) = \mu(d^{\prime \prime})$, then by restriction of \eqref{eq:diagMaps} we get $R^{\theta \mhyphen \mathsf{ss}}_{d^{\prime}} \times R^{\theta \mhyphen \mathsf{ss}}_{d^{\prime \prime}}
\overset{\pi}{\twoheadleftarrow} R^{\theta \mhyphen \mathsf{ss}}_{d^{\prime},d^{\prime \prime}} \overset{i}{\hookrightarrow} R^{\theta \mhyphen \mathsf{ss}}_d$ along which the trace maps pull back according to
\[
\pi^* \left( \tr(W)_{d^{\prime}} \boxplus \tr(W)_{d^{\prime \prime}} \right) = \tr(W)_{d^{\prime}, d^{\prime \prime}}= i^* \tr(W)_d.
\]

Let $\varphi_{\tr(W)_d} \mathbb{Q}_{R^{\theta \mhyphen \mathsf{ss}}_d} \in D_c^b(R^{\theta \mhyphen \mathsf{ss}}_d)$ be the sheaf of vanishing cycles of $\tr(W)_d$, henceforth denoted by $\varphi_{\tr(W)_d}$. See \cite{kashiwara1994} for background. The slope $\mu$ semistable critical CoHA \cite[\S 7]{kontsevich2011} has underlying vector space
\[
\mathcal{H}^{\theta \mhyphen \mathsf{ss}}_{Q,W, \mu} = \bigoplus_{d \in \Lambda^+_{Q,\mu}} H^{\bullet}_{c, \mathsf{GL}_d}(R^{\theta \mhyphen \mathsf{ss}}_d, \varphi_{\tr(W)_d})^{\vee} \{ \chi(d,d) \slash 2 \}
\]
where $\Lambda^+_{Q,\mu} =\{ d \in \Lambda_Q^+ \mid \mu(d) = \mu \} \cup \{0 \}$. An associative product is defined on $\mathcal{H}^{\theta \mhyphen \mathsf{ss}}_{Q,W, \mu}$ analogously to Section \ref{sec:cohaDef}; see \cite[\S 7]{kontsevich2011}, \cite[\S 3.2]{davison2013b}. The inclusions $R_d^{\theta \mhyphen \mathsf{ss}} \hookrightarrow R_d$ induce an algebra homomorphism $\mathcal{H}^{\theta}_{Q,W, \mu} \rightarrow \mathcal{H}_{Q, W, \mu}^{\theta \mhyphen \mathsf{ss}}$. Here $\mathcal{H}^{\theta}_{Q,W, \mu} \subset \mathcal{H}_{Q,W}$ is the subalgebra associated to the submonoid $\Lambda^+_{Q,\mu} \subset \Lambda_Q^+$, with no semistability imposed.

Suppose now that $Q$ has an involution and duality structure and that $\theta$ is $\sigma$-compatible. We call a potential $S$-compatible if its associated trace maps are invariant under the isomorphisms $R_d \xrightarrow[]{\sim} R_{\sigma(d)}$. In this case, by restriction we obtain maps $\tr(W)^{\sigma}_e : R^{\sigma, \theta \mhyphen \mathsf{ss}}_e \rightarrow \mathbb{C}$ and $\tr(W)^{\sigma}_{d^{\prime},e^{\prime}}: R^{\sigma, \theta \mhyphen \mathsf{ss}}_{d^{\prime},e^{\prime}} \rightarrow \mathbb{C}$ which are invariant under $\mathsf{G}^{\sigma}_e$ and $\mathsf{G}^{\sigma}_{d^{\prime},e^{\prime}}$, respectively. We need the following observation.

\begin{Lem}
\label{lem:vanCycleScale}
Let $X$ be a complex manifold and $f: X \rightarrow \mathbb{C}$ a holomorphic function. For any $c \in \mathbb{R}_{>0}$ there is a canonical isomorphism of vanishing cycle functors $\varphi_f \simeq \varphi_{cf}$. In particular, $\varphi_f \mathbb{Q}_X \simeq\varphi_{cf} \mathbb{Q}_X$.
\end{Lem}

\begin{Prop}
\label{prop:ssCoHM}
Let $\theta$ be a $\sigma$-compatible stability and let $W$ be a $S$-compatible potential. Then
\[
\mathcal{M}_{Q,W}^{\theta \mhyphen \mathsf{ss}} = \bigoplus_{ e \in \Lambda_Q^{\sigma, +} } H^{\bullet}_{c, \mathsf{G}^{\sigma}_e}(R^{\sigma, \theta \mhyphen \mathsf{ss}}_e, \varphi_{\tr(W)_e^{\sigma}})^{\vee}\{ \mathcal{E}(e) \slash 2 \}
\]
has a cohomological Hall module structure over $\mathcal{H}_{Q,W,\mu=0}^{\theta \mhyphen \mathsf{ss}}$. Moreover, the map $\mathcal{M}_{Q,W} \rightarrow \mathcal{M}_{Q,W}^{\theta \mhyphen \mathsf{ss}}$ induced by the $\mathsf{G}^{\sigma}_e$-equivariant open inclusions $R^{\sigma, \theta \mhyphen \mathsf{ss}}_e \hookrightarrow R^{\sigma}_e$ is a module homomorphism over $\mathcal{H}^{\theta}_{Q,W,\mu=0} \rightarrow \mathcal{H}_{Q,W,\mu=0}^{\theta \mhyphen \mathsf{ss}}$.
\end{Prop}

\begin{proof}
We use the following simple result. Let $U \subset N$ be an isotropic subrepresentation. If $U$ is semistable of slope zero and $N \git U$ is $\sigma$-semistable, then $N$ is $\sigma$-semistable. Indeed, in this situation we obtain a pair of short exact sequences,
\[
0 \rightarrow U \rightarrow U^{\perp} \rightarrow N \git U \rightarrow 0, \qquad \qquad 0 \rightarrow U^{\perp} \rightarrow N \rightarrow S(U) \rightarrow 0.
\]
Since $N \git U$ is $\sigma$-semistable, it is semistable \cite[Proposition 3.2]{mbyoung2015}. The exact sequences then imply that $U^{\perp}$ and $N$ are semistable of slope zero. Hence $N$ is $\sigma$-semistable.

Using the above result, for a pair $(d,e) \in \Lambda_{Q,\mu=0}^+ \times \Lambda_Q^{\sigma,+}$ we obtain morphisms $R^{\theta \mhyphen \mathsf{ss}}_d \times R^{\sigma, \theta \mhyphen \mathsf{ss}}_e
\overset{\pi}{\twoheadleftarrow} R^{\sigma, \theta \mhyphen \mathsf{ss}}_{d,e} \overset{i}{\hookrightarrow} R^{\sigma,\theta \mhyphen \mathsf{ss}}_{H(d) + e}$ for which
\[
i^* \tr(W)^{\sigma}_{H(d) + e} = \tr(W)^{\sigma}_{d, e} = \pi^* \left( 2 \, \tr( W)_d \boxplus \tr(W)^{\sigma}_e \right).
\]
Combining Lemma \ref{lem:vanCycleScale} with the Thom-Sebastiani isomorphism \cite{massey2001} gives
\begin{align*}
H^{\bullet}_{c, \mathsf{GL}_d}(R^{\theta \mhyphen \mathsf{ss}}_d,  \varphi_{\tr(W)_d})^{\vee}  \otimes  H^{\bullet}_{c, \mathsf{G}^{\sigma}_e}&(R^{\sigma, \theta \mhyphen \mathsf{ss}}_e , \varphi_{\tr(W)^{\sigma}_e})^{\vee} \xrightarrow[]{\sim} \\
& H^{\bullet}_{c, \mathsf{GL}_d \times \mathsf{G}^{\sigma}_e }(R^{\theta \mhyphen \mathsf{ss}}_d \times R^{\sigma, \theta \mhyphen \mathsf{ss}}_e, \varphi_{2 \tr( W)_d \boxplus \tr(W)^{\sigma}_e})^{\vee}.
\end{align*}
From this point on the construction of the $\mathcal{H}_{Q,W,\mu=0}^{\theta \mhyphen \mathsf{ss}}$-module structure of $\mathcal{M}_{Q,W}^{\theta \mhyphen \mathsf{ss}}$ is the natural common generalization of \cite[\S 7]{kontsevich2011} and Section \ref{sec:cohmDef}. We omit the details.

The second statement follows from the fact that the diagram
\[
\begin{tikzpicture}
  \matrix (m) [matrix of math nodes,nodes={anchor=center},row sep=1.6em,column sep=1.6em,minimum width=1.4em] {
R^{\sigma, \theta \mhyphen \mathsf{ss}}_{d,e} & R^{\sigma}_{d,e}  \\
R^{\theta \mhyphen \mathsf{ss}}_d \times R^{\sigma, \theta \mhyphen \mathsf{ss}}_e & R_d \times R_e^{\sigma} \\};
\draw  (m-1-1) edge[right hook->] (m-1-2);
\draw  (m-1-1) edge[->] (m-2-1);
\draw  (m-2-1) edge[right hook->] (m-2-2);
\draw  (m-1-2) edge[->] (m-2-2);\end{tikzpicture}
\]
is Cartesian, which in turn follows from the first paragraph of the proof.
\end{proof}

When $Q$ is $\sigma$-symmetric and $W=0$, define $W_Q^{\mathsf{prim}, \theta} = \mathcal{M}^{\theta \mhyphen \mathsf{ss}}_Q \slash (\mathcal{H}^{\theta \mhyphen \mathsf{ss}}_{Q,\mu=0,+} \star \mathcal{M}^{\theta \mhyphen \mathsf{ss}}_Q)$ with associated motivic invariant $\Omega_Q^{\sigma, \theta}$.
As in Section \ref{sec:symCOHM}, we expect $\mathcal{M}_Q^{\theta \mhyphen \mathsf{ss}}$ to be a direct sum of free modules over subalgebras of $\mathcal{H}^{\theta \mhyphen \mathsf{ss}}_{Q, \mu =0}$, leading to a factorization
\[
A^{\sigma, \theta \mhyphen \mathsf{ss}}_Q \overset{\mbox{\tiny (Conj.)}}{=} \sum_{e \in \Lambda_Q^{\sigma,+}} A_{Q, \mu=0}^{\theta \mhyphen \mathsf{ss}}(e) \cdot \Omega_{Q,e}^{\sigma, \theta} \xi^e.
\]
If such a factorization indeed exists, then $\Omega_Q^{\sigma, \theta}$ is independent of $\theta$. This follows from a short argument using the wall-crossing formula \cite[Theorem 4.5]{mbyoung2015}
\begin{equation}
\label{eq:oriWallCrossing}
A^{\sigma}_Q = \prod_{\mu \in \mathbb{Q}_{>0}} A^{\theta \mhyphen \mathsf{ss}}_{Q, \mu} \star A_Q^{\sigma, \theta \mhyphen \mathsf{ss}}.
\end{equation}

To end this section we outline the expected structure of $\mathcal{M}_{Q,W}$; Sections \ref{sec:examples} and \ref{sec:finiteTypeQuivers} give evidence for these expectations. Let $Q$, $W$ and $\theta$ be arbitrary. Motivated by the existence and uniqueness of Harder-Narasimhan filtrations, in \cite[\S 5.2]{kontsevich2011} (see also \cite[\S 8.1]{chuang2014}) it was asked if there exist algebra embeddings $\mathcal{H}_{Q,W, \mu}^{\theta \mhyphen \mathsf{ss}} \hookrightarrow \mathcal{H}_{Q,W}$ such that the slope ordered CoHA multiplication
\[
\overset{\leftarrow}{\boxtimes}^{\mathsf{tw}}_{\mu \in \mathbb{Q}} \,
\mathcal{H}^{\theta \mhyphen \mathsf{ss}}_{Q,W,\mu} \rightarrow \mathcal{H}_{Q,W}
\]
is an isomorphism in $D^{lb}(\mathsf{Vect}_{\mathbb{Z}})_{\Lambda_Q^+}$. If $\theta$ is generic, then $\mathcal{H}^{\theta \mhyphen \mathsf{ss}}_{Q,W,\mu}$ is conjectured to be the universal enveloping algebra of a Lie superalgebra structure on $V^{\mathsf{prim},\theta}_{Q,W, \mu} \otimes \mathbb{Q}[u]$. In this way $\mathcal{H}_{Q,W}$ obtains a Poincar\'{e}-Birkhoff-Witt (PBW) type basis. See \cite{davison2016a} for results in this direction. Conjecturally, $V_{Q,W}^{\mathsf{prim}, \theta}$ can be interpreted as a space of oriented single-particle BPS states.

In the orientifold setting, each self-dual representation $M$ has a unique $\sigma$-Harder-Narasimhan filtration \cite[Proposition 3.3]{mbyoung2015}, an isotropic filtration
\[
0 = U_0 \subset U_1 \subset \cdots \subset U_r \subset M
\]
such that $U_1 \slash U_0, \dots, U_r \slash U_{r-1}$ are semistable with strictly decreasing positive slopes and $M \git U_r$ is zero or $\sigma$-semistable. It is then natural to ask for a $\mathcal{H}_{Q,W,\mu=0}^{\theta \mhyphen \mathsf{ss}}$-module\footnote{As above, we should restrict to subalgebras of $\mathcal{H}_{Q,W,\mu=0}^{\theta \mhyphen \mathsf{ss}}$.} embedding $\mathcal{M}^{\theta \mhyphen \mathsf{ss}}_{Q,W} \hookrightarrow \mathcal{M}_{Q,W}$ such that the ordered CoHA action
\begin{equation}
\label{eq:moduleFactorization}
\overset{\leftarrow}{\boxtimes}^{\mathsf{tw}}_{\mu \in \mathbb{Q}_{>0}} \,
\mathcal{H}^{\theta \mhyphen \mathsf{ss}}_{Q,W,\mu} \boxtimes^{S \mhyphen \mathsf{tw}} \mathcal{M}^{\theta \mhyphen \mathsf{ss}}_{Q,W} \xrightarrow[]{\star} \mathcal{M}_{Q,W}
\end{equation}
is an isomorphism in $D^{lb}(\mathsf{Vect}_{\mathbb{Z}})_{\Lambda_Q^{\sigma,+}}$. Together with the natural extension of Conjecture \ref{conj:freeCOHM} to $\mathcal{M}_{Q,W}^{\theta \mhyphen \mathsf{ss}}$, an isomorphism of the form \eqref{eq:moduleFactorization} would determine a PBW type basis of $\mathcal{M}_{Q,W}$ in terms of $W_{Q,W}^{\mathsf{prim}, \theta}$ and the PBW bases of $\mathcal{H}_{Q,W, \mu \geq 0}^{\theta \mhyphen \mathsf{ss}}$. Conjecturally, $W_{Q,W}^{\mathsf{prim}, \theta}$ can be interpreted as a space of single-particle BPS states of the orientifolded theory. Decompositions similar to \eqref{eq:moduleFactorization} occur in physical definitions of unoriented BPS invariants \cite{sinha2000}, \cite{walcher2009}.

\section{\texorpdfstring{$\sigma$}{}-Symmetric examples}
\label{sec:examples}

\subsection{Disjoint union quivers}
\label{sec:disjointUnions}

Let $Q$ and $Q^{\prime}$ be arbitrary quivers. The disjoint union quiver $Q \sqcup Q^{\prime}$ has nodes $Q_0 \sqcup Q_0^{\prime}$ and arrows $Q_1 \sqcup Q_1^{\prime}$. The opposite quiver $Q^{\mathsf{op}}$ has nodes $Q_0$ and an arrow $j \xrightarrow[]{\alpha^{\mathsf{op}}} i$ for each arrow $i \xrightarrow[]{\alpha} j$ of $Q$.

\begin{Lem}
\label{lem:cohaFunctor}
There are canonical algebra isomorphisms
\[
\mathcal{H}_{Q \sqcup Q^{\prime}} \simeq \mathcal{H}_Q \otimes \mathcal{H}_{Q^{\prime}}, \qquad 
\mathcal{H}_{Q^{\mathsf{op}}} \simeq \mathcal{H}_Q^{\mathsf{op}}
\]
where $\mathcal{H}_Q^{\mathsf{op}}$ is the opposite algebra of $\mathcal{H}_Q$.
\end{Lem}

\begin{proof}

The isomorphism $\mathcal{H}_{Q \sqcup Q^{\prime}} \xrightarrow[]{\sim} \mathcal{H}_Q \otimes \mathcal{H}_{Q^{\prime}}$ is the pullback along the isomorphisms $R_d (Q) \times R_{d^{\prime}}(Q^{\prime}) \xrightarrow[]{\sim} R_{(d,d^{\prime})} (Q \sqcup Q^{\prime})$ while $\mathcal{H}_{Q^{\mathsf{op}}} \xrightarrow[]{\sim} \mathcal{H}_Q^{\mathsf{op}}$ is the pullback along the isomorphisms $R_d(Q) \xrightarrow[]{\sim} R_d(Q^{\mathsf{op}})$ sending a representation to its transpose.
\end{proof}

The quiver $Q^{\sqcup} = Q \sqcup Q^{\mathsf{op}}$ has a canonical involution $\sigma$ which swaps the nodes and arrows of $Q$ and $Q^{\mathsf{op}}$. Representations of $Q^{\sqcup}$ are of the form $U_1 \oplus S(U_2)$ for $U_1, U_2 \in \mathsf{Rep}_{\mathbb{C}}(Q)$. Self-dual representations are hyperbolics on representations of $Q$. This gives isomorphisms $R_d \xrightarrow[]{\sim} R_{H(d)}^{\sigma}$ which induce a vector space isomorphism $\mathcal{M}_{Q^{\sqcup}} \rightarrow \mathcal{H}_Q$. Lemma \ref{lem:cohaFunctor} implies that $\mathcal{M}_{Q^{\sqcup}}$ is a $\mathcal{H}_Q \otimes \mathcal{H}_Q^{\mathsf{op}}$-module. Similarly, $\mathcal{H}_Q$ is the regular left $\mathcal{H}_Q$-bimodule.

\begin{Thm}
\label{thm:disjointCoHM}
The map $\mathcal{M}_{Q^{\sqcup}} \rightarrow \mathcal{H}_Q$ is an isomorphism of $\mathcal{H}_Q \otimes \mathcal{H}_Q^{\mathsf{op}}$-modules.
\end{Thm}

\begin{proof}
The action of $f_1 \otimes f_3 \in\mathcal{H}_Q \otimes \mathcal{H}_Q^{\mathsf{op}}$ on $f_2 \in \mathcal{H}_Q$ is $f_1 \cdot f_2 \cdot f_3 \in \mathcal{H}_Q$, which is in turn the image of $f_1 \otimes f_2 \otimes f_3$ under the composition (omitting degree shifts)
\[
\bigotimes_{i=k}^3 
H^{\bullet}_{\mathsf{GL}_{d_k}}(R_{d_k})  \xrightarrow[]{\sim} H^{\bullet}_{\mathsf{GL}_{d_1, d_2, d_3}}(R_{d_1, d_2, d_3}) \rightarrow H^{\bullet}_{\mathsf{GL}_{d_1 + d_2 +d_3}}(R_{d_1 + d_2 + d_3}).
\]
The isomorphism $R_d \simeq R^{\sigma}_{H(d)}$ identifies $R_{d_1+ \sigma(d_2), H(d_3)}^{\sigma} \subset R_{H(d_1 + d_2 + d_3)}^{\sigma}$ with the subspace $R_{d_1, d_3, d_2} \subset R_{d_1 +d_2 + d_3}$ preserving the $Q_0$-graded flag
\[
\mathbb{C}^{d_1} \subset (\mathbb{C}^{\sigma(d_2)})^{\perp} \cap \mathbb{C}^{d_1 +d_2 + d_3} \subset  \mathbb{C}^{d_1 +d_2 + d_3}
\]
and identifies $\mathsf{G}_{d_1+ \sigma(d_2), H(d_3)}^{\sigma} \subset \mathsf{G}_{H(d_1 +d_2 + d_3)}^{\sigma}$ with $\mathsf{GL}_{d_1, d_3, d_2} \subset \mathsf{GL}_{d_1 + d_2 + d_3}$. Using this we find that $(f_1 \otimes f_3) \star f_2$ is equal to $f_1 \cdot f_2 \cdot f_3$. That the isomorphism $\mathcal{M}_{Q^{\sqcup}} \xrightarrow[]{\sim} \mathcal{H}_Q$ respects the gradings follows from the equality
\begin{equation}
\label{eqn:sdEulerDisjoint}
\mathcal{E}_{Q^{\sqcup}}(U_1 \oplus S(U_2)) = \chi_Q(U_2, U_1),
\end{equation}
which holds for all $U_1, U_2 \in \mathsf{Rep}_{\mathbb{C}}(Q)$.
\end{proof}

\begin{Cor}
Conjectures \ref{conj:freeCOHM} and \ref{conj:hodgeEqual} hold for $Q^{\sqcup}$.
\end{Cor}
\begin{proof}
By equation \eqref{eqn:sdEulerDisjoint}, $Q^{\sqcup}$ is $\sigma$-symmetric if and only if $Q$ is symmetric. Assume then that $Q$ is symmetric and consider $\mathcal{H}_Q$ with its twisted supercommutative multiplication. Theorems \ref{thm:freeCoHA} and \ref{thm:disjointCoHM} give algebra isomorphisms
\[
\mathcal{H}_{Q^{\sqcup}} \simeq \mathcal{H}_Q\otimes \mathcal{H}_Q^{\mathsf{op}} \simeq \Sym \big( \big(V_Q^{\mathsf{prim}} \oplus S_{\mathcal{H}}(V_Q^{\mathsf{prim}}) \big) \otimes \mathbb{Q}[u] \big).
\]
Lift the supercommutative twist of $\mathcal{H}_Q$ by taking $\mathcal{M}_{Q^{\sqcup}}$ to be the regular super $\mathcal{H}_Q$-bimodule. Then $\mathcal{M}_{Q^{\sqcup}}$ is a rank one free module with basis $\mathbf{1}_0^{\sigma} \in \mathcal{M}_{Q^{\sqcup},0}$ over the subalgebra of $\mathcal{H}_{Q^{\sqcup}}$ generated by the image of
\[
V_Q \hookrightarrow V_Q \oplus S_{\mathcal{H}}(V_Q) \simeq \big(V_Q^{\mathsf{prim}} \oplus S_{\mathcal{H}}(V_Q^{\mathsf{prim}}) \big) \otimes \mathbb{Q}[u], \qquad v \mapsto v + S_{\mathcal{H}}(v).
\]
Conjecture \ref{conj:freeCOHM} follows. Conjecture \ref{conj:hodgeEqual} holds as $\mathfrak{M}_e^{\sigma, \mathsf{st}}(Q^{\sqcup}) = \varnothing$ if $e \neq 0$.
\end{proof}

Similarly, $\mathcal{M}_{Q^{\sqcup}}$ is a rank one free $\mathcal{H}_Q$-module. This is the PBW factorization \eqref{eq:moduleFactorization} associated to a $\sigma$-compatible stability $\theta$ whose restriction to $\Lambda_Q^+ \subset \Lambda_{Q^{\sqcup}}^+$ is positive. Again, we have $\mathfrak{M}_e^{\sigma, \theta \mhyphen \mathsf{st}}(Q^{\sqcup}) = \varnothing$ if $e \neq 0$.

\subsection{Loop quivers}
\label{sec:quiverWithOneNode}

Let $L_m$ be the quiver with one node and $m \geq 0$ loops. It is symmetric and $\mathcal{H}_{L_m}$ is supercommutative. If $f_1 \in \mathcal{H}_{L_m,d^{\prime}}$ and $f_2 \in \mathcal{H}_{L_m,d^{\prime \prime}}$, then
\[
f_1 \cdot f_2 = \sum_{\pi \in \mathfrak{sh}_{d^{\prime}, d^{\prime \prime}}} \pi \Big( f_1(x^{\prime}_1, \dots, x^{\prime}_{d^{\prime}})  f_2(x^{\prime \prime}_1, \dots,  x^{\prime \prime}_{d^{\prime \prime}})  \prod_{l=1}^{d^{\prime \prime}}  \prod_{k=1}^{d^{\prime}} ( x^{\prime \prime}_l - x^{\prime}_k )^{m-1} \Big).
\]

The (unique) involution of $L_m$ fixes the node and arrows. A duality structure is determined by a sign $s$ and signs $\tau_1, \dots, \tau_m$. Suppose that $\tau_+$ of the latter are positive and $\tau_- = m - \tau_+$ are negative. Note that $L_m$ is $\sigma$-symmetric. When $s=1$ Proposition \ref{prop:cohmDecomp} gives $\mathcal{M}_{L_m} = \mathcal{M}_{L_m}^D \oplus \mathcal{M}_{L_m}^B$, the summands associated to even and odd dimensional self-dual representations, respectively. When $s=-1$ we write $\mathcal{M}_{L_m}^C$ for $\mathcal{M}_{L_m}$. Given $f \in \mathcal{H}_{L_m,d}$ and $g \in \mathcal{M}_{L_m,e}$, we have
\begin{align*}
f \star g = 2^{(\tau_s - \frac{1-s}{2})d}  \sum_{ \pi \in \mathfrak{sh}_{d,e}^{\sigma}} & \pi  \Bigg(   f   (x_1, \dots, x_d) g(z_1, \dots, z_{\lfloor \frac{e}{2} \rfloor}) \times \\ & \prod_{i=1}^d (-x_i)^{N(s,\tau)}  \Big( \prod_{1 \leq i < j \leq d} (- x_i - x_j)  \prod_{i=1}^d \prod_{j=1}^{\lfloor \frac{e}{2} \rfloor} (x_i^2 - z_j^2) \Big)^{m-1} \Bigg)
\end{align*}
where
\[
N(s,\tau) = \begin{cases}  m + \tau_+ -1 & \mbox{in type } B, \\  \tau_- -1 & \mbox{in type } C, \\ \tau_+ & \mbox{in type } D.  \end{cases}
\]
Since the cases $m=0,1$ serve as building blocks for more complicated examples, we will describe them in detail.

\subsubsection{Zero loops}
\label{sec:zeroLoopQuiver}
The algebra $\mathcal{H}_{L_0}$ is free supercommutative on the odd variables $x^i \in \mathcal{H}_{L_0,1}$, $i \geq 0$, of degree $(1, 2i+1)$ \cite[\S 2.5]{kontsevich2011}. Explicitly, if $\mathbf{i}=(i_d, \dots, i_1)$ is strictly decreasing, then $x^{i_1} \cdots  x^{i_d} = s_{\mathbf{i} - \delta_d}$. Here $s_{\lambda}$ is the Schur polynomial of a partition $\lambda$ and $\delta_r = (r-1, \dots, 1, 0)$. Hence $V_{L_0}^{\mathsf{prim}} = \mathbb{Q} \cdot \mathbf{1}_1 = \mathbb{Q}_{(1,1)}$.

Let $\phi : \mathcal{H}_{L_0} \rightarrow \mathcal{H}_{L_0}$ be the unital algebra automorphism determined by $\phi(x^i) = 2 x^i$ and let $(\mathcal{M}_{L_m}^B)_{\phi}$ be the corresponding twisted $\mathcal{H}_{L_m}$-module. Using the explicit form of $\star$ we see that $(\mathcal{M}_{L_m}^B)_{\phi} \simeq \mathcal{M}_{L_m}^C[1]$ as graded $\mathcal{H}_{L_0}$-modules, where $[1]$ denotes $\Lambda_Q^{\sigma,+}$-degree shift by one. We therefore consider only $\mathcal{M}_{L_0}^B$ and $\mathcal{M}_{L_0}^D$.

Given $f \in \mathbb{Q}[x_1, \dots, x_d]$ set $\tilde{f}(x_1, \dots, x_d) = f(x_1^2, \dots, x_d^2)$. Let $\mathbf{i}$ be a strictly decreasing partition of length $d$. Short induction arguments show the following:
\begin{enumerate}
\item Type $B$: If all $i_j$ are odd, then $s_{\mathbf{i} - \delta_d} \star \mathbf{1}^{\sigma}_0 = (-2)^{d} \tilde{s}_{\frac{\mathbf{i} -\mathbf{1}}{2} - \delta_d}$.

\item Type $D$: If all $i_j$ are even, then $s_{\mathbf{i} - \delta_d} \star \mathbf{1}^{\sigma}_0 = 2^d \tilde{s}_{\frac{\mathbf{i}}{2} - \delta_d}$.
\end{enumerate}
Let $\mathcal{H}_{L_0}^{\mathsf{even}}$ and $\mathcal{H}_{L_0}^{\mathsf{odd}}$ be the algebras generated by $\{ x^{2i} \}_{i \geq 0}$ and $\{ x^{2i+1} \}_{i \geq 0}$, respectively. Equivalently, $\mathcal{H}_{L_0}^{\mathsf{even}} = \Sym(\mathbb{Q}_{(1,1)} \otimes \mathbb{Q}[u^2])$ and $\mathcal{H}_{L_0}^{\mathsf{odd}} = \Sym(\mathbb{Q}_{(1,1)} \otimes u\mathbb{Q}[u^2])$. These are the subalgebras of the CoHA introduced above Lemma \ref{lem:superModule}. 

\begin{Prop}
\label{prop:zeroLoopCoHM}
\leavevmode
\begin{enumerate}[leftmargin=0cm,itemindent=.6cm,labelwidth=\itemindent,labelsep=0cm,align=left]

\item $\mathcal{M}_{L_0}^B$ is a free $\mathcal{H}_{L_0}^{\mathsf{odd}}$-module with basis $\mathbf{1}^{\sigma}_1 \in \mathcal{M}_{L_0,1}^B$.

\item $\mathcal{M}_{L_0}^D$ is a free $\mathcal{H}_{L_0}^{\mathsf{even}}$-module with basis $\mathbf{1}^{\sigma}_0 \in \mathcal{M}_{L_0,0}^D$.
\end{enumerate}
In particular, $\Omega_{L_0}^B =\xi$, $\Omega_{L_0}^C = \xi^0$ and $\Omega_{L_0}^D = \xi^0$ and Conjectures \ref{conj:freeCOHM} and \ref{conj:hodgeEqual} hold.
\end{Prop}

\begin{proof}
The map $\mathbf{i} \mapsto \frac{\mathbf{i}-\mathbf{1}}{2}$ is a bijection between the set of strictly decreasing purely odd partitions of length $d$ and the set of strictly decreasing partitions of length $d$. Since the Schur functions $\tilde{s}_{\mathbf{i}^{\prime} - \delta_d}$ parameterized by the former set are a basis of $\mathcal{M}_{L_0,2d+1}^B$, the first statement follows. In type $D$ we use instead the bijection $\mathbf{i} \mapsto \frac{\mathbf{i}}{2}$ between the sets of strictly decreasing purely even and strictly decreasing partitions. That Conjecture \ref{conj:hodgeEqual} holds follows from the observations
\[
\mathfrak{M}_{2e}^{\mathfrak{sp} , \mathsf{st}} = \varnothing, \;\; e \geq 1, \qquad
\mathfrak{M}_e^{\mathfrak{o} , \mathsf{st}} = \begin{cases}
\Spec(\mathbb{C}) & \mbox{ if } e =1, \\
\varnothing & \mbox{ if } e \geq 2,
\end{cases}
\]
with $\mathfrak{sp}$ and $\mathfrak{o}$ indicating type $C$ or types $B$ or $D$, respectively.
\end{proof}

\subsubsection{One loop}
The algebra $\mathcal{H}_{L_1}$ is free supercommutative on the even variables $x^i \in \mathcal{H}_{L_1,1}$, $i \geq 0$, of degree $(1, 2i)$ \cite[\S 2.5]{kontsevich2011}. Explicitly, $x^{i_1} \cdots x^{i_d} = N(\mathbf{i}) m_{\mathbf{i}}$ where $N(\mathbf{i}) = \prod_{ k \geq 0}  \# \{ j \geq 1\; \vert \; i_j =k \} !$ and $m_{\mathbf{i}}$ is the monomial symmetric polynomial. Hence $V_{L_1}^{\mathsf{prim}}=\mathbb{Q} \cdot \mathbf{1}_1 = \mathbb{Q}_{(1,0)}$.

Similar to the case $m=0$, we have module isomorphisms $\mathcal{M}_{L_1}^B \simeq \mathcal{M}_{L_1}^D[1]$ if $\tau=1$ and $(\mathcal{M}_{L_1}^B)_{\phi} \simeq \mathcal{M}_{L_1}^C[1] \simeq (\mathcal{M}_{L_1}^D)_{\phi}[1]$ if $\tau =-1$. So we consider only $\mathcal{M}_{L_1}^B$ if $\tau=-1$ and $\mathcal{M}_{L_1}^{C,D}$ if $\tau=1$. Given a partition $\mathbf{i}$ of length $d$, we find:
\begin{enumerate} 
\item Type $B$, $\tau=-1$: If $\mathbf{i}$ is purely even, then $m_{\mathbf{i}} \star \mathbf{1}^{\sigma}_0 = 2^d \tilde{m}_{\frac{\mathbf{i}}{2}}$.

\item Type C, $\tau=1$: If $\mathbf{i}$ is purely odd, then $m_{\mathbf{i}} \star \mathbf{1}^{\sigma}_0 = (-2)^d  \tilde{m}_{\frac{\mathbf{i} - \mathbf{1}}{2}}$.

\item Type D, $\tau=1$: If $\mathbf{i}$ is purely odd, then $m_{\mathbf{i}} \star \mathbf{1}^{\sigma}_{2e} = (-2)^d \tilde{m}_{(\frac{\mathbf{i}+\mathbf{1}}{2}, \mathbf{0}^e)}$, where $(\mathbf{i}, \mathbf{0}^e)$ is the length $d+e$ partition obtained by appending $e$ zeros to $\mathbf{i}$.
\end{enumerate}
Let $\mathcal{H}^{\mathsf{even}}_{L_1} = \Sym(\mathbb{Q}_{(1,0)} \otimes \mathbb{Q}[u^2])$ and $\mathcal{H}_{L_1}^{\mathsf{odd}} = \Sym(\mathbb{Q}_{(1,0)} \otimes u\mathbb{Q}[u^2])$.

\begin{Prop}
\label{prop:oneLoopCoHM}
\leavevmode
\begin{enumerate}
\item If $\tau =-1$, then $\mathcal{M}^B_{L_1}$ is a free $\mathcal{H}_{L_1}^{\mathsf{even}}$-module with basis $\mathbf{1}^{\sigma}_0 \in \mathcal{M}^B_{L_1,1}$.

\item If $\tau =1$, then $\mathcal{M}_{L_1}^C$ is a free $\mathcal{H}_{L_1}^{\mathsf{odd}}$-module with basis $\mathbf{1}^{\sigma}_0 \in \mathcal{M}^C_{L_1,0}$.

\item If $\tau =1$, then $\mathcal{M}_{L_1}^D$ is a free $\mathcal{H}_{L_1}^{\mathsf{odd}}$-module with basis $\mathbf{1}^{\sigma}_{2e} \in \mathcal{M}^D_{L_1,2e}$, $e \geq 0$.
\end{enumerate}
In particular, if $\tau =-1$, then $\Omega_{L_1}^B =\xi$, $\Omega_{L_1}^C = \xi^0$ and $\Omega_{L_1}^D = \xi^0$ while if $\tau=1$, then
\[
\Omega_{L_1}^B =\frac{q^{-\frac{1}{2}} \xi}{1-q^{-1} \xi^2}, \qquad    \Omega_{L_1}^C = \xi^0, \qquad \Omega_{L_1}^D = \frac{1}{1-q^{-1} \xi^2}.
\]
Conjectures \ref{conj:freeCOHM} and \ref{conj:hodgeEqual} hold for $L_1$.\end{Prop}

\begin{proof}
Freeness is proved as in Proposition \ref{prop:zeroLoopCoHM}. When $\tau =1$ we have
\[
\mathfrak{M}_{2e}^{\mathfrak{sp} , \mathsf{st}} = \varnothing, \;\; e \geq 1, \qquad \mathfrak{M}_e^{\mathfrak{o} , \mathsf{st}} =\begin{cases}
\Spec(\mathbb{C}) & \mbox{ if } e =1, \\
\varnothing & \mbox{ if } e \geq 2
\end{cases}
\]
while for $\tau=-1$ we have $\mathfrak{M}_{2e}^{\mathfrak{sp} , \mathsf{st}} = \varnothing$ and
\[
\mathfrak{M}_e^{\mathfrak{o}} = \mathsf{Symm}_{e \times e} \git \mathsf{O}_e \simeq  \Sym^e\, \mathbb{C}, \qquad \mathfrak{M}_e^{\mathfrak{o} , \mathsf{st}} \simeq \Sym^e \, \mathbb{C} \backslash \Delta, \qquad e \geq 1.
\]
Here $\mathsf{Symm}_{e \times e}$ is the variety of symmetric $e \times e$ matrices and $\Delta$ is the big diagonal. Conjecture \ref{conj:hodgeEqual} is now immediate except in type $D$ with $\tau=1$ where it reads
\[
PH^0(\mathfrak{M}_e^{\mathfrak{o}, \mathsf{st}}) \simeq \mathbb{Q}(0), \qquad PH^k(\mathfrak{M}_e^{\mathfrak{o}, \mathsf{st}}) = 0, \qquad e, k \geq 1
\]
and follows from the isomorphism $H^{\bullet}(\Sym^e \, \mathbb{C} \backslash \Delta) \simeq H^{\bullet}(\mathbb{C} \backslash \{0\})$.
\end{proof}

\subsubsection{Higher loops} 
\label{sec:higerLoopCalcs}

If $m \geq 2$, then neither $\mathcal{H}_{L_m}$ nor $\mathcal{M}_{L_m}$ is finitely generated. The twisting factor $(-1)^{\chi(e,d) + \mathcal{E}(d)}$ (see Section \ref{sec:symCOHM}) depends on $e$ only through the type $B,C$ or $D$. We therefore write $\widetilde{\mathcal{H}}_Q$ for the subalgebra $\mathcal{H}_Q(e) \subset \mathcal{H}_Q$. Each homogeneous summand $\mathcal{H}_{L_m,(d,k)}$ is isotypical as a twisted $\mathbb{Z}_2$-representation and the equivariant DT invariants are simply
\[
\widetilde{\Omega}_{L_m,(2d,k)}^+ = \begin{cases} \Omega_{L_m,(d,k)} & \mbox{if } \chi(e,d) + \mathcal{E}(d) + \frac{k -\chi(d,d)}{2} \equiv 0 \mod 2, \\
0 & \mbox{if } \chi(e,d) + \mathcal{E}(d) + \frac{k -\chi(d,d)}{2} \equiv 1 \mod 2
\end{cases}
\]
and
\[
\widetilde{\Omega}_{L_m,(2d,k)}^- = \begin{cases} 0 & \mbox{if } \chi(e,d) + \mathcal{E}(d) + \frac{k -\chi(d,d)}{2} \equiv 0 \mod 2, \\
\Omega_{L_m, (d,k)} & \mbox{if } \chi(e,d) + \mathcal{E}(d) + \frac{k -\chi(d,d)}{2} \equiv 1 \mod 2. 
\end{cases}
\]

Since $Q$ is $L_m$ in what follows, we sometimes omit it from the notation. Let $\{v_{d,\beta}\}_{1 \leq \beta \leq \dim V^{\mathsf{prim}}_d}$ be an ordered homogeneous basis of $V^{\mathsf{prim}}_d$. Then $\{v_{d,\beta} \sigma_d^m \}_{d, \beta, m}$ is a basis of $V_Q$ with a natural lexicographic order $\geq$. Let also $\{w_{e, \beta}\}_{1 \leq \beta \leq \dim W^{\mathsf{prim}}_e}$ be an ordered homogeneous basis of $W^{\mathsf{prim}}_e$. For each $e \in \Lambda_Q^{\sigma,+}$ let $\mathsf{Seq}_e^{\sigma}$ be the set of all sequences of the form
\[
(v_{d^1, \beta_1} \sigma_{d^1}^{m_1}, \dots, v_{d^l, \beta_l} \sigma_{d^l}^{m_l}; w_{e^{\infty}, \beta_{\infty}})
\]
which have the following properties:
\begin{enumerate}
\item Each $d^p$ is non-zero and $e^{\infty} \geq 0$.

\item We have $e = \sum_{p=1}^l H(d^p) + e^{\infty}$.

\item We have $v_{d^1, \beta_1} \sigma_{d^1}^{m_1} \geq \cdots \geq v_{d^l, \beta_l} \sigma_{d^l}^{m_l}$.

\item If $(d^p, \beta_p, m_p) = (d^{p+1}, \beta_{p+1}, m_{p+1})$, then $\chi(d^p, d^p) \equiv 0 \mod 2$.

\item If $v_{d^p, \beta_p} \in \mathcal{H}_{(d^p, k)}$, then $\chi(e^{\infty}, d^p) + \mathcal{E}(d^p) \equiv \frac{k +2m_p - \chi(d^p, d^p)}{2} \mod 2$.
\end{enumerate}
A lexicographic order $\geq$ is defined on $\mathsf{Seq}_e^{\sigma}$ by first comparing the CoHM components of the sequence and then comparing the remaining CoHM sequences.

\begin{Thm}
\label{thm:loopFreeCoHM}
Conjecture \ref{conj:freeCOHM} holds for $m$-loop quivers.
\end{Thm}

\begin{proof}
We have seen that the theorem holds if $m \leq 1$, so assume that $m \geq 2$. The general structure of the proof is similar to \cite[\S 3]{efimov2012}; we focus on the differences.

For $t \in \mathsf{Seq}_e^{\sigma}$ let $M_t \in \mathcal{M}_e$ be the corresponding ordered product. We need to show that for each strictly decreasing sequence $t_1 > \cdots > t_n$ in $\mathsf{Seq}_e^{\sigma}$ and each tuple $(\lambda_1, \dots, \lambda_n) \in (\mathbb{Q}^{\times})^n$ we have $\sum_{i=1}^n \lambda_i M_{t_i} \neq 0$ in $\mathcal{M}_e$.

Denote by $\Dim t_1 = (d^1, \dots, d^l; e^{\infty})$ the underlying sequence of dimension vectors of $t_1$. It will sometimes be convenient to replace $e^{\infty}$ with its reduction $d^{\infty} =\lfloor \frac{e^{\infty}}{2} \rfloor$. Using notation from the proof of Theorem \ref{thm:oriDTfinite}, we have an algebra isomorphism
\[
Z_e \simeq X_{d^1} \otimes \cdots \otimes X_{d^l} \otimes Z_{e^{\infty}} = Z_{d^{\bullet}, e^{\infty}}
\]
which induces an algebra embedding
\[
\mathcal{M}_e \hookrightarrow   \mathcal{H}_{d^1} \otimes \cdots \otimes \mathcal{H}_{d^l} \otimes \mathcal{M}_{e^{\infty}} =  
\mathcal{M}_{d^{\bullet}, e^{\infty}}.
\]
Setting $\mathfrak{W}_{d^{\bullet}, e^{\infty}} = \prod_{p=1}^l \mathfrak{S}_{d^p} \times \mathfrak{W}_{e^{\infty}}$ we have $\mathcal{M}_{d^{\bullet}, e^{\infty}} \simeq Z_{d^{\bullet}, e^{\infty}}^{\mathfrak{W}_{d^{\bullet}, e^{\infty}}}$.

As in \cite{efimov2012}, for each $d \in \Lambda_{L_m}^+$ define a subalgebra of $X_d$ by
\[
X^{\mathsf{prim}}_d = \mathbb{Q}[(x_j - x_k) \mid 1 \leq j < k \leq d]
\]
and let $J_d$ be the minimal $\mathfrak{S}_d$-stable $X^{\mathsf{prim}}_d$-submodule of $X_d$ which contains, for each non-trivial decomposition $d = d^{\prime} + d^{\prime \prime}$, the CoHA kernel
\[
\mathcal{K}_{d^{\prime}, d^{\prime \prime}}(x^{\prime}, x^{\prime \prime}) = \prod_{k=1}^{d^{\prime \prime}} \prod_{j=1}^{d^{\prime}} (x^{\prime \prime}_k - x^{\prime}_j)^{m-1}.
\]
Define also a $\mathfrak{W}_{d^{\bullet}, e^{\infty}}$-stable ideal of $Z_{d^{\bullet}, e^{\infty}}$ by
\[
L_{d^{\bullet}, e^{\infty}} = J_{d^1} Z_{d^{\bullet}, e^{\infty}} + \cdots + J_{d^l} Z_{d^{\bullet}, e^{\infty}} + (L_{e^{\infty}}  \cap Z_{e^{\infty}} ) Z_{d^{\bullet}, e^{\infty}}.
\]
Note that $L_{e^{\infty}} \subset Z_{e^{\infty}}$ except in type $C$ with $\tau_-=0$.

Write $z_i^{(p)} \in X_{d^p}$, $1 \leq p \leq l$ and $z_i^{(\infty)} \in Z_{e^{\infty}}$ for the standard algebra generators considered as elements of $Z_{d^{\bullet}, e^{\infty}}$. Then $(z_i^{(q)})^2 - (z_j^{(p)})^2$, $1 \leq p < q \leq \infty$, is not a zero divisor in $Z_{d^{\bullet}, e^{\infty}} / L_{d^{\bullet}, e^{\infty}}$. This can be verified in the same way as the corresponding statement from \cite{efimov2012}.

Consider the composition
\[
\rho: \mathcal{M}_e \hookrightarrow \mathcal{M}_{d^{\bullet}, e^{\infty}} \twoheadrightarrow \mathcal{M}_{d^{\bullet}, e^{\infty}} \slash L_{d^{\bullet}, e^{\infty}}^{\mathfrak{W}_{d^{\bullet}, e^{\infty}}}.
\]
We claim that $\rho (M_{t_i})=0$ whenever $\Dim t_1 > \Dim t_i$; here we view the reduced self-dual component as an ordinary dimension vector and $\Dim t_1$ is ordered so as to be non-increasing, and similarly for $\Dim t_i$.  Indeed, if $\Dim t_1 > \Dim t_i$, then for each $\pi \in \mathfrak{sh}_{\Dim t_i}$ there exists a component of $\Dim t_1$, say $d^*$, which is partitioned by $\pi$ into at least two components of $\Dim t_i$. The summand $M_{t_i}(\pi)$ of $M_{t_i}$ obtained by summing over all lifts of $\pi$ to $\mathfrak{sh}_{\Dim t_i}^{\sigma}$ lies in $(L_{e^{\infty}} \cap Z_{e^{\infty}}) Z_{d^{\bullet}, e^{\infty}}$ if $d^*$ is the reduced self-dual component of $\Dim t_1$ and lies in $J_{d^*} Z_{d^{\bullet}, e^{\infty}}$ otherwise. Summing over $\mathfrak{sh}_{\Dim t_i}$ establishes the claim.

So assume that $\Dim t_1 =  \cdots = \Dim t_n$, again viewing the reduced self-dual component as an ordinary dimension vector. We claim that $\rho(M_{t_i})=0$ unless the self-dual component of $\Dim t_i$ is $e^{\infty}$. Arguing as in the previous paragraph, the only shuffles $\pi \in\mathfrak{sh}_{\Dim t_i}$ for which the contribution of $M_{t_i}(\pi)$ to $\rho(M_{t_i})$ may be non-zero lie in the subgroup
\[
\widetilde{\mathfrak{S}}_{l+1} = \{ \pi \in \mathfrak{S}_{l+1} \mid d^p = d^{\pi(p)} \mbox{ for } 1 \leq p \leq \infty\}.
\]
However, if $e^{\infty}$ is greater than the self-dual component of $\Dim t_i$, then for each $\pi \in \widetilde{\mathfrak{S}}_{l+1}$ we have $M_{t_i}(\pi) \in (L_{e^{\infty}} \cap Z_{e^{\infty}})Z_{d^{\bullet},e^{\infty}}$. Hence we can assume that $\Dim t_1= \cdots = \Dim t_r$ with equal self-dual components.

Suppose that we are not in type $C$ with $\tau_-=0$. The shuffle description gives
\begin{align*}
\rho (& v_{d^1, \beta_1} \sigma_{d^1}^{m_1}  \cdots v_{d^l, \beta_l} \sigma_{d^l}^{m_l} \star w_{e^{\infty}, \beta_{\infty}}) =  \\
  & 2^l F^{\sigma}_{d^{\bullet},e^{\infty}} \sum_{\pi \in \widetilde{\mathfrak{S}}_l} \mathsf{s}(\pi) v_{d^1, \beta_{\pi(1)}} \sigma_{d^1}^{m_{\pi(1)}} \mathcal{K}_{d^1}^{\sigma} \otimes \cdots \otimes v_{d^k, \beta_{\pi(l)}} \sigma_{d^l}^{m_{\pi(l)}} \mathcal{K}_{d^l}^{\sigma} \otimes w_{e^{\infty}, \beta_{\infty}}  \numberthis \label{eq:bigSum}
\end{align*}
with $\mathsf{s}(\pi)$ the Koszul sign associated to $\pi$ and $\mathcal{K}_{d^p}^{\sigma}= \mathcal{K}_{d^p,0}^{\sigma}$. \textit{A priori} the sum is over $\widetilde{\mathfrak{S}}_{l+1}$, but if $\pi \in \widetilde{\mathfrak{S}}_{l+1}$ with $\pi(p) = \infty$ for some $1 \leq p \leq l$, then the $Z_{e^{\infty}}$ factor of $M_{t_i}(\pi)$ is
\[
\sum_{\tilde{\pi} \in \mathbb{Z}_2^{d^p}} \tilde{\pi} (v_{d^p, \beta_p} \sigma_{d^p}^{m_p} \mathcal{K}^{\sigma}_{d^p}) = v_{d^p, \beta_p} \sigma_{d^p}^{m_p} \star \mathbf{1}_0^{\sigma} \in (L_{e^{\infty}} \cap Z_{e^{\infty}})^{\mathfrak{W}_{e^{\infty}}}.
\]
Similarly, that most sign changes in $\mathfrak{sh}^{\sigma}_{d^{\bullet}, e^{\infty}}$ do not contribute to \eqref{eq:bigSum} can be seen as follows. Write $f_p$ for $v_{d^p, \beta_p} \sigma_{d^p}^{m_p}$ and observe that if $\pm 1 \neq \tilde{\pi} \in \mathbb{Z}_2^{d^p}$, then up to a permutation there is a non-trivial decomposition $d^p = d^{p \prime} + d^{p \prime \prime}$ such that
\[
\tilde{\pi} (f_p\mathcal{K}_{d^p}^{\sigma})=
f_p(x^{\prime}_1, \dots, x^{\prime}_{d^{p \prime}}, -x_1^{\prime}, \dots, -x_{d^{p \prime \prime}}^{\prime \prime}) \mathcal{K}_{d^{p \prime}}^{\sigma} (x^{\prime}) \mathcal{K}_{d^{p \prime \prime}}^{\sigma} (-x^{\prime \prime})  \mathcal{K}_{d^{p \prime}, d^{p \prime \prime}}(x^{\prime}, x^{\prime \prime}),
\]
which is an element of $J_{d^p} Z_{d^{\bullet},e^{\infty}}$. On the other hand, if $\tilde{\pi} = -1$, then the fifth defining condition of $\mathsf{Seq}_e^{\sigma}$ implies that $\tilde{\pi} (f_p \mathcal{K}_{d^p}^{\sigma})=f_p \mathcal{K}_{d^p}^{\sigma}$, leading to the factor $2^l$. The factor $F^{\sigma}_{d^{\bullet}, e^{\infty}}$ is a product of terms of the form $(z_i^{(q)})^2 - (z_j^{(p)})^2$, $1 \leq p < q \leq \infty$, and so is not a zero-divisor. To finish the proof it remains to show that the sum in equation \eqref{eq:bigSum} is not an element of $L_{d^{\bullet}, e^{\infty}}^{\mathfrak{W}_{d^{\bullet}, e^{\infty}}}$. This is the case because $v_{d^p, \beta_p} \sigma_{d^p}^{m_p} \notin J_{d^p} Z_{d^{\bullet}, e^{\infty}}$ and $w_{e^{\infty}, \beta_{\infty}} \notin L_{e^{\infty}} Z_{d^{\bullet}, e^{\infty}}$ by definition and, moreover, the explicit form of $\mathcal{K}_{d^p}^{\sigma}$ implies that $v_{d^p, \beta_p} \sigma_{d^p}^{m_p} \mathcal{K}_{d^p}^{\sigma} \notin J_{d^p} Z_{d^{\bullet}, e^{\infty}}$.

A slight modification is required in type $C$ with $\tau_-=0$ as the kernel $\mathcal{K}^{\sigma}_d$ has a denominator, namely $x_1 \cdots x_d$. Note that $x^{(p)}_i$, $1 \leq p \leq l$, are not zero divisors in $Z_{d^{\bullet}, e^{\infty}} \slash L_{d^{\bullet}, e^{\infty}}$. Define
\[
Z_{d^{\bullet}, e^{\infty}}^{\prime} = Z_{d^{\bullet}, e^{\infty}}[(z_i^{(p)})^{-1} \mid 1 \leq p \leq l]
\]
and put $\mathcal{M}_{d^{\bullet}, e^{\infty}}^{\prime} = Z_{d^{\bullet}, e^{\infty}}^{\prime \mathfrak{W}_{d^{\bullet}, e^{\infty}}}$. Let $\eta : Z_{d^{\bullet}, e^{\infty}} \rightarrow Z_{d^{\bullet}, e^{\infty}}^{\prime}$ be the canonical map and define $L^{\prime}_{d^{\bullet}, e^{\infty}} = \eta(L_{d^{\bullet}, e^{\infty}}) Z^{\prime}_{d^{\bullet}, e^{\infty}}$. We get an inclusion
\[
\overline{\eta} : \mathcal{M}_{d^{\bullet}, e^{\infty}} \slash L_{d^{\bullet}, e^{\infty}}^{\mathfrak{W}_{d^{\bullet}, e^{\infty}}} \hookrightarrow \mathcal{M}^{\prime}_{d^{\bullet}, e^{\infty}} \slash L^{\prime \mathfrak{W}_{d^{\bullet}, e^{\infty}}}_{d^{\bullet}, e^{\infty}}.
\]
The expression \eqref{eq:bigSum} now computes $\overline{\eta} \rho (M_{t_i})$. After collecting the denominators of $\mathcal{K}^{\sigma}_{d^p}$, $1 \leq p \leq l$, in $F_{d^{\bullet}, e^{\infty}}^{\sigma}$ the remainder of the proof can applied without change.
\end{proof}

\begin{Cor}
\label{cor:numericalDTLoops}
The orientifold DT invariants $\Omega_{L_m}^{\sigma}$ are uniquely determined by the orientifold DT series $A_{L_m}^{\sigma}$ and the DT invariants $\Omega_{L_m}$ via the equation
\[
A^{\sigma}_{L_m} = \widetilde{A}_{L_m} \Omega_{L_m}^{\sigma}.
\]
\end{Cor}

\begin{proof}
This follows immediately from Theorem \ref{thm:loopFreeCoHM} and equation \eqref{eq:cohmNumericalFactorization}.
\end{proof}

The invariants $\Omega_{L_m}$ have been computed by Reineke \cite[Theorem 6.8]{reineke2012}. Since $A_{L_m}^{\sigma}$ is given explicitly by equation \eqref {eq:motivicOriDTSeries}, Corollary \ref{cor:numericalDTLoops} gives a way to compute $\Omega_{L_m}^{\sigma}$ without finding a minimal system of generators of $\mathcal{M}_{L_m}$

\begin{Ex}
Let $m=2$ with duality structure $(s, \tau)= (1, -1)$. The ordinary and $\mathbb{Z}_2$-equivariant DT invariants are
\[
\Omega_{L_2}= -q^{-\frac{1}{2}} t + q^{-2} t^2 -q^{-\frac{9}{2}} t^3 + q^{-8}(1+q^2)t^4 + O(t^5).
\]
and
\[
\widetilde{\Omega}^+_{L_2} = -q^{-\frac{9}{2}} \xi^6 + q^{-8}(1+q^2)\xi^8 + O(\xi^{10}) , \qquad
\widetilde{\Omega}^-_{L_2} = -q^{-\frac{1}{2}} \xi^2 + q^{-2} \xi^4 + O(\xi^{10}).
\]
Using Corollary \ref{cor:numericalDTLoops} we compute
\begin{multline*}
\Omega_{L_2}^B= \xi - q^{-\frac{3}{2}}\xi^3 + q^{-5}(1+ q^2)\xi^5 - q^{-\frac{21}{2}}(1 + q^2 + 2q^4+q^6)\xi^7 + \\
q^{-18}(1+ q^2 +2q^4 + 3 q^6 + 4 q^8 + 3 q^{10} + q^{12}) \xi^9 + O(\xi^{11}).
\end{multline*}
Up to $\Lambda_Q^{\sigma,+}$-degree five, minimal generators of $\mathcal{M}_{L_2}^B$ are $\mathbf{1}^{\sigma}_1 ,\mathbf{1}^{\sigma}_3$, $\mathbf{1}^{\sigma}_5$ and $z_1^2 + z_2^2$.
\end{Ex}

\begin{Ex}
Let $m=3$ with duality structure $(s, \tau)=(1,1)$. Then
\begin{multline*}
\Omega_{L_3} = q^{-1} t + q^{-4} t^2 + q^{-9} (1+ q^2 + q^3) t^3 + \\
q^{-16}(1 + q^2 + q^3 +2 q^4 + q^5 + 2 q^6 + q^7 + q^8) t^4 + O(t^5)
\end{multline*}
from which we compute
\begin{multline*}
\Omega_{L_3}^D= \xi^0 + q^{-4}(1+ q^2)\xi^2 + q^{-12}(1+q^2 +2q^4 + 2q^6 + q^8)\xi^4 + \\
q^{-24}(1+q^2 + 2q^4 + 3q^6 + 4q^8 + 5q^{10} + 6 q^{12} + 4q^{16} +q^{18})\xi^6 + \\ q^{-40}(1+q^2 + 2q^4 + 3q^6 + 5q^8 + 6q^{10} + 9q^{12} + 11 q^{14} +  14q^{16} + 16q^{18} \\ + 19q^{20} + 20q^{22} + 21q^{24} + 19q^{26} + 14q^{28} + 6q^{30} + q^{32})\xi^8 + O(\xi^{10}).
\end{multline*}
\end{Ex}

Finally, we give some results for arbitrary $m$ and small dimension vector. Define the quantum integers by $[0]_q=0$ and $
[n]_q = \frac{q^n-1}{q-1}$ if $n \in \mathbb{Z}_{\geq 1}$. Then
\[
\Omega_{L_m}= (-q^{\frac{1}{2}})^{1-m}t + q^{2(1-m)} \left[\lfloor \frac{m}{2} \rfloor \right]_{q^2} t^2 + O(t^3).
\]
Restrict attention to duality structures with $\tau=-1$. Using Corollary \ref{cor:numericalDTLoops} we find
\begin{multline*}
\Omega^C_{L_m}= \xi^0 +  (-q^{\frac{1}{2}})^{3(1-m)} \left[ \lfloor \frac{m}{2} \rfloor \right]_{q^2} \xi^2 + \\
q^{5(1-m)} \left\{ \begin{matrix} 
	\left[ \frac{m}{4} \right]_{q^4} ( \left[ \frac{3m-2}{2} \right]_{q^2} + q^m \left[ \frac{2m-2}{2} \right]_{q^2} + q^{2m} \left[ \frac{m}{2} \right]_{q^2} ) \\
      \left[ \frac{m-1}{4} \right]_{q^4} ( \left[ \frac{3m-1}{2} \right]_{q^2} + q^{m-1} \left[ m \right]_{q^2} + q^{2m-2} \left[ \frac{m-1}{2} \right]_{q^2} ) \\
	\left[ \frac{m}{2} \right]_{q^2} ( \left[ \frac{3m-2}{4} \right]_{q^4} + q^m \left[ \frac{2m}{4} \right]_{q^4} + q^{2m} \left[ \frac{m-2}{4} \right]_{q^4} )  \\
	\left[ \frac{m-1}{2} \right]_{q^2} ( \left[ \frac{3m-1}{4} \right]_{q^4} + q^{m-1} \left[ \frac{2m-2}{4} \right]_{q^4} + q^{2m-2} \left[ \frac{m+1}{4} \right]_{q^4} )
\end{matrix} \right\} \xi^4 + O(\xi^6).
\end{multline*}
The rows of the braces correspond to the congruence class of $m$ modulo four, with $m \equiv 0 \mod 4$ in the top row increasing to $m \equiv 3 \mod 4$ in the bottom row. In the same way, we find
\[
\Omega^D_{L_m}= \xi^0 + q^{3(1-m)} \left[ 2 \lfloor \frac{m}{4} \rfloor +1 \right]_{q^2} \left[\lfloor \frac{m+2}{4} \rfloor \right]_{q^4} \xi^4 + O(\xi^6).
\]
Note that $\mathfrak{M}_{L_m,2}^{D, \mathsf{st}} = \varnothing$, consistent with the vanishing $\Omega_{L_m,2}^D=0$.

\subsection{Symmetric \texorpdfstring{$\widetilde{A}_1$}{} quiver}
\label{sec:symmA1}

Let $Q$ be the following affine Dynkin quiver:
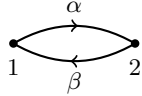
\[
\begin{tikzpicture}[thick,scale=.8,decoration={markings,mark=at position 0.53 with {\arrow{>}}}]
\draw[postaction={decorate}] (0,0) [out=30,in=150]  to  (2,0);
\draw[postaction={decorate}] (2,0) [out=210,in=330]  to  (0,0);

\draw (0,-.4) node {\small $1$};
\draw (2,-.4) node {\small $2$};

\draw (1,0.6) node {\small $\alpha$};
\draw (1,-0.65) node {\small $\beta$};

\fill (0,0) circle (2pt);
\fill (2,0) circle (2pt);
\end{tikzpicture}
\]
A representation of dimension vector $(d_1, d_2)$ consists of a pair of complex matrices $(A,B) \in \mathsf{Mat}_{d_2 \times d_1} \times \mathsf{Mat}_{d_1 \times d_2}$. For $\theta = (1,-1)$ the semistable representations are
\begin{enumerate}[label=(\roman*)]
\item the direct sums of simple representations $S_1^{\oplus k}$, $k \geq 1$, having slope $1$,

\item the direct sums of simple representations $S_2^{\oplus k}$, $k \geq 1$, having slope $-1$, and

\item the pairs $(A,B) \in \mathsf{GL}_d(\mathbb{C}) \times \mathsf{Mat}_{d\times d}$, $d \geq 1$, having slope $0$.
\end{enumerate}
The algebra $\mathcal{H}_Q$ is supercommutative. For $f_1 \in \mathcal{H}_{Q,d^{\prime}}$ and $f_2 \in \mathcal{H}_{Q,d^{\prime \prime}}$ we have
\begin{align*}
f_1 \cdot f_2 = \sum_{\pi \in \mathfrak{sh}_{d^{\prime},d^{\prime \prime}}} \pi \Big( f_1(x^{\prime}_1, \dots,  x^{\prime}_{d^{\prime}_1}, y^{\prime}_1 &, \dots, y^{\prime}_{d^{\prime}_2})  f_2(x^{\prime \prime}_1, \dots, x^{\prime \prime}_{d_1^{\prime \prime}}, y^{\prime \prime}_1, \dots, y^{\prime \prime}_{d_2^{\prime \prime}}) \times  \\ &\frac{\prod_{j=1}^{d_2^{\prime \prime}} \prod_{i=1}^{d^{\prime}_1} (y_j^{\prime \prime} - x^{\prime}_i) \prod_{j=1}^{d_1^{\prime \prime}} \prod_{i=1}^{d^{\prime}_2} (x_j^{\prime \prime} - y^{\prime}_i)}{\prod_{i=1}^{d_1^{\prime \prime}} \prod_{j=1}^{d^{\prime}_1} (x_i^{\prime \prime} - x^{\prime}_j) \prod_{i=1}^{d_2^{\prime \prime}} \prod_{j=1}^{d^{\prime}_2} (y_i^{\prime \prime} - y^{\prime}_j)}  \Big).
\end{align*}
The semistable algebras $\mathcal{H}_{Q, \mu=1}^{\theta \mhyphen \mathsf{ss}}$ and $\mathcal{H}_{Q, \mu=-1}^{\theta \mhyphen \mathsf{ss}}$ are both isomorphic to $\mathcal{H}_{L_0}$ and embed canonically as subalgebras of $\mathcal{H}_Q$. On the other hand, the inclusion
\[
\mathsf{Mat}_{d \times d} \hookrightarrow \mathsf{GL}_d(\mathbb{C}) \times \mathsf{Mat}_{d \times d}, \qquad B \mapsto (\mathbb{I}_{d \times d}, B)
\]
descends to an isomorphism from the stack of $d$-dimensional representations of the one loop quiver to the stack of $(d,d)$-dimensional semistable representations of $Q$. This induces an algebra isomorphism $\mathcal{H}_{Q, \mu=0}^{\theta \mhyphen \mathsf{ss}} \simeq \mathcal{H}_{L_1}$. The map
\[
\Psi_0: \mathcal{H}_{L_1} \rightarrow \mathcal{H}_Q, \qquad x^i \mapsto x^i y^0
\]
extends to an algebra embedding. In \cite[Proposition 2.4]{franzen2015} Franzen proved that the slope ordered CoHA multiplication
\[
\mathcal{H}_{Q, \mu=1}^{\theta \mhyphen \mathsf{ss}} \boxtimes \mathcal{H}_{Q, \mu=0}^{\theta \mhyphen \mathsf{ss}} \boxtimes \mathcal{H}_{Q, \mu=-1}^{\theta \mhyphen \mathsf{ss}}  \rightarrow \mathcal{H}_Q, \qquad a \otimes b \otimes c \mapsto a \Psi_0(b) c
\]
is an algebra isomorphism. In particular,
\[
V_Q^{\mathsf{prim}} = \mathbb{Q} \cdot \mathbf{1}_{(1,0)} \oplus \mathbb{Q} \cdot \mathbf{1}_{(1,1)} \oplus \mathbb{Q} \cdot \mathbf{1}_{(0,1)}.
\]

Let $\sigma$ be the involution of $Q$ that swaps the nodes and fixes the arrows. Then
\[
\mathcal{E}(d_1, d_2) = d_1 d_2 - \frac{d_1(d_1 + s \tau_{\alpha})}{2} - \frac{d_2 (d_2 + s \tau_{\beta})}{2}.
\]
It follows that $Q$ has two inequivalent $\sigma$-symmetric duality structures, say $s = 1$ and $\tau = \pm 1$. The structure maps $(A, B)$ of a self-dual representation are symmetric matrices if $\tau = 1$ and are skew-symmetric matrices if $\tau = -1$. For $f \in \mathcal{H}_{Q,(d_1,d_2)}$ and $g \in \mathcal{M}_{Q,(e,e)}$ we have
\begin{align*}
f  \star g & =  \sum_{\pi \in \mathfrak{sh}_{d,e}^{\sigma}}  \pi \Big( f(x_1, \dots, x_{d_1},  y_1, \dots, y_{d_2})  g(z_1, \dots, z_e) \times  \\ & \frac{ \displaystyle \prod_{1 \leq j \leq_{\tau} l \leq d_1} (-x_j - x_l) \prod_{l=1}^{d_1} \prod_{k=1}^e (-z_k - x_l) \prod_{1 \leq j \leq_{\tau} m \leq d_2} (-y_j - y_m) \prod_{m=1}^{d_2} \prod_{k=1}^e (z_k - y_m)}{\displaystyle \prod_{l=1}^{d_1} \prod_{k=1}^e (z_k - x_l) \prod_{k=1}^e \prod_{m=1}^{d_2} (-y_m -z_k) \prod_{l=1}^{d_1} \prod_{m=1}^{d_2} (-y_m - x_l)} \Big).
\end{align*}

The subvarieties of semistable self-dual representations are
\[
\tau = 1: \;\;\; R_{(e,e)}^{\sigma, \theta \mhyphen \mathsf{ss}} = (\mathsf{Symm}_{e \times e}  \cap \mathsf{GL}_e(\mathbb{C}) ) \times \mathsf{Symm}_{e \times e}
\]
and
\[
\tau = -1: \;\;\; R_{(e,e)}^{\sigma, \theta \mhyphen \mathsf{ss}} = (\mathsf{Skew}_{e \times e} \cap \mathsf{GL}_e(\mathbb{C}) ) \times \mathsf{Skew}_{e \times e},
\]
where $e$ is even if $\tau=-1$. The group $\mathsf{G}_{(e,e)}^{\sigma} = \mathsf{GL}_e(\mathbb{C})$ acts on $R_{(e,e)}^{\sigma, \theta \mhyphen \mathsf{ss}}$ in the canonical way. We see that the stack of semistable self-dual representations of $Q$ is isomorphic to the stack of self-dual representations of $L_1$ with duality structure $(s_{L_1}= \tau, \tau_{L_1} = +1)$. The induced map $\mathcal{M}_Q^{\theta \mhyphen \mathsf{ss}} \xrightarrow[]{\sim} \mathcal{M}_{L_1}$ is an isomorphism over $\mathcal{H}_{Q, \mu =0}^{\theta \mhyphen \mathsf{ss}} \xrightarrow[]{\sim} \mathcal{H}_{L_1}$.

\begin{Lem}
\label{lem:restrKernel}
The kernel of the restriction $\mathcal{M}_Q \rightarrow \mathcal{M}_Q^{\theta \mhyphen \mathsf{ss}}$ in dimension vector $(e,e) \in \Lambda_Q^{\sigma, +}$ is the image of the CoHA action map
\[
\bigoplus_{d=1}^e \mathcal{H}_{Q,(d,0)} \boxtimes \mathcal{M}_{Q,(e-d,e-d)} \xrightarrow[]{\star} \mathcal{M}_{Q,(e,e)}.
\]
\end{Lem}

\begin{proof}
Let $M$ be a self-dual representation determined by matrices $(A,B)$. Then $0 \subset \ker A \subset M$ is the $\sigma$-HN filtration of $M$. The $\sigma$-HN strata of $R_e^{\sigma}$ are therefore the locally closed subsets consisting of self-dual representations with fixed $\dim_{\mathbb{C}} \ker A$ and the closure of a stratum is a union of strata. Using this observation, \cite[Lemma 2.1]{franzen2015} can be applied with only straightforward modifications to prove the lemma. In slightly more detail, the methods of \cite{franzen2015} can be used to prove the lemma for the Chow theoretic Hall module, defined similarly to $\mathcal{M}_Q$ but using equivariant Chow groups instead of equivariant cohomology. In the case at hand the (semistable) cohomological and Chow theoretic Hall modules are isomorphic, as can be verified directly. Hence the lemma also holds in the cohomological case.
\end{proof}

Let $\widetilde{\mathcal{H}}_Q \subset \mathcal{H}_Q$ be the subalgebra generated by
\[
\left(\mathbb{Q} \cdot  \mathbf{1}_{(1,0)} \otimes \mathbb{Q}[u] \right) \oplus \left( \mathbb{Q} \cdot  \mathbf{1}_{(1,1)} \otimes u\mathbb{Q}[u^2] \right) \subset V_Q.
\]
Then $\widetilde{\mathcal{H}}_Q \simeq \mathcal{H}_{Q, \mu=1}^{\theta \mhyphen \mathsf{ss}} \otimes \mathcal{H}_{Q, \mu=0}^{\theta \mhyphen \mathsf{ss}, \mathsf{odd}}$ as algebras, the second factor being the image of $\mathcal{H}_{L_1}^{\mathsf{odd}}$. The map sending $\mathbf{1}_0^{\sigma} \in \mathcal{M}_{L_1,0}$ to $\mathbf{1}_{(0,0)}^{\sigma} \in \mathcal{M}_{Q,(0,0)}$ extends to a $\mathcal{H}_{Q,\mu=0}^{\theta \mhyphen \mathsf{ss}, \mathsf{odd}}$-module embedding $\mathcal{M}_Q^{\theta \mhyphen \mathsf{ss}} \hookrightarrow \mathcal{M}_Q$.

\begin{Thm}
\label{thm:symmA1CoHM}
The module $\mathcal{M}_Q^{\theta \mhyphen \mathsf{ss}}$ is free over $\mathcal{H}_{Q, \mu=0}^{\theta \mhyphen \mathsf{ss}, \mathsf{odd}}$ with basis
\begin{enumerate}
\item $\mathbf{1}_0^{\sigma} \in \mathcal{M}^{\theta \mhyphen \mathsf{ss}}_{Q,(0,0)}$ if $\tau = -1$, and
\item $\mathbf{1}_{(e,e)}^{\sigma} \in \mathcal{M}^{\theta \mhyphen \mathsf{ss}}_{Q,(e,e)}$, $e \geq 0$, if $\tau =1$. 
\end{enumerate}
Moreover, the action map $\mathcal{H}_{Q,\mu=1}^{\theta \mhyphen \mathsf{ss}} \boxtimes \mathcal{M}_Q^{\theta \mhyphen \mathsf{ss}} \xrightarrow[]{\star} \mathcal{M}_Q$ is an isomorphism of $\widetilde{\mathcal{H}}_Q$-modules. In particular, $\mathcal{M}_Q$ is a free $\widetilde{\mathcal{H}}_Q$-module.
\end{Thm}

\begin{proof}
The first statement follows from Proposition \ref{prop:oneLoopCoHM} and the $\mathcal{H}_{L_1}$-module isomorphism $\mathcal{M}_Q^{\theta \mhyphen \mathsf{ss}} \simeq \mathcal{M}_{L_1}$. By direct calculation the restriction $\mathcal{M}_Q \rightarrow \mathcal{M}_Q^{\theta \mhyphen \mathsf{ss}}$ is surjective. From this and Lemma \ref{lem:restrKernel} we conclude that the CoHA action map is surjective. The wall-crossing formula \eqref{eq:oriWallCrossing} for $Q$ reads $A_{Q,\mu= 1}^{\theta \mhyphen \mathsf{ss}}\star A^{\sigma, \theta \mhyphen \mathsf{ss}}_Q =A^{\sigma}_Q$. This implies that the Hilbert-Poincar\'{e} series of $\mathcal{H}_{Q,\mu=1}^{\theta \mhyphen \mathsf{ss}} \boxtimes \mathcal{M}_Q^{\theta \mhyphen \mathsf{ss}}$ and $\mathcal{M}_Q$ are equal. Hence the action map is a $\Lambda_Q^{\sigma,+} \times \mathbb{Z}$-graded vector space isomorphism. That this map respects $\widetilde{\mathcal{H}}_Q$-module structures is clear.
\end{proof}

\begin{Cor}
\label{cor:oriDTAffA1}
The motivic orientifold DT invariants are $\Omega^{\sigma}_Q = 1$ if $\tau = -1$ and $\Omega^{\sigma}_Q =  (1-q^{-\frac{1}{2}} \xi^{(1,1)})^{-1}$ if $\tau = 1$. Conjecture \ref{conj:hodgeEqual} holds for $Q$.
\end{Cor}

\begin{proof}
The statement for $\tau=-1$ follows from Theorem \ref{thm:symmA1CoHM}, so let $\tau =1$. Theorem \ref{thm:symmA1CoHM} implies that $(1-q^{-\frac{1}{2}} \xi^{(1,1)})^{-1}$ is a coefficient-wise upper bound for $\Omega^{\sigma}_Q$. To see that it is also a lower bound observe that since the unshifted cohomological degree of   elements of $\mathcal{H}_{Q,d} \star \mathcal{M}_{Q,e}$ is at least
\[
-2 \mathcal{E}(d) = (d_1-d_2)^2 + d_1 + d_2 > 0,
\]
the element $\mathbf{1}_{(e,e)}^{\sigma}$ is nonzero in $W^{\mathsf{prim}}_{Q,e}$. Hence $\Omega^{\sigma}_Q$ is as stated. To verify Conjecture \ref{conj:hodgeEqual} we must prove that  $PH^{\bullet}(\mathfrak{M}_{(e,e)}^{\sigma, \mathsf{st}}) \simeq \mathbb{Q}(0)$. Note that we use the trivial stability. We have $\mathfrak{M}_{(1,1)}^{\sigma, \mathsf{st}} \simeq \mathbb{C}^{\times}$, consisting of regularly $\sigma$-stable representations. Moreover these are the only regularly $\sigma$-stable self-dual representations, from which it follows that $\mathfrak{M}_{(e,e)}^{\sigma, \mathsf{st}} \simeq \Sym^e \, \mathbb{C}^{\times} \backslash \Delta$. Consider the open inclusions
\[
\Sym^e \, \mathbb{C}^{\times} \backslash \Delta \overset{i}{\hookrightarrow} \Sym^e \, \mathbb{C} \backslash \Delta \hookrightarrow \Sym^e \, \mathbb{P}^1.
\]
Since $\Sym^e \, \mathbb{P}^1$ is a smooth compactification of both $\Sym^e \, \mathbb{C}^{\times}$ and $\Sym^e \, \mathbb{C}$, we obtain a commutative diagram
\[
\begin{tikzpicture}[description/.style={fill=white,inner sep=2pt},baseline=(current  bounding  box.center)]
    \matrix (m) [matrix of math nodes, row sep=2.8em,
    column sep=2.5em, text height=1.9ex, text depth=0.25ex]
    { 
H^{\bullet}(\Sym^e \, \mathbb{P}^1) & H^{\bullet}(\Sym^e \, \mathbb{C} \backslash \Delta) & H^{\bullet}(\Sym^e \, \mathbb{C}^{\times} \backslash \Delta)  \\
& PH^{\bullet}(\Sym^e \, \mathbb{C} \backslash \Delta) & PH^{\bullet}(\Sym^e \, \mathbb{C}^{\times} \backslash \Delta)  \\ };
\draw[->] (m-1-1) edge (m-1-2);
\draw[->] (m-1-2) edge (m-1-3);
\draw[right hook-latex] (m-2-3) edge (m-1-3);
\draw[right hook-latex] (m-2-2) edge (m-1-2);
\draw[->>] (m-1-1) edge (m-2-2);
\draw[->] (m-2-2) edge node[above]{$i^*$} (m-2-3);
\draw[->>] (m-1-1) edge [in=200, out=280] (m-2-3);
\end{tikzpicture}
\]
the surjections following from \cite[Proposition 6.29]{peters2008}. Hence $i^*$ is also surjective. Since $PH^{\bullet}(\Sym^e \, \mathbb{C} \backslash \Delta) \simeq \mathbb{Q}(0)$ we also have $PH^{\bullet}(\Sym^e \, \mathbb{C}^{\times} \backslash \Delta) \simeq \mathbb{Q}(0)$.
\end{proof}

\begin{Rems}
\leavevmode
\begin{enumerate}
\item The isomorphism $\mathcal{H}_{Q,\mu=1}^{\theta \mhyphen \mathsf{ss}} \boxtimes \mathcal{M}_Q^{\theta \mhyphen \mathsf{ss}}  \xrightarrow[]{\sim} \mathcal{M}_Q$ is an instance of the PBW factorization \eqref{eq:moduleFactorization}.

\item Let $\theta=(1,-1)$. If $\tau=1$, then $\mathfrak{M}_e^{\sigma, \theta \mhyphen \mathsf{st}} \simeq \Sym^e \, \mathbb{C}\backslash \Delta$ and the proof of Corollary \ref{cor:oriDTAffA1} shows that $i^*: PH^{\bullet}(\mathfrak{M}_e^{\sigma, \theta \mhyphen \mathsf{st}}) \xrightarrow[]{\sim} P H^{\bullet}(\mathfrak{M}_e^{\sigma, \mathsf{st}})$. This is an example of the lack of wall-crossing for $\sigma$-symmetric quivers.
\end{enumerate}
\end{Rems}

\section{Cohomological Hall modules of finite type quivers}
\label{sec:finiteTypeQuivers}

A quiver is called finite type if it has only finitely many indecomposable representations up to isomorphism. Gabriel \cite{gabriel1972} proved that a quiver is finite type if and only if it is a disjoint union of quivers whose underlying graphs are Dynkin diagrams of $ADE$ type. Note that a finite type quiver is symmetric if and only if it is a finite collection of points. The only connected finite type quivers with involution are of type $A$. All other finite type quivers with involution are disjoint unions of these and quivers of the form $ADE^{\sqcup}$, in the notation of Section \ref{sec:disjointUnions}. By Theorem \ref{thm:disjointCoHM} the CoHM of a $ADE^{\sqcup}$ quiver can be described entirely in terms of the CoHA of the corresponding $ADE$ quiver, whose structure we recall in Section \ref{sec:finiteTypeCoHA}. The main task is therefore to describe the CoHM of a type $A$ Dynkin quiver.

\subsection{Finite type CoHA}
\label{sec:finiteTypeCoHA}

Let $Q$ be a connected finite type quiver. We assume that $Q$ is not of type $E_8$; for this case see \cite[Remark 11.3]{rimanyi2013}. The sets $\Pi$ of positive simple roots and $\Delta$ of positive roots of $Q$ are in bijection with the sets of isomorphism classes of simple and indecomposable representations of $Q$, respectively \cite{gabriel1972}. Identify $\Delta$ with a subset of $\Lambda_Q^+$ using the dimension vector map and write $I_{\beta}$ for the indecomposable representation of dimension vector $\beta \in \Delta$. Fix a total order $\beta_1 < \dots < \beta_N$ on $\Delta$ such that $\Hom(I_{\beta_i}, I_{\beta_j}) = 0 = \Ext^1(I_{\beta_j}, I_{\beta_i})$ if $i <j $. Such an order exists because the Auslander-Reiten quiver $\Gamma_Q$ of $Q$ is acyclic.

Fix a positive root $\beta \in \Delta$. Consider
\[
\mathcal{H}_Q^{\langle \beta \rangle} = \bigoplus_{n \geq 0} H^{\bullet}_{\mathsf{GL}_{n \beta}}(R_{n\beta}) \{  \chi( n \beta, n \beta) / 2 \}
\]
and
\[
\mathcal{H}_Q^{\langle \beta \rangle, \simeq} = \bigoplus_{n \geq 0} H^{\bullet}_{\mathsf{GL}_{n \beta}} (\eta_{I_{\beta}^{\oplus n}})\{  \chi( n\beta, n\beta) / 2 \}
\]
where $\eta_{I_{\beta}^{\oplus n}} \subset R_{n\beta}$ is the $\mathsf{GL}_{n\beta}$-orbit of representations which are isomorphic to $I_{\beta}^{\oplus n}$. Then $\mathcal{H}_Q^{\langle \beta \rangle}$ is a subalgebra of $\mathcal{H}_Q$ and the natural associative Hall product on $\mathcal{H}_Q^{\langle \beta \rangle, \simeq}$ is such that the restriction $\rho: \mathcal{H}_Q^{\langle \beta \rangle} \rightarrow \mathcal{H}_Q^{\langle \beta \rangle, \simeq}$ a surjective algebra homomorphism. Moreover, $\mathcal{H}_Q^{\langle \beta \rangle, \simeq} \simeq \mathcal{H}_{L_0}$ as algebras. Let $\{ \tilde{x}^j \}_{j \geq 0}$ be the corresponding algebra generators of $\mathcal{H}_Q^{\langle \beta \rangle, \simeq}$, as defined in Section \ref{sec:zeroLoopQuiver}. Choose $i(\beta) \in Q_0$ such that $\dim_{\mathbb{C}} (I_{\beta})_{i(\beta)} = 1$; such a choice cannot be made in type $E_8$. Define a section $\psi$ of $\rho$ by $\psi(\tilde{x}^j) = x_{i(\beta)}^j$. Write $\mathcal{H}_Q^{(\beta)} \subset \mathcal{H}_Q$ for the isomorphic image of $\psi$. Elements of $\mathcal{H}_Q^{(\beta)}$ depend only on the variables associated to the node $i(\beta)$.

The following result is due to Rim\'{a}nyi. It was stated by Kontsevich and Soibelman for $Q$ of type $A_2$ as  \cite[Proposition 2.1]{kontsevich2011}.

\begin{Thm}[{\cite[Theorem 11.2]{rimanyi2013}}]
\label{thm:finiteTypeCoHA}
The ordered CoHA multiplication maps
\[
\overset{\longleftarrow}{\boxtimes}_{\alpha \in \Pi}^{\mathsf{tw}} \,\mathcal{H}_Q^{(\alpha)} \rightarrow \mathcal{H}_Q, \qquad \qquad\overset{\longrightarrow}{\boxtimes}_{\beta \in \Delta}^{\mathsf{tw}} \,\mathcal{H}_Q^{(\beta)} \rightarrow \mathcal{H}_Q
\]
are isomorphisms in $D^{lb}(\mathsf{Vect}_{\mathbb{Z}})_{\Lambda^+_Q}$.
\end{Thm}

\subsection{Preliminary results for the self-dual case}
\label{sec:rimMod}

Let $(Q,\sigma)$ be of Dynkin type $A$. Then $Q$ has two inequivalent duality structures which, for concreteness, we take to be $\tau = -1$ and $s= 1$ or $\tau=-1$ and $s = -1$, giving orthogonal or symplectic representations in the language of \cite{derksen2002}, respectively. In type $A_{2n}$ (respectively, $A_{2n+1}$) all orthogonal (symplectic) representations are hyperbolic. In the remaining two cases, henceforth referred to as non-hyperbolic, each $\sigma$-invariant positive root admits a unique self-dual structure.

To describe $\mathcal{M}_Q$ we will modify Rim\'{a}nyi's approach to the study of $\mathcal{H}_Q$. Fix $d^{\bullet}=(d^1, \dots, d^r) \in (\Lambda_Q^+)^r$, $e^{\infty} \in \Lambda_Q^{\sigma, +}$ and put $e=\sum_{i=1}^r H(d^i) + e^{\infty}$. Let $\mathsf{G}_{d^{\bullet}, e^{\infty}}^{\sigma} \subset \mathsf{G}_e^{\sigma}$ be the stabilizer of a $Q_0$-graded isotropic flag of $\mathbb{C}^e$ of the form
\[
0 = U_0 \subset U_1 \subset \cdots \subset U_r \subset \mathbb{C}^e
\]
with
\[
\mathbf{dim}\, U_k \slash U_{k-1} = d^k, \qquad \mathbf{dim}\, \mathbb{C}^e \git U_r = e^{\infty}.
\]
Let $\mathsf{Fl}^{\sigma}_{d^{\bullet}, e^{\infty}} \simeq  \mathsf{G}_e^{\sigma} \slash \mathsf{G}_{d^{\bullet}, e^{\infty}}^{\sigma}$ be the corresponding isotropic flag variety.  Each flag $U_{\bullet}$ can be extended to a flag of length $2r+1$ by setting $U_{2r-k+1}=U_k^{\perp}$ for $k=0, \dots, r$. 

For $k=1, \dots, 2r+1$, let $\mathcal{V}_{i, k}$ be the tautological vector bundle over $\mathsf{Fl}^{\sigma}_{d^{\bullet}, e^{\infty}}$ parameterizing the $k$th subspace of $\mathbb{C}^e$ at the node $i$. The quotient bundle $\mathcal{F}_{i,k} = \mathcal{V}_{i,k} \slash \mathcal{V}_{i,k-1}$ has rank $d_i^k$. The self-dual structure on $\mathbb{C}^e$ induces isomorphisms $\mathcal{F}_{i,k} \simeq \mathcal{F}_{\sigma(i), 2r+1-k}^{\vee}$. By duality this gives a chain of vector bundle isomorphisms
\[
\Hom(\mathcal{F}_{i,k}, \mathcal{F}_{j,l}) \simeq \Hom(\mathcal{F}_{j,l}^{\vee}, \mathcal{F}_{i,k}^{\vee}) \simeq \Hom(\mathcal{F}_{\sigma(j), 2r+1-l}, \mathcal{F}_{\sigma(i), 2r+1 - k})
\]
which induces a linear $\mathbb{Z}_2$-action on
\[
\mathcal{G}=\bigoplus_{i \xrightarrow[]{\alpha} j \in Q_1} \bigoplus_{1 \leq k < l \leq 2r+1} \Hom(\mathcal{F}_{i,k}, \mathcal{F}_{j,l}) .
\]
Denote by $\mathcal{G}^{\sigma}$ the subbundle of anti-fixed points.

The following result is motivated by \cite[Lemmas 8.1 and 8.2]{rimanyi2013}.

\begin{Lem}
\label{lem:geomMult}
Let $f_k \in \mathcal{H}_{Q,d^k}$, $k=1, \dots, r$, and $g \in \mathcal{M}_{Q,e^{\infty}}$. Then
\[
(f_1 \cdots  f_r) \star g = \pi_*^{\sigma} \left[ \left( \prod_{k=1}^r f_k (\mathcal{F}_{\bullet, k}) \right) g( \mathcal{F}_{\bullet, 0} ) \mathsf{Eu}_{\mathsf{G}_e^{\sigma}}(\mathcal{G}^{\sigma}) \right]
\]
where $\pi^{\sigma}: \mathsf{Fl}^{\sigma}_{d^{\bullet}, e^{\infty}} \rightarrow \Spec(\mathbb{C})$ is the structure map and $\mathsf{Eu}_{\mathsf{G}_e^{\sigma}}(\mathcal{G}^{\sigma})$ is the $\mathsf{G}_e^{\sigma}$-equivariant Euler class of $\mathcal{G}^{\sigma} \rightarrow \mathsf{Fl}^{\sigma}_{d^{\bullet}, e^{\infty}}$.
\end{Lem}

\begin{proof}
Like the left-hand side, the right-hand side of the claimed equality can be computed by equivariant localization with respect to $\mathsf{T}_e \subset \mathsf{G}_e^{\sigma}$. The $\mathsf{T}_e$-fixed points of $\mathsf{Fl}^{\sigma}_{d^{\bullet}, e^{\infty}}$ are those appearing in the proof of (the $r$-fold iteration of) Theorem \ref{thm:cohmLoc}. Since the weights of $\mathsf{Eu}_{\mathsf{G}_e^{\sigma}}(\mathcal{G}^{\sigma})$ and $\mathsf{Eu}_{\mathsf{G}_{d^{\bullet},e^{\infty}}}(N_{R^{\sigma}_{e} \slash R_{d^{\bullet},e^{\infty}}^{\sigma}})$ at a $\mathsf{T}_e$-fixed point agree, the lemma follows.
\end{proof}

Define a $\mathsf{G}_e^{\sigma}$-stable closed subvariety of $\mathsf{Fl}^{\sigma}_{d^{\bullet}, e^{\infty}} \times R_e^{\sigma}$ by
\[
\Sigma^{\sigma} = \{(U_{\bullet} , m) \in \mathsf{Fl}_{d^{\bullet}, e^{\infty}}^{\sigma} \times R_e^{\sigma} \mid m_{\alpha} (U_{i, k}) \subset U_{j, k}, \;\;\; \forall \, i \xrightarrow[]{\alpha} j \in Q_1, \;\; k=1, \dots, r \}.
\]
It has a $\mathsf{G}_e^{\sigma}$-equivariant fundamental class
\[
[\Sigma^{\sigma}] \in H^{\bullet}_{\mathsf{G}_e^{\sigma}}(\mathsf{Fl}^{\sigma}_{d^{\bullet}, e^{\infty}} \times R_e^{\sigma}) \simeq H^{\bullet}_{\mathsf{G}_e^{\sigma}}(\mathsf{Fl}^{\sigma}_{d^{\bullet}, e^{\infty}}).
\]

\begin{Lem}
\label{lem:eulerIden}
The equality $\mathsf{Eu}_{\mathsf{G}_e^{\sigma}}(\mathcal{G}^{\sigma}) = [\Sigma^{\sigma}]$ holds in $H^{\bullet}_{\mathsf{G}_e^{\sigma}}(\mathsf{Fl}^{\sigma}_{d^{\bullet}, e^{\infty}})$.
\end{Lem}
\begin{proof}
This can be proved in the same way as \cite[Lemma 8.3]{rimanyi2013}.
\end{proof}

The duality structure on $\mathsf{Rep}_{\mathbb{C}} (Q)$ defines an involution of the Auslander-Reiten quiver $\Gamma_Q$, sending an indecomposable  representation $I$ to $S(I)$. This involution preserves the levels of $\Gamma_Q$ which, being in type $A$, are exactly the orbits of the Auslander-Reiten translation. Fix a partition $\Delta = \Delta^- \sqcup \Delta^{\sigma} \sqcup \Delta^+$ such that $\Delta^{\sigma}$ is fixed pointwise by $S$ and $S(\Delta^-) = \Delta^+$. Without loss of generality, we can assume that $\beta_u < S(\beta_u)$ for all $\beta_u \in \Delta^-$.
 Write $\Delta^- = \{ \beta_{u_1} < \dots < \beta_{u_r} \}$.

\begin{Lem}
\label{lem:sdFilt}
Every self-dual representation $M$ has a unique isotropic filtration
\[
0 = U_0 \subset U_1 \subset \cdots \subset U_{r} \subset M
\]
such that $U_j \slash U_{j-1} \simeq I_{\beta_{u_j}}^{ \oplus m_{u_j}}$, $j = 1, \dots, r$, and $M \git U_{r} \simeq \bigoplus_{\beta_u \in \Delta^{\sigma}} I_{\beta_u}^{\oplus m_u}$.
\end{Lem}

\begin{proof}
Any self-dual representation can be written uniquely as an orthogonal direct sum of indecomposable self-dual representations. Explicitly, we have
\begin{equation}
\label{eq:sdKrullSchmidt}
M = \bigoplus_{l=1}^r H(I_{\beta_{u_l}})^{\oplus m_{u_l}} \oplus \bigoplus_{\beta_u \in \Delta^{\sigma}} I_{\beta_u}^{\oplus m_u}
\end{equation}
for some $m_u \in \mathbb{Z}_{\geq 0}$. Setting $U_j =  \bigoplus_{l=1}^j I_{\beta_{u_l}}^{\oplus m_{u_l}}$ gives a filtration with the desired properties.

Suppose that $U_{\bullet}^{\prime} \subset M$ is another filtration with the stated properties. The ordering assumption $\beta_{u_1} < \cdots < \beta_{u_r}$ implies that $U_{\bullet}^{\prime} = U_{\bullet}$. So it suffices to show that there is a unique isotropic embedding $U_r \hookrightarrow M$. To do so, first note that $\Hom(I_{\beta}, I_{\beta^{\prime}})=0$ for all $\beta \in \Delta^-$ and $\beta^{\prime} \in \Delta^{\sigma}$. Indeed, if $\Hom(I_{\beta}, I_{\beta^{\prime}}) \neq 0$, then $\Hom(I_{\beta^{\prime}}, S(I_{\beta}) ) \neq 0$. Hence $\beta > \beta^{\prime}$ and $\beta^{\prime} > S(\beta)$ so that $\beta > S(\beta)$, a contradiction. It follows that the summand $U_1 \subset U_r$ maps isomorphically onto $I_{\beta_1}^{\oplus m_1}$. While $U_2 \subset U_r$ could potentially map non-trivially to $S(I_{\beta_{u_1}})$, this would contradict the condition that $U_2$ be isotropic. Hence $U_2$ maps isomorphically onto $I_{\beta_{u_1}}^{\oplus m_{u_1}} \oplus I_{\beta_{u_2}}^{\oplus m_{u_2}}$. Continuing in this way we see that $U_r \hookrightarrow M$ is the canonical isotropic embedding.
\end{proof}

We derive two results using Lemma \ref{lem:sdFilt}. The first is a self-dual extension of a theorem of Reineke \cite[Theorem 2.2]{reineke2003b} and appears in the unpublished thesis of Lovett \cite{lovett2003}. For $M \in R_e^{\sigma}$ let $\eta^{\sigma}_M \subset R_e^{\sigma}$ be the $\mathsf{G}_e^{\sigma}$-orbit $M$ and let $\overline{\eta}_M^{\sigma}\subset R_e^{\sigma}$ be the closure of $\eta^{\sigma}_M$. We call elements of $\overline{\eta}_M^{\sigma}$ self-dual degenerations of $M$.

\begin{Thm}[\cite{lovett2003}]
\label{thm:sdResolution}
Let $M$ be a self-dual representation. Keeping the notation of Lemma \ref{lem:sdFilt}, set $d^j = m_j \beta_j$, $j =1, \dots, r$, and $e^{\infty} = \mathbf{dim}\, M \git U_{r}$. Then the canonical morphism $\pi_M^{\sigma}: \Sigma^{\sigma} \rightarrow R_e^{\sigma}$ is a $\mathsf{G}_e^{\sigma}$-equivariant resolution of $\overline{\eta}^{\sigma}_M$.
\end{Thm}

\begin{proof}
This is proved for $Q$ of type $A_3$ in \cite[Proposition 2.3]{lovett2005}. We will prove the general case using a self-dual version of Reineke's argument.

The variety $\Sigma^{\sigma}$ is smooth, being the total space of a vector bundle over $\mathsf{Fl}^{\sigma}_{d^{\bullet}, e^{\infty}}$, and the morphism $\pi_M^{\sigma}$ is proper and equivariant. We prove that $\pi_M^{\sigma} (\Sigma^{\sigma})= \overline{\eta}_M^{\sigma}$. If $N \in \pi_M^{\sigma} (\Sigma^{\sigma})$, then there exsists an isotropic filtration
\[
0 = V_0 \subset V_1 \subset \cdots \subset V_{r} \subset N, \qquad \mathbf{dim} \, V_j \slash  V_{j-1} = d^j.
\]
Since $\Ext^1(I_{\beta},I_{\beta}) =0$ for all $\beta \in \Delta$, Voigt's lemma implies that $V_j \slash  V_{j-1}$ is a degeneration of $I_{\beta_{u_j}}^{\oplus m_{u_j}}$. Similarly, $\Ext^1(I_\beta, I_{\beta^{\prime}})=0$ for all $\beta, \beta^{\prime} \in \Delta^{\sigma}$ and $N \git V_r$ is a degeneration of $\oplus_{\beta_u \in \Delta^{\sigma}} I_{\beta_u}^{\oplus m_u}$. Applying \cite[Lemma 2.3]{reineke2003b} we conclude that $N$ is a degeneration of $M$. It is proved in \cite[Theorem 2.6]{derksen2002} that two self-dual representations are isometric if and only if they are isomorphic. Using this we see that $N$ is in fact a self-dual degeneration of $M$. Hence $\eta_M^{\sigma} \subset \pi_M^{\sigma} (\Sigma^{\sigma}) \subset  \overline{\eta}_M^{\sigma}$, implying $\pi_M^{\sigma} (\Sigma^{\sigma})= \overline{\eta}_M^{\sigma}$.

It remains to show that $\pi_M^{\sigma}$ restricts to a bijection over $\eta_M^{\sigma}$. Consider an arbitrary isotropic filtration
\[
0 = U_0 \subset U_1 \subset \cdots \subset U_{r} \subset M, \qquad \mathbf{dim} \, U_j \slash  U_{j-1} = d^j.
\]
Arguing as above, $U_j \slash  U_{j-1}$ and $M \git U_r$ are degenerations of $I_{\beta_{u_j}}^{\oplus m_{u_j}}$ and $\oplus_{\beta_u \in \Delta^{\sigma}} I_{\beta_u}^{\oplus m_u}$, respectively. Since $\Hom(I_{\beta_i}, I_{\beta_j}) = 0$ if $i <j$, we can apply \cite[Lemma 2.3]{reineke2003b} to conclude that $U_j \slash  U_{j-1} \simeq I_{\beta_{u_j}}^{\oplus m_{u_j}}$ and $M \git U_r \simeq\oplus_{\beta_u \in \Delta^{\sigma}} I_{\beta_u}^{\oplus m_u}$. Lemma \ref{lem:sdFilt} now implies that $U_{\bullet} \subset M$ is the canonical filtration.
\end{proof}

We can now prove an analogue of \cite[Theorem 10.1]{rimanyi2013}.

\begin{Cor}
\label{cor:fundClassStructConst}
Keeping the above notation, the equality
\[
[\overline{\eta}^{\sigma}_M]= (\mathbf{1}_{m_{u_1} \beta_{u_1}} \cdots \mathbf{1}_{m_{u_r} \beta_{u_r}} ) \star \mathbf{1}^{\sigma}_{\sum_{\beta u \in \Delta^{\sigma}} m_u \beta_u}
\]
holds in $\mathcal{M}_Q$.
\end{Cor}

\begin{proof}
Theorem \ref{thm:sdResolution} implies that $\pi_*^{\sigma} [\Sigma^{\sigma}]= [\overline{\eta}_M^{\sigma}]$. The desired equality then follows from Lemmas \ref{lem:geomMult} and \ref{lem:eulerIden}.
\end{proof}

The class $[\overline{\eta}_M^{\sigma}] \in H^{\bullet}_{\mathsf{G}_e^{\sigma}}(R_e^{\sigma})$ is the Thom polynomial of the orbit $\eta_M^{\sigma} \subset R_e^{\sigma}$. These classes play the role of quiver polynomials \cite{buch1999} in the self-dual setting.

\begin{Ex}
Let $Q$ be the $A_2$ quiver $\begin{tikzpicture}[thick,scale=.33,decoration={markings,mark=at position 0.6 with {\arrow{>}}}]
\draw[postaction={decorate}] (0,0) to  (4,0);
\fill (0,0) circle (4pt);
\fill (4,0) circle (4pt);
\draw (0,+.6) node {\tiny $1$};
\draw (4,+.6) node {\tiny $2$};
\end{tikzpicture}$ and set $Q_0^+ = \{ 1 \}$. If $f \in \mathcal{H}_{Q,(d_1, d_2)}$ and $g \in \mathcal{M}_{Q,(e,e)}$, then
\begin{align*}
f \star g = \sum_{\pi \in \mathfrak{sh}_{d_1, e, d_2}}   \pi \cdot \Big(   f( x_1, & \dots, x_{d_1}, y_1, \dots, y_{d_2}) g( z_1, \dots, z_e ) \times  \\ &
\frac{  \prod_{1 \leq i \leq_{-s} j \leq d_1} (-x_i - x_j) }{\prod_{l=1}^{d_2} \prod_{i=1}^{d_1} (-y_l - x_i) \prod_{m=1}^{d_2} \prod_{k=1}^e (-y_l - z_k)} \Big).
\end{align*}
For orthogonal representations this gives
\[
\mathbf{1}_{(d,0)} \star \mathbf{1}_{(e,e)}^{\sigma} = \sum_{\pi \in \mathfrak{sh}_{d,e}} \pi \cdot \Big( \prod_{1 \leq i < j \leq d} (- x_i - x_j) \Big) = s_{(d-1, \dots, 1,0,\mathbf{0}^{e})}(-z)
\]
while for symplectic representations
\[
\mathbf{1}_{(d,0)} \star \mathbf{1}_{(e,e)}^{\sigma} = \sum_{\pi \in \mathfrak{sh}_{d,e}} \pi \cdot \Big( \prod_{1 \leq i \leq j \leq d} (- x_i - x_j) \Big) = 2^d s_{(d, \dots, 2,1,\mathbf{0}^e)}(-z).
\]
Corollary \ref{cor:fundClassStructConst} implies that $\mathbf{1}_{(d,0)} \star \mathbf{1}_{(e,e)}^{\sigma}$ is the Thom polynomial of the orbit of matrices having rank $e$ in the $\mathsf{GL}_{d+e}(\mathbb{C})$ representation $\bigwedge\nolimits^2 \mathbb{C}^{d+e}$ or $\Sym^2 \mathbb{C}^{d+e}$, respectively. These Thom polynomials were computed using different methods in \cite{jozefiak1981}, \cite{harris1984}, \cite{feher2004}.
\end{Ex}

Turning to the second application of Lemma \ref{lem:sdFilt}, define putative orientifold DT invariants $\Omega_{Q,e}^{\sigma, \mathsf{indec}}$ to be one if $e \in \Lambda_Q^{\sigma,+}$ is a sum of pairwise distinct positive roots, each of which is the dimension vector of an indecomposable representation which admits a self-dual structure. Otherwise, set $\Omega_{Q,e}^{\sigma, \mathsf{indec}} =0$. Define $\Omega_{Q,e}^{\sigma, \mathsf{simp}}$ similarly.  Put $\Omega_{Q,0}^{\sigma} =1$ in both cases. Set also $\Pi^+ =\Pi \cap \Delta^+$ and $\Pi^{\sigma} =\Pi \cap \Delta^{\sigma}$. Let $h=0$ in the hyperbolic case and $h=1$ otherwise.

Recall that $A_{L_0}(q^{\frac{1}{2}}, t) = (q^{\frac{1}{2}} t; q)_{\infty} = \mathbb{E}_q(t)$ is the quantum dilogarithm.

\begin{Thm}
\label{thm:oriDilog}
The identity
\begin{align*}
\overset{\longleftarrow}{\prod_{\alpha \in \Pi^+}} & \mathbb{E}_q (t^{\alpha})   \star  \sum_{\varnothing \subseteq \pi \subseteq \Pi^{\sigma}} \prod_{\alpha \in \pi} \mathbb{E}_{q^2}(q^{-\frac{1}{2} +h} t^{\alpha}) \star \Omega_{Q,\pi}^{\sigma, \mathsf{simp}} \xi^{\pi} =  
\\ &
\overset{\longrightarrow}{\prod_{\beta \in \Delta^-}} \mathbb{E}_q  (t^{\beta}) \star 
\sum_{ \varnothing \subseteq \pi \subseteq \Delta^{\sigma}} \Big(\prod_{\beta \in \pi} \mathbb{E}_{q^2}(q^{-\frac{1}{2} +h } t^{\beta}) \cdot\prod_{\beta \not \in \pi} \mathbb{E}_{q^2}(q^{-\frac{1}{2}} t^{\beta}) \Big) \star \Omega_{Q,\pi}^{\sigma, \mathsf{indec}} \xi^{\pi}
\end{align*}
holds in $\hat{\mathbb{S}}_Q$. Here we have written $\xi^{\pi}$ for $\xi^{\sum_{\beta \in \pi} \beta}
$ and similarly for $\Omega^{\sigma,\cdots}_{Q,\pi}$.
\end{Thm}

\begin{proof}
It is straightforward to construct a $\sigma$-compatible stability $\theta_{\mathsf{simp}}$ whose stable representations are the simple representations and whose order by increasing slope agrees with $<$. Existence and uniqueness of $\sigma$-Harder-Narasimhan filtrations implies a factorization of the identity characteristic function in the finite field Hall module of $Q$. Applying the Hall module integration map \cite[Theorem 4.1]{mbyoung2015} to this factorization gives the left-hand side of the desired equality. Lemma \ref{lem:sdFilt} leads to a second factorization of identity characteristic function, the integral of which gives the right-hand side.
\end{proof}

Alternatively, Theorem \ref{thm:oriDilog} can be proved using Kazarian spectral sequences as in \cite[\S 6]{rimanyi2013}. The new ingredient is a self-dual version of Voigt's lemma, stating that the codimension of $\eta_M^{\sigma} \subset R_{\Dim M}^{\sigma}$ is $\dim_{\mathbb{C}} \Ext^1(M,M)^S$. This can be verified using the chain level description of $\Ext^1(M,M)^S$ given in \cite[Proposition 3.3]{mbyoung2016}.

\subsection{CoHM of type \texorpdfstring{$A$}{} quivers}

We begin with a calculation in rank two.

\begin{Ex}
Consider orthogonal representations of the $A_2$ quiver. Then $x^i \star \mathbf{1}_0^{\sigma} = z^i$ and $\{ x^i \star \mathbf{1}_0^{\sigma} \}_{i \geq 0}$ spans $\mathcal{M}_{Q,(1,1)}$. Let $\beta_2$ be the non-simple indecomposable representation and let $\nu_i = y^i \in \mathcal{H}_{Q,(1,1)}$ be a  generator of $\mathcal{H}_Q^{(\beta_2)}$. We compute
\[
(x^i \cdot x^j) \star \mathbf{1}^{\sigma}_0 = -(z_1+z_2) \frac{z_1^i z_2^j - z_1^j z_2^i}{z_1 - z_2}, \qquad \nu_i \star \mathbf{1}^{\sigma}_0 = (-1)^i \frac{z^i_1 - z_2^i}{z_1 - z_2}
\]
so that $\{(x^i \cdot x^j) \star \mathbf{1}^{\sigma}_0 \}_{i > j}$ spans $(z_1+z_2) \mathbb{Q}[z_1, z_2]^{\mathfrak{S}_2}$. To generate the remainder of $\mathcal{M}_{Q,(2,2)} \simeq \mathbb{Q}[z_1, z_2]^{\mathfrak{S}_2}$ it suffices to include $\{ \nu_{2i+1} \star \mathbf{1}^{\sigma}_0\}_{i \geq 0}$. In three variables
\[
(x^i \cdot x^j \cdot x^k) \star \mathbf{1}^{\sigma}_0 = - (z_1 +z_2) (z_1 + z_3) (z_2 + z_3) s_{(i,j,k) - \delta_3}
\]
which freely generate $(z_1 +z_2) (z_1 + z_3) (z_2 + z_3) \mathbb{Q}[z_1, z_2, z_3]^{\mathfrak{S}_3}$. We also have
\begin{multline*}
(x^i \cdot \nu_j) \star \mathbf{1}^{\sigma}_0 =  \frac{(-1)^j}{(z_1 - z_2)(z_1-z_3)(z_2-z_3)}
\left[ z_1^i (z_2^j - z_3^j)(z_1+z_2)(z_1+z_3) - \right. \\ \left. z_2^i (z_1^j - z_3^j)(z_1+z_2)(z_3+z_2) + z_3^i (z_1^j - z_1^j)(z_1+z_3)(z_2+z_3) \right].
\end{multline*}
Writing $\beta_1$ for the simple root associated to $1 \in Q_0$, we conclude that
\[
\mathcal{H}_Q^{(\beta_1)} \boxtimes^{\mathsf{tw}}\mathcal{H}_Q^{ (\beta_2), \mathsf{odd}} \boxtimes^{S \mhyphen \mathsf{tw}} \mathbf{1}^{\sigma}_0  \xrightarrow[]{\star} \mathcal{M}_Q
\]
is a $D^{lb}(\mathsf{Vect}_{\mathbb{Z}})_{\Lambda^{\sigma,+}_Q}$-isomorphism up to $\Lambda_Q^{\sigma, +}$-degree $(3,3)$.
\end{Ex}

These calculations can be generalized as follows. For each $\beta \in \Delta^{\sigma}$ let
\[
\mathcal{M}_Q^{\langle \beta \rangle} = \bigoplus_{n \geq 0} H^{\bullet}_{\mathsf{G}^{\sigma}_{n \beta}}(R^{\sigma}_{n\beta})\{ \mathcal{E}(n \beta) / 2 \}
\]
and
\[
\mathcal{M}_Q^{\langle \beta \rangle, \simeq} = \bigoplus_{n \geq 0} H^{\bullet}_{\mathsf{G}^{\sigma}_{n \beta}} (\eta^{\sigma}_{I_{\beta}^{\oplus n}}) \{ \mathcal{E}(n \beta) / 2 \}
\]
considered as modules over $\mathcal{H}_Q^{\langle \beta \rangle}$ and $\mathcal{H}_Q^{\langle \beta \rangle, \simeq}$, respectively. If $I_{\beta}$ does not admit a self-dual structure, then $n$ must be even. We have $\mathcal{M}_Q^{\langle \beta \rangle, \simeq} \simeq \mathcal{M}_{L_0}$ compatibly with $\mathcal{H}_Q^{\langle \beta \rangle, \simeq} \simeq \mathcal{H}_{L_0}$, where the duality structure on $L_0$ is $s_{L_0} = -1$ in the hyperbolic case and $s_{L_0}=1$ in the non-hyperbolic case. The structure of $\mathcal{M}_Q^{\langle \beta \rangle, \simeq}$ is therefore determined by Proposition \ref{prop:zeroLoopCoHM}. The restriction $\rho^{\sigma} : \mathcal{M}_Q^{\langle \beta \rangle} \twoheadrightarrow \mathcal{M}_Q^{\langle \beta \rangle, \simeq}$ is a surjective module homomorphism over $\rho : \mathcal{H}_Q^{\langle \beta \rangle} \rightarrow \mathcal{H}_Q^{\langle \beta \rangle, \simeq}$. Using the identification $\mathcal{M}_Q^{\langle \beta \rangle, \simeq} \simeq \mathcal{M}_{L_0}$, define a section of $\rho^{\sigma}$ by
\[
\psi^{\sigma}: 
\begin{cases}
\tilde{x}^{2 i_1+1} \cdots \tilde{x}^{2 i_d + 1} \star \mathbf{1}_1^{\sigma} \mapsto \psi (\tilde{x}^{2 i_1 +1} \cdots \tilde{x}^{2 i_d + 1}) \star \mathbf{1}_{\beta}^{\sigma} & \mbox{ in type } B, \\
\tilde{x}^{2i_1+1} \cdots \tilde{x}^{2i_d+1} \star \mathbf{1}_0^{\sigma} \mapsto \psi (\tilde{x}^{2i_1+1} \cdots \tilde{x}^{2i_d+1}) \star \mathbf{1}_0^{\sigma}
& \mbox{ in type } C, \\
\tilde{x}^{2i_1} \cdots \tilde{x}^{2i_d} \star \mathbf{1}_0^{\sigma} \mapsto \psi (\tilde{x}^{2i_1} \cdots \tilde{x}^{2i_d}) \star \mathbf{1}_0^{\sigma}
& \mbox{ in type } D.
\end{cases} 
\]
The map $\psi^{\sigma}$ is a module embedding over the restriction of $\psi$ to the appropriate even/odd subalgebra of $\mathcal{H}_Q^{\langle \beta \rangle, \simeq}$. Write $\mathcal{M}_Q^{(\beta)}$ for the image of $\psi^{\sigma}$ in types $C$ or $D$ and $\mathcal{M}_Q^{(\beta), +}$ for the image of $\psi^{\sigma}$ in type $B$.

In the non-hyperbolic case, for each subset $\varnothing \subseteq \pi \subseteq \Delta^{\sigma}$ define
\[
\mathcal{M}_Q^{(\pi),\mathsf{indec}} = \bigotimes_{\beta \in \Delta^{\sigma} \backslash \pi} \mathcal{M}_Q^{(\beta)} \otimes \bigotimes_{\beta \in \pi} \mathcal{M}_Q^{(\beta), +}.
\]
This is a rank one free $\bigotimes_{\beta \in \Delta^{\sigma} \backslash \pi} \mathcal{H}_Q^{(\beta), \mathsf{even}} \otimes \bigotimes_{\beta \in \pi} \mathcal{H}_Q^{(\beta), \mathsf{odd}}$-module. If $\varnothing \subseteq \pi \subseteq \Pi^{\sigma}$, then $\mathcal{M}_Q^{(\pi),\mathsf{simp}}$ is defined similarly. In the hyperbolic case $\mathcal{M}_Q^{(\varnothing)}$ is still defined.

The following is the main result of this section.

\begin{Thm}
\label{thm:finiteTypeCoHM}
Let $Q$ be of Dynkin type $A$. The ordered CoHA action maps
\begin{equation}
\label{eq:simpleProd}
\overset{\longleftarrow}{\boxtimes}_{ \alpha \in \Pi^+}^{\mathsf{tw}} \mathcal{H}_Q^{(\alpha)} \boxtimes^{S \mhyphen \mathsf{tw}} \bigoplus _{\varnothing \subseteq \pi \subseteq \Pi^{\sigma}} \Omega^{\sigma, \mathsf{simp}}_{Q,\pi} \cdot \mathcal{M}_Q^{(\pi),\mathsf{simp}} \longrightarrow \mathcal{M}_Q
\end{equation}
and
\begin{equation}
\label{eq:indecompProd}
\overset{\longrightarrow}{\boxtimes}_{\beta \in \Delta^-}^{\mathsf{tw}} \mathcal{H}_Q^{(\beta)} \boxtimes^{S \mhyphen \mathsf{tw}} \bigoplus _{\varnothing \subseteq \pi \subseteq \Delta^{\sigma}} \Omega^{\sigma,\mathsf{indec}}_{Q,\pi} \cdot \mathcal{M}_Q^{(\pi),\mathsf{indec}} \longrightarrow \mathcal{M}_Q
\end{equation}
are isomorphisms in $D^{lb}(\mathsf{Vect}_{\mathbb{Z}})_{\Lambda_Q^{\sigma,+}}$
\end{Thm}

\begin{proof}
Let $f_{\alpha} \in \mathcal{H}_Q^{(\alpha)}$ and $g \in \bigoplus _{\varnothing \subseteq \pi \subseteq \Pi^{\sigma}} \Omega^{\sigma, \mathsf{simp}}_{Q,\pi} \cdot \mathcal{M}_Q^{(\pi),\mathsf{simp}}$. Concretely, $g$ is constant if $\Pi^{\sigma} = \varnothing$ and is a symmetric polynomial in the variables associated to the node $Q_0^{\sigma}$ otherwise. Taking into account the ordering of the roots, Theorem \ref{thm:cohmLoc} gives
\[
(\overset{\longleftarrow}{\prod_{\alpha \in \Pi^+}} f_{\alpha}) \star g = (\prod_{\alpha \in \Pi^+} f_{\alpha} ) g,
\]
the products on the right-hand side being ordinary polynomial multiplication. It follows that the map \eqref{eq:simpleProd} is an isomorphism.

To show that the map \eqref{eq:indecompProd} is an isomorphism we use an argument similar to \cite[Theorem 11.2]{rimanyi2013}. Fix non-negative integers $\{m_u\}_{\beta_u \in \Delta}$ such that $m_{S(u)}= m_u$ for all $\beta_u \in \Delta^+$. Let $M$ be the self-dual representation determined by equation \eqref{eq:sdKrullSchmidt} and let $e = \Dim M$. The isometry group of $M$ is homotopy equivalent to $\prod_{\beta_u \in \Delta^-} \mathsf{GL}_{m_u} \times \prod_{\beta_u \in \Delta^{\sigma}} \mathsf{G}_{m_u}^{s_u}$, where $\mathsf{G}^{s_u}_{m_u}$ is a symplectic group in the hyperbolic case and is an orthogonal group otherwise.

Define sets $\mathcal{T}_{i,k,v} \subset \{1, \dots, e_i\}$ for $ i \in Q_0$, $k=1, \dots, \vert \Delta \vert$ and $v = 1, \dots, m_{u_k}$ by requiring $\vert \mathcal{T}_{i,k,v} \vert =1$ if $\dim_{\mathbb{C}} (I_{\beta_{u_k}})_i=1$ and $\mathcal{T}_{i,k,v} = \varnothing$ otherwise, and
\[
\mathcal{T}_{i,1,1} \sqcup \cdots \sqcup \mathcal{T}_{i,\vert \Delta \vert,m_{u_{\vert \Delta \vert}}} = \{1, \dots, e_i\}
\]
as ordered sets. Write $\{\epsilon_{i,1}, \dots, \epsilon_{i,e_i} \}$ for the standard basis of $\mathbb{C}^{e_i}$. Let $A_{k,v}$ be the indecomposable representation of type $\beta_{u_k}$ with basis $\{\epsilon_{i,j}\}_{i \in Q_0, j \in \mathcal{T}_{i,k,v}}$ and put
\[
\Phi^{\sigma} =\bigoplus_{k=1}^{\vert \Delta \vert} \bigoplus_{v=1}^{m_{u_k}} A_{k,v}.
\]
Define a self-dual structure on $\Phi^{\sigma}$ by requiring that
\begin{enumerate}[label=(\roman*)]
\item $A_{k,v} \oplus A_{S(k), v}$ be hyperbolic if $\beta_{u_k} \in \Delta^-$ and $v= 1, \dots, m_{u_k}$,

\item $A_{k,v} \oplus A_{k, \lfloor \frac{m_{u_k}}{2} \rfloor +v}$ be hyperbolic if $\beta_{u_k} \in \Delta^{\sigma}$ and $v = 1, \dots, \lfloor \frac{m_{u_k}}{2} \rfloor$, and

\item $A_{k,m_{u_k}}$ have its canonical self-dual structure if $\beta_{u_k} \in \Delta^{\sigma}$ and $m_{u_k}$ is odd.
\end{enumerate}
Then $\Phi^{\sigma}$ and $M$ are isometric self-dual representations. The restriction
\[
\rho^{\sigma}_{M} : H^{\bullet}_{\mathsf{G}_e^{\sigma}} (R_e^{\sigma}) \rightarrow H^{\bullet}_{\mathsf{G}_e^{\sigma}} (\eta_M^{\sigma}) \simeq H^{\bullet}(B\Aut_S(\Phi^{\sigma}))
\]
can be computed by identifying $H^{\bullet}_{\mathsf{G}_e^{\sigma}}(R_e^{\sigma})$ and $H^{\bullet}(B\Aut_S(\Phi^{\sigma}))$ with appropriately symmetric polynomials in variables $\{z_{i,j}\}$ and $\{ \theta_{k,v}\}$, respectively, and using Lemma \ref{lem:equivCoRestr}. We find that
\begin{enumerate}[label=(\roman*)]
\item if $i \in Q_0^+$ and $j \in \mathcal{T}_{i,k,v}$, then
\[
\rho^{\sigma}_{M} (z_{i,j}) =
\begin{cases}
 \theta_{k,v} & \mbox{ if } \beta_{u_k} \in \Delta^-, \\
 -\theta_{S(k),v} & \mbox{ if } \beta_{u_k} \in \Delta^+, \\
\theta_{k,v} & \mbox{ if } \beta_{u_k} \in \Delta^{\sigma} \mbox{ and } 1 \leq j \leq \lfloor \frac{m_{u_k}}{2} \rfloor, \\
-\theta_{k,v} & \mbox{ if } \beta_{u_k} \in \Delta^{\sigma} \mbox{ and } \lfloor \frac{m_{u_k}}{2} \rfloor + 1 \leq j \leq 2 \lfloor \frac{m_{u_k}}{2} \rfloor, \\
0 & \mbox{ if } \beta_{u_k} \in \Delta^{\sigma} \mbox{ and } j=m_{u_k} \mbox{ is odd, and}
\end{cases}
\]
\item if $i \in Q_0^{\sigma}$ and $j \in \mathcal{T}_{i,k,v}$, then
\[
\rho^{\sigma}_{M} (z_{i,j}) =
\begin{cases}
\theta_{k,v} & \mbox{ if } \beta_{u_k} \in \Delta^-, \\
- \theta_{S(k),v} & \mbox{ if } \beta_{u_k} \in \Delta^+, \\
\theta_{k,v} & \mbox{ if } \beta_{u_k} \in \Delta^{\sigma}.
\end{cases}
\]
\end{enumerate}

Let $f_k \in \mathcal{H}_{Q,m_{u_k} \beta_{u_k}}^{(\beta_{u_k})}$ and let $g_u$ be an element of $\mathcal{M}_{Q, m_u\beta_u}^{(\beta_u)}$ or $\mathcal{M}_{Q, m_u\beta_u}^{(\beta_u),+}$, depending on the parity of $m_u$. We will show that the image of
\begin{equation}
\label{eq:tensElement}
\overset{\longrightarrow}{\boxtimes}_{ \beta_{u_k} \in \Delta^-}^{\mathsf{tw}} f_k \boxtimes^{S \mhyphen \mathsf{tw}} \otimes_{\beta_u \in \Delta^{\sigma}} g_u
\end{equation}
under the map \eqref{eq:indecompProd} is non-zero by showing that its image under $\rho^{\sigma}_M$ is non-zero. Since $\pi_M^{\sigma}: \Sigma^{\sigma} \rightarrow R_e^{\sigma}$ is an equivariant resolution of $\overline{\eta}_M^{\sigma}$ (Theorem \ref{thm:sdResolution}), there is a single $\mathsf{T}_e$-fixed point above the $\mathsf{T}_e$-fixed point $\Phi^{\sigma} \in \eta^{\sigma}_M$. Hence the image of \eqref{eq:tensElement} under $\rho^{\sigma}_M$ consists of the single term
\begin{equation}
\label{eq:restricted}
 \prod_{\beta_{u_k} \in \Delta^-} f_k(\theta_{u_k,1}, \dots, \theta_{u_k,m_{u_k}}) \Big( \prod_{u \in \Delta^{\sigma}} g_u(z) \Big)  \mathcal{K}^{(r),\sigma} (z)_{\vert z \mapsto \theta}.
\end{equation}
Here $\mathcal{K}^{(r),\sigma}(z)$ is the $r$-fold iteration of the CoHM kernel. By assumption $g_u$ is of the form $f_u \star \mathbf{1}_{0 \slash \beta_u}^{\sigma}$ for some
\[
f_u(x_{i(\beta_u),1}, \dots, x_{i(\beta_u),\lfloor \frac{m_u}{2} \rfloor}) \in \mathcal{H}_{Q,\lfloor \frac{m_u}{2} \rfloor \beta_u}^{(\beta_u), \mathsf{even} \slash \mathsf{odd}}.
\]
Then $g_u(z)_{z \mapsto \theta} = f_u(\theta_{u,1}, \dots, \theta_{u,\lfloor \frac{m_u}{2} \rfloor}) \star \mathbf{1}_{0 \slash 1}^{\sigma}$, computed in $\mathcal{M}_{L_0}$ with the appropriate type. Each $g_u(z)_{z \mapsto \theta}$ is thus non-zero and the non-vanishing of \eqref{eq:restricted} is equivalent to the non-vanishing of $\mathcal{K}^{(r),\sigma}(z)_{\vert z \mapsto \theta}$. By Corollary \ref{cor:fundClassStructConst} we have
\[
\mathcal{K}^{(r),\sigma}(z)_{\vert z \mapsto \theta}= \rho^{\sigma}_M ( [\overline{\eta}_M^{\sigma}]) =\mathsf{Eu}_{\mbox{\tiny Aut}_S(M)}(N_{R_e^{\sigma} \slash \eta_M^{\sigma} }).
\]
That $\mathsf{Eu}_{\mbox{\tiny Aut}_S(M)}(N_{R_e^{\sigma} \slash \eta_M^{\sigma} })$ is non-zero can be shown by a straightforward modification of the argument in the ordinary case \cite[Corollary 3.15]{feher2002}.

We have proved that the restriction of the map \eqref{eq:indecompProd} to the subspace spanned by elements of the form \eqref{eq:tensElement} is injective. To see that \eqref{eq:indecompProd} itself is injective, let $m_{\beta}^{(t)} \in \mathbb{Z}_{\geq 0}$, $t=1, \dots, N$, satisfy $\sum_{\beta \in \Delta} m^{(t)}_{\beta} \beta = e$ and consider a relation of the form
\[
\sum_{t=1}^N a_t \left( \overset{\longrightarrow}{\prod_{\beta \in \Delta^-}} f_{\beta}^{(t)} \star g^{(t)} \right)=0
\]
for some $a_t \in \mathbb{Q}$ and non-zero $f_{\beta}^{(t)}$ and $g^{(t)}$ as above. Let $M^{(t)}$ be the self-dual representation associated to $m_{\bullet}^{(t)}$ via equation \eqref{eq:sdKrullSchmidt}. It suffices to consider the case in which the $M^{(t)}$ are pairwise non-isometric. Relabelling if necessary, assume that $\eta^{\sigma}_{M^{(1)}} \not\subset \overline{\eta}^{\sigma}_{M^{(t)}}$, or equivalently $\eta^{\sigma}_{M^{(1)}} \cap \overline{\eta}^{\sigma}_{M^{(t)}} = \varnothing$, for $t \geq 2$. Each summand in the localization sum presentation of $[\overline{\eta}^{\sigma}_{M^{(t)}}]$ (Corollary \ref{cor:fundClassStructConst}) is then annihilated by the restriction $\rho^{\sigma}_{M^{(1)}}$, giving
\[
0= \rho^{\sigma}_{M^{(1)}} \sum_{t=1}^N a_t \left( \overset{\longrightarrow}{\prod_{\beta \in \Delta^-}} f_{\beta}^{(t)} \star g^{(t)} \right) = a_1 \rho^{\sigma}_{M^{(1)}}\left( \overset{\longrightarrow}{\prod_{\beta \in \Delta^-}} f_{\beta}^{(1)} \star  g^{(1)} \right).
\]
The previous paragraph implies that $a_1=0$. Repeating this argument we find that $a_1=\cdots = a_N=0$, proving injectivity of \eqref{eq:indecompProd}.

To complete the proof note that, together with the isomorphism \eqref{eq:simpleProd}, Theorem \ref{thm:oriDilog} implies that the Hilbert-Poincar\'{e} series of the domain and codomain of \eqref{eq:indecompProd} are equal. Since the map \eqref{eq:indecompProd} is injective, it is also surjective.
\end{proof}

The isomorphism \eqref{eq:simpleProd} is the PBW factorization \eqref{eq:moduleFactorization} associated to the stability $\theta_{\mathsf{simp}}$ described in the proof of Theorem \ref{thm:oriDilog}. We expect an analogous statement to hold for the isomorphism \eqref{eq:indecompProd}, with $\theta_{\mathsf{simp}}$ replaced by a $\sigma$-compatible stability $\theta_{\mathsf{indec}}$ whose stable objects are the indecomposable representations and whose order by increasing slope is opposite to $<$. Without the assumption of $\sigma$-compatibility such a stability is known to exist and in many examples, such as equioriented quivers, we can check directly that it may be chosen $\sigma$-compatibly.

\setcounter{secnumdepth}{0}

\footnotesize 

\bibliographystyle{plain}
\bibliography{mybib}

\end{document}